\documentclass[1 [leqno,11pt]{amsart}
\usepackage{amssymb, amsmath}
\usepackage{amsthm, amsfonts,mathrsfs}
\def\longformule#1#2{
\displaylines{ \qquad{#1} \hfill\cr \hfill {#2} \qquad\cr } }
\def\inte#1{
\displaystyle\mathop{#1\kern0pt}^\circ }

\let\pa=\partial
\let\p=\partial
\let\al=\alpha

\let\d=\delta
\let\e=\varepsilon
\let\ep=\varepsilon

\let\r=\rho
\let\s=\sigma
\let\f=\frac
\let\Om=\Omega
\let\th=\theta
\let\vf=\varphi

\let\D=\Delta

\let\wt=\widetilde

\def\cF{{\mathcal F}}

\def\cM{{\mathcal M}}

\def\cY{{\mathcal Y}}

\def\la{\lambda}

\def\La{\Lambda}

\def\dH{\dot{H}}

\def\virgp{\raise 2pt\hbox{,}}
\def\cdotpv{\raise 2pt\hbox{;}}

\def\eqdefa{\buildrel\hbox{\footnotesize def}\over =}

\def\C{\mathop{\mathbb C\kern 0pt}\nolimits}
\def\DD{\mathop{\mathbb D\kern 0pt}\nolimits}
\def\EE{\mathop{  {\mathbb E \kern 0pt}}\nolimits}
\def\K{\mathop{\mathbb K\kern 0pt}\nolimits}
\def\N{\mathop{\mathbb N\kern 0pt}\nolimits}
\def\Q{\mathop{\mathbb Q\kern 0pt}\nolimits}
\def\R{\mathop{\mathbb R\kern 0pt}\nolimits}
\def\SS{\mathop{\mathbb S\kern 0pt}\nolimits}
\def\ZZ{\mathop{\mathbb Z\kern 0pt}\nolimits}
\def\TT{\mathop{\mathbb T\kern 0pt}\nolimits}

\newcommand{\Z}{{\ZZ}}

\newcommand{\Rmnum}[1]{\uppercase\expandafter{\romannumeral #1} }
 \numberwithin{equation}{section}

\def\dive{\mathop{\rm div}\nolimits}

\def\no{\noindent}
\def\na{\nabla}

\def\p{\partial}

\def\u3{u^3}

\def\b3{b^3}

\def\PP{{\Bbb P}}
\def\v{{\rm v}}

\def\om3{\omega^3}

\newcommand{\w}[1]{\langle {#1} \rangle}
\newcommand{\beq}{\begin{equation}}
\newcommand{\eeq}{\end{equation}}
\newcommand{\ben}{\begin{eqnarray}}
\newcommand{\een}{\end{eqnarray}}
\newcommand{\beno}{\begin{eqnarray*}}
\newcommand{\eeno}{\end{eqnarray*}}
\newcommand{\andf}{\quad\hbox{and}\quad}
\newcommand{\with}{\quad\hbox{with}\quad}
\newtheorem{defi}{Definition}[section]
\newtheorem{thm}{Theorem}[section]
\newtheorem{lem}{Lemma}[section]
\newtheorem{rmk}{Remark}[section]
\newtheorem{col}{Corollary}[section]
\newtheorem{prop}{Proposition}[section]

\begin{document}
\title[Regularity of 2-D Density patch]
{ Striated Regularity of 2-D   inhomogeneous incompressible Navier-Stokes system with variable viscosity }
 \author[Marius Paicu]{Marius Paicu}
\address [M. PAICU]
{Universit\'e  de Bordeaux \\
 Institut de Math\'ematiques de Bordeaux\\
F-33405 Talence Cedex, France}
\email{marius.paicu@math.u-bordeaux1.fr}
\author[P. Zhang]{Ping Zhang} \address[P. Zhang]{Academy of Mathematics $\&$ Systems Science
and  Hua Loo-Keng Key Laboratory of Mathematics, The Chinese Academy of
Sciences, Beijing 100190, CHINA, and School of Mathematical Sciences, University of Chinese Academy of Sciences, Beijing 100049, China.} \email{zp@amss.ac.cn}

\date{\today}

\maketitle
\begin{abstract}
In this paper, we investigate the global existence and uniqueness of strong solutions to
 2D incompressible inhomogeneous Navier-Stokes equations with viscous coefficient  depending on the density and with initial
 density being discontinuous across some smooth interface. Compared with the previous results   for the inhomogeneous
 Navier-Stokes equations with constant viscosity, the main difficulty here lies in the fact that the $L^1$ in time Lipschitz estimate of
 the  velocity field can not be obtained by energy method (see \cite{DM17,LZ1, LZ2} for instance). Motivated by
 the key idea of Chemin to solve 2-D vortex patch of ideal fluid (\cite{Chemin91, Chemin93}), namely, striated regularity can help to get
  the $L^\infty$ boundedness of
 the double Riesz transform, we derive the {\it a priori} $L^1$ in time Lipschitz  estimate of the velocity field
 under the assumption that the viscous coefficient  is close enough to a positive constant in the bounded function space. As an application,
 we shall prove the propagation of $H^3$ regularity of the interface between fluids with different densities.
\end{abstract}


\vskip 0.2cm


\setcounter{equation}{0}
\section{Introduction}

The purpose of this paper is first to investigate the global existence and uniqueness of strong solutions to the two-dimensional
incompressible  inhomogeneous  Navier-Stokes equations with viscous
coefficient depending on the density and with  initial density being discontinuous across some smooth interface.
Then we are going to study the
 propagation of regularity for the interface between fluids with different densities. In general, inhomogeneous incompressible Navier-Stokes system reads
 \beq\label{INS} \left\{
\begin{array}
{l} \displaystyle \pa_t\rho + \dive(\rho u) = 0, \qquad (t,x)\in\Bbb{R}^+\times\Bbb{R}^2,\\
\displaystyle \pa_t (\rho u) + \dive(\rho u \otimes u) - \dive(2\mu(\rho)\cM(u)) + \nabla \Pi = 0, \\
\displaystyle \dive u = 0,
\end{array}
\right. \eeq where $\rho,$ $u=(u^1, u^2)$ stand for the density and
velocity of the fluid respectively, $\cM(u)=\f12\left(\na u +\na^T u\right),$ the stress tensor, $\Pi$
is a scalar pressure function, which guarantees the divergence free condition of the velocity field, and  the viscous
coefficient $\mu (\rho)$ is a smooth non-decreasing positive function on
$[0,\infty).$ Such a system describes a fluid which is obtained by
mixing several immiscible fluids that are incompressible and that have
different densities. It may also describe a fluid containing a
melted substance.\smallskip

When one assumes that the viscous coefficient is a positive constant, there are tremendous literatures on this topic.
One may check \cite{AGZ1, DM12, DM17, HPZ, LS, LZ1, LZ2, PZZ1} and the references therein. In general,
 Lions \cite{pl}  proved
the global existence of weak solutions to \eqref{INS} with finite
energy in any space dimension. Yet the uniqueness and regularity of such weak solutions
are big open questions even in two space dimensions  (see
pages 31-32 of \cite{pl}).

Let $\mathcal{R}$ be the usual Riesz transform,
 $\mathbb{Q}\eqdefa \nabla(-\Delta)^{-1}\dive,$  and
$\mathbb{P}\eqdefa I+\mathbb{Q}$ be the Leray projection operator
to the solenoidal vector field space. Then
 under
the additional assumptions that \beq\label{desjassume}
\|\mu(\rho_0)-1\|_{L^\infty({\Bbb T}^2)}\leq \ep \andf u_0\in H^1({\Bbb T}^2),
\eeq Desjardins \cite{desj} introduced the so-called ``pesudo-energy method" and proved the
following  interesting result:

\begin{thm}\label{Desj}
{\sl Let $\r_0\in L^\infty({\Bbb T}^2),$ $u_0\in H^1({\Bbb T}^2)$ with
$\dive\, u_0 = 0.$  Then there exists a positive constant
$\varepsilon$ such that under the assumption of \eqref{desjassume}, Lions
weak solutions (\cite{pl}) to \eqref{INS} satisfy the following
regularity properties  for all $T>0:$
\begin{itemize}
  \item $u\in L^\infty([0,T];\,H^1({\Bbb T}^2)) $ and $\sqrt{\rho}u_t\in L^2(]0,T[\times {\Bbb T}^2);$

  \item $\rho$ and $\mu(\rho)\in L^\infty([0,T]\times {\Bbb T}^2)\cap
  C([0,T];\,L^p({\Bbb T}^2))$ for all $p\in[1,\infty[;$

  \item $\nabla(\Pi-\mathcal{R}_i\mathcal{R}_j(2\mu \cM_{ij}))$ and
  $\nabla(\mathbb{P}\bigotimes\mathbb{Q}(2\mu \cM))_{ij}\in L^2(]0,T[\times {\Bbb T}^2);$

  \item $\Pi$ may be renormalized in such a way that for some
universal constant $C>0,$ \beq\label{upi} \Pi\quad\mbox{and}\quad
\nabla u\in L^2(]0,T[;\, L^p({\Bbb T}^2))\quad\mbox{for all}\,\;
p\in [4,p^*], \eeq where \beq\label{pstar}
\frac{1}{p^*}=2C\|\mu(\rho_0)-1\|_{L^\infty}. \eeq
\end{itemize}
Moreover, if $\mu(\r_0)\geq \underline{\mu}$ and $\log(\mu(\r_0))\in
W^{1,r}({\Bbb T}^2)$ for some $r>2,$ there exists some positive time
$\tau$ so that $u\in L^2(]0,\tau[; H^2({\Bbb T}^2))$ and $\mu(\r)\in
C([0,\tau];W^{1,\bar{r}}({\Bbb T}^2))$ for any $\bar{r}<r.$ }
\end{thm}

 The solution provided by Theorem \ref{Desj} still has neither uniqueness nor regularity. However, if the initial density
 belongs to some Besov spaces with positive index
 which guarantee that the initial density is a continuous function,
 Abidi and the second author (\cite{AZ1}) and
Huang and the first author (\cite{HP14}) can prove not only the uniqueness but also the regularity of the solution provided by  Theorem \ref{Desj} in
the whole plane.\smallskip

On the other hand,  Lions proposed the following open question in
 \cite{pl}: suppose the initial density $\r_0={\bf 1}_{D}$ for some
 smooth domain $D,$ Theorem 2.1 of \cite{pl} provides at least
 one global weak solution $(\r,u)$ of \eqref{INS} such that for all
 $t\geq 0,$ $\r(t)={\bf 1}_{D(t)}$ for some set $D(t)$ with ${\rm vol}(D(t))={\rm vol}(D).$
 Then whether or not the regularity of $D$ is preserved by  time
 evolution?

 When one assumes that the  viscous coefficient is a positive constant,
  Liao and the second author \cite{LZ1,LZ2} solved the case when the system \eqref{INS} is supplemented with  the initial
density, $\r_0(x)=\eta_1{\bf 1}_{\Om_0}+\eta_2{\bf 1}_{\Om_0^c},$ for some pair of positive constants $(\eta_1,\eta_2),$ and for any
 bounded, simply connected  domain $\Om_0$ with $W^{k+2,p}(\R^2)$ ($p\in ]2,4[$) boundary regularity.   Danchin and Zhang \cite{DZX} and  Gancedo and Garcia-Juarez
  \cite{GG}  proved the propagation of $C^{k+\gamma}$ regularity
of the interface for $k=1$ or $k=2.$  Lately Danchin and Mucha \cite{DM17} proved the propagation of $C^{1+\gamma}$ regularity of density patch
which allows vacuum.

\smallskip

 The goal of this paper is first to study the global well-posedness of the 2-D inhomogeneous Navier-Stokes equations  \eqref{INS}
 with initial density having striated regularity, which in particular allows initial density to be discontinuous across some smooth
 interface.  We point out that the main difficulty to derive the $L^1$ in time Lipschitz estimate of the  velocity field
 lies in the fact that Riesz transform does not map continuously from $L^\infty$ space to $L^\infty$ space. Motivated by
 the key idea of Chemin to solve the two-dimensional vortex patch of ideal flow in \cite{Chemin91, Chemin93}, namely,  striated regularity
 can help to get
  the $L^\infty$ boundedness of
 the double Riesz transform, we derive the {\it a priori} $L^1$ in time Lipschitz  estimate of the velocity field
 under the assumption that the viscous coefficient is close enough to a positive constant in the bounded function space.

  In order to do so, let us first recall the following definition from \cite{Ch1995}:

\begin{defi}[Definition 3.3.1 of \cite{Ch1995}]\label{S1def1}
{\sl Let $X\eqdefa\bigl(X_\la\bigr)_{\la\in\Lambda}$ be a family of solenoidal vector fields on $\R^2.$
We call $X$ to be a non-degenerate family of vector fields if there holds
\beq\label{S1eq1}
I(X)\eqdefa \inf_{x\in \R^2}\sup_{\la\in\Lambda}|X_\la(x)|>0.
\eeq}
\end{defi}

Our first result  states as follows:

\begin{thm}\label{thm1}
{ Let $p\in ]2, \infty[,$  $X\eqdefa\bigl(X_\la\bigr)_{\la\in\Lambda}$  be a non-degenerate family of vector fields in $W^{1,p}(\R^2).$ Let $u_0\in \dot{H}^{-2\d}(\R^2)\cap H^1(\R^2)$
for some $\d\in ]1/p,1/2[,$
 and $\r_0\in
L^\infty(\R^2)$ with  $\r_0-1\in L^2(\R^2)$ and $\sup_{\la\in\La}\bigl\|\p_{X_\la}\mu(\r_0)\bigr\|_{L^\infty}\leq C_{\mu,X}.$ We also assume that there exist positive constants $m, M$ such that
\beq \label{S1eq7}
m\leq \r_0(x)\leq M\quad \mbox{for any} \  x\in \R^2. \eeq  Then there exists a small enough
positive constant $\e_0$ such that if
\beq \label{S1eq2}
\|\mu(\rho_0)-1\|_{L^\infty}\leq \ep_0,
\eeq \eqref{INS} has a unique global solution $u\in L^\infty(\R^+; H^1(\R^2))\cap L^1(\R^+; \mbox{Lip}(\R^2)),$ which satisfies
\beq \label{S1eq3}
\begin{split}
&\|\w{t}^{\d}u\|_{L^\infty_t(L^2)}+\|\w{t}^{\d_-}\na u\|_{L^2_t(L^2)}+\bigl\|\w{t}^{(\f{1}2+\d)_-}\na
u\bigr\|_{L^\infty_t(L^2)}\\
&+\bigl\|\w{t}^{(\f{1}2+\d)_-}u_t\bigr\|_{L^2_t(L^2)}+ \bigl\|\w{t}^{(\f{1}2+\d)_-}D_tu\bigr\|_{L^2_t(L^2)}\\
&+
\bigl\|\sqrt{t}\w{t}^{\f12\left(1+\d\right)_-}D_tu\bigr\|_{L^\infty_t(L^2)}+\bigl\|\sqrt{t}\w{t}^{\left(\f12+\d\right)_-}\na D_tu\bigr\|_{L^2_t(L^2)}\leq C_0.
\end{split}
\eeq
Moreover, there exists a non-degenerate family of vector fields $X(t)\eqdefa\left(X_\la(t)\right)_{\la\in\La}$ for any $t>0$ so that
$X_\la(0)$ coincides with $X_\la,$ $X_\la(t)\in L^\infty(\R^+; W^{1,p}(\R^2))$ for each $\la\in\La$ and
$$\sup_{\la\in\La}\bigl\|\p_{X_\la(t)}\mu(\r)\bigr\|_{L^\infty_t(L^\infty)}\leq C_{\mu,X}.$$
Here and in all that follows, we always denote $\w{t}\eqdefa \left(1+t^2\right)^{\f12},$ $\p_Xf\eqdefa X\cdot\na f,$   $D_t\eqdefa \p_t+u\cdot\na$  to be the material derivative,
 $\d_-$ to be any positive constant smaller than $\d,$ and
$C_0$ to be a positive constant depending on  $m,M, \|\r_0-1\|_{L^2},$  $\|u_0\|_{\dot H^{-2\d}},$ and $\|u_0\|_{H^1}.$
}
\end{thm}

As an application of the above result, we can prove the propagation of regularity for the interface between
fluids with different densities. More precisely, let  $\Omega_0$ be a simply connected $H^3(\R^2),$
   bounded domain. Let $f_0\in
H^3(\R^2)$ such that $ \p\Omega_0=f_0^{-1}(\{0\})$ and $\nabla
f_0 $ does not vanish on $\p\Omega_0. $ Then we can parametrize
$\p\Omega_0$ as
\beq\label{S1eq5} \gamma_0: {\Bbb S}^1\mapsto \p\Omega_0
\hbox{ via } s\mapsto \gamma_0(s) \hbox{ with
}\p_s\gamma_0(s)=\nabla^\bot f_0(\gamma_0(s)). \eeq
Let us denote
\beno
\begin{split}
\Om_\e\eqdefa \{ x\in\R^2\ |\ &dist(x,\p\Om_0)<\e\ \},\quad \Om_\e^+\eqdefa \{ x\in\R^2\setminus\Om_0 |\ dist(x,\p\Om_0)>\e\ \},\\
&\quad \andf \Om_\e^-\eqdefa \{ x\in\Om_0 |\ dist(x,\p\Om_0)>\e\ \}.
\end{split}
\eeno
We now take $\e$ sufficiently small so that $|\na f_0(x)|\geq c_0$ for $x\in \Om_{2\e}.$
Let  $\chi_\e^+$ and $\chi_\e^-$ be smooth cut-off functions so that $$\chi_\e^+(x)=\left\{
\begin{array}
{l}  0, \ x\in \Om_\e\cup\Om_0,\\
\displaystyle  1, \ x\in \Om_{2\e}^+,
\end{array}
\right. \andf \chi_\e^-(x)=\left\{
\begin{array}
{l}  0, \ x\in \R^2\setminus\Om_\e^-,\\
\displaystyle  1, \ x\in \Om_{2\e}^-.
\end{array}
\right.$$
We denote
\beq \label{S1eq8}
X_0\eqdefa \na^\perp f_0,\quad X_1\eqdefa\na^\perp\left(x_1\chi_\e^+\right), \andf X_2\eqdefa\na^\perp\left(x_1\chi_\e^-\right).
\eeq
For any $\eta_1,\eta_2>0,$  we take the
initial density $\r_0$ and the initial velocity $u_0$ of
\eqref{INS} as  \beq\label{S1eq6}
\begin{split}
 &  {\rho_0}(x)={\eta_1}{\bf 1}_{\Om_0}(x)+\eta_2{\bf 1}_{\Om_0^c}(x),
   \  u_0\in\dot{H}^{-2\d}(\R^2)\cap
H^1(\R^2) \with \p_{X_\la}u_0\in \dot H^1(\R^2)
\end{split}\eeq
for some $\d\in ]1/3,1/2[$ and $\la=0,1,2.$

\smallbreak

\begin{thm}\label{thm2}
{\sl Let the initial data $(\r_0, u_0)$ be given by \eqref{S1eq5} and \eqref{S1eq6}
for some pair of positive constants $(\eta_1,\eta_2).$ Then under the smallness \eqref{S1eq2}, \eqref{INS} has a unique global
solution $(\r, u)$ such that $\r(t,x)={\eta_1}{\bf 1}_{\Om(t)}(x)+\eta_2{\bf
1}_{\Om(t)^c}(x)$,  with $\Om(t)$  being a bounded,  simply connected  $H^3(\R^2)$ domain for any $t>0.$
}
\end{thm}

\begin{rmk} \begin{itemize}
\item Theorem \ref{thm2} presents an example that Theorem \ref{thm1} works for initial density being discontinuous across a smooth
interface. Thus Theorem \ref{thm1} improves the previous results in \cite{AZ1, HP14};

\item In fact, if $f_0\in W^{2,p}(\R^2)$ for some $p\in ]2,\infty[$ and without the additional assumption that $\p_{X_\la}u_0\in H^1(\R^2),$ we can prove that the
time-evolved interface
still belongs to the class of $W^{2,p}(\R^2)$ for any $t>0;$

\item We may also propagate higher order regularities of the interface with more striated regularities on the initial data.
   For simplicity, we shall not pursue on this direction here.

\item The idea to prove the propagation of boundary regularity of the interface comes from \cite{Chemin91, Ch1995}.
   One may check \cite{LZ2} for more references on this topic;
\end{itemize}
\end{rmk}

Let us end this section with the notations we shall use in the following context:

\no{\bf Notations:} Let $A, B$ be two operators, we
denote $[A,B]\eqdefa AB-BA,$ the commutator between $A$ and $B$. For
$a\lesssim b$, we mean that there is a uniform constant $C,$ which
may be different on different lines, such that $a\leq Cb$. We
denote by $(a|b)$ (or $\int_{\R^2} a | b \,dx$) the $L^2(\R^2)$ inner product
of $a$ and $b.$ We shall also
 use Einstein summation notation for repeated indices. Let $a>0,$ we always denote $a_-$ to be any positive constant smaller than $a.$

\setcounter{equation}{0}
\section {Outline of the proof}
\label{structureproof}

In this section, we shall sketch the proof of Theorems \ref{thm1} and \ref{thm2}.

Section \ref{Sect3} is devoted to the proof of the following basic energy estimates for \eqref{INS}:

\begin{prop}\label{S2prop1}
{\sl Let $u$ be a smooth enough solution of \eqref{INS} on $[0,T^\ast[.$ Then under the smallness
condition \eqref{S1eq2}, for any $t<T^\ast,$ we have
\beq\label{ener2}\begin{split} \|&u\|_{L^\infty_t(L^2)}^2
+ \|\na u\|_{L^2_t(L^2)}^2 \leq C
\|u_0\|_{L^2}^2, \end{split}\eeq
and
\beq\label{ener1}\begin{split}
\|\na u\|_{L^\infty_t(L^2)}^2 &+\int_0^t\int_{\R^2}\Bigl(|D_tu|^2+|u_t|^2+|\PP\dive\bigl(\mu(\rho)\cM(u)\bigr)|^2\\
&+ |\Q\dive\bigl(\mu(\rho)\cM(u)\bigr)-\na\Pi|^2\Bigr)\,dx\,dt'\leq C\|\na u_0\|_{L^2}^2\exp\left(C\|u_0\|_{L^2}^2\right). \end{split}\eeq
 }
\end{prop}

With Proposition \ref{S2prop1}, we deduce from Proposition 2.2 of \cite{AZ1} that

\begin{prop}\label{S2prop2}
{\sl If we assume in addition that $\r_0-1\in L^2(\R^2),$
$u_0\in \dot H^{-2\d}(\R^2)$ for $\d\in ]0, 1/2[,$ then for any $t\in [0,T^\ast[,$  we have  \beq\label{p.2a}
\begin{split} \|\w{t}^{\d}u\|_{L^\infty_t(L^2)}&+\|\w{t}^{\d_-}\na u\|_{L^2_t(L^2)}+\|\w{t}^{(\f{1}2+\d)_-}\na
u\|_{L^\infty_t(L^2)}\\
&+\|\w{t}^{(\f{1}2+\d)_-}u_t\|_{L^2_t(L^2)}+ \|\w{t}^{(\f{1}2+\d)_-}D_tu\|_{L^2_t(L^2)}\leq
C_0.
\end{split}
\eeq
 } \end{prop}

 In Section \ref{Sect4}, we shall prove the following time-weighted energy estimate for $D_tu.$

 \begin{prop}\label{S3prop1}
{\sl Under the assumptions of Proposition \ref{S2prop2} and $\mu(\r_0)\geq\f34,$ for any $t<T^\ast,$ we have
\beq \label{S3eq6}
\|\sqrt{t}\w{t}^{\f12\left(1+\d\right)_-}D_tu\|_{L^\infty_t(L^2)}+\|\sqrt{t}\w{t}^{\left(\f12+\d\right)_-}\na D_tu\|_{L^2_t(L^2)}\leq C_0.
\eeq
}
\end{prop}

We point out that the estimate \eqref{S3eq6} is completely new compared with the previous references \cite{AZ1, HP14}.
Part of the estimates in Propositions \ref{S2prop1} and \ref{S2prop2} were obtained in \cite{AZ1, HP14}. But the proof here will be more concise.
We also emphasize that the reason why there is no Lipschitz estimate for the velocity field in Theorem \ref{Desj} is that
Riesz transform can not map continuously from $L^\infty(\R^2)$ to $L^\infty(\R^2)$.
To overcome this difficulty, we need to appeal to the idea from \cite{Chemin91, Chemin93}. Toward this, for any
 smooth vector field $X_0$ with $\dive X_0=0,$ we define the vector field $X(t)=(X^1(t),X^2(t))$ for $t>0$ by
\beq\label{Xtrans} \left\{
\begin{array}
{l} \displaystyle \pa_t X+u\cdot\nabla X=X\cdot\nabla u=\p_Xu,\\
\displaystyle X|_{t=0}=X_0.
\end{array}
\right. \eeq
It is easy to check that
\beq \label{S3eq6a}
\begin{split}
&\p_t\dive X+u\cdot\na \dive X=0 \ \Rightarrow \ \dive X(t)=0,\\
&[\p_X; D_t]=0\with \p_X\eqdefa X\cdot\na \andf D_t\eqdefa \p_t+u\cdot \na.  \end{split} \eeq
Let us denote $\r_X\eqdefa\p_X\r, u_X\eqdefa\p_Xu$ and $\Pi_X\eqdefa \p_X\Pi.$ Then  by applying the operator $\p_X$ to \eqref{INS} and using
\eqref{S3eq6a}, we write
 \beq\label{XLINS} \left\{
\begin{array}
{l} \displaystyle \pa_t\rho_X + u\cdot\na \rho_X  = 0, \\
\displaystyle \rho\pa_t u_X + \rho u\cdot\na u_X - \dive(2\mu(\rho)\cM(u_X)) + \nabla \Pi_X = G, \\
\displaystyle \dive u_X = \p_iX^k\p_ku^i=\dive(u\cdot\na X), \\
\displaystyle (\rho_X, u_X)|_{t=0}=(\p_{X_0}\rho_0,\p_{X_0}u_0),
\end{array}
\right. \eeq
where $G=(G_1,G_2)$ and
\beq \label{S3eq7}\begin{split}
G_j=&-\r_X D_tu^j+\p_i\left(2\p_X\mu(\r)\cM_{ij}(u)\right)+\p_jX\cdot\na\Pi\\
&-\p_i\left(\mu(\r)\bigl(\p_iX\cdot\na u^j+\p_jX\cdot\na u^i\bigr)\right)
-\p_iX\cdot\na\left( 2\mu(\r)\cM_{ij}(u)\right).
\end{split}
\eeq

Observing from \eqref{Xtrans} that the estimate of $\|\na X(t)\|_{L^p}$ depends on $\|\na u_X\|_{L^p}.$
In Section \ref{Sect5}, we will find that one difficulty to derive the estimate of $\|\na u_X\|_{L^1_t(L^p)}$  lies in the fact that there is no control
 of $\na \Pi$ with $\mu(\r)$ belonging to the bounded function space. Indeed
 by applying the space divergence operator to the momentum equation
of \eqref{INS}, we obtain
\beq\label{b.4a}
\begin{aligned}
\Pi& =(-\Delta)^{-1}\dive(\rho D_tu) -(-\Delta)^{-1}\dive\dive\bigl(2\mu(\rho)\cM(u)\bigr)\eqdefa \Pi_1+\Pi_2,
\end{aligned}
\eeq
The best estimate we can expect for $\Pi_2$ is that $\Pi_2\in L^\infty(\R^2).$ In order to do so, we need the following proposition:

\begin{prop}\label{S9prop1}
{\sl Let $p\in ]2, \infty[$ and $X=\bigl(X_\la\bigr)_{\la\in\Lambda}$ be a non-degenerate family of vector fields in the sense of  Definition \ref{S1def1}.
 Then  for any $s\in ]2/p,1[,$ one has
\beq \label{S9eq4}
\begin{split}
\|&\na^2(-\D)^{-1}g\|_{L^\infty}\leq C_s\Bigl\{\|g\|_{L^p}+\|g\|_{L^\infty}+\f1{I(X)}\sup_{\la\in\Lambda}\Bigl( \bigl(\|X_\la\|_{L^\infty}\|g\|_{L^p}\bigr)^{1-\f2p}
\\
&\ \times\bigl(\|\na X_\la\|_{L^p}\|g\|_{L^\infty}+\|\p_{X_\la}g\|_{L^p}\bigr)^{\f2p}+\bigl(\|X_\la\|_{L^\infty}\|g\|_{L^p}\bigr)^{1-\f2{ps}}\bigl(\|\na X_\la\|_{L^p}\|g\|_{L^\infty}\bigr)^{\f2{ps}}\Bigr)\Bigr\},
\end{split}
\eeq with $I(X)$ being given by \eqref{S1eq1}. }
\end{prop}

The proof of this proposition is motivated by Theorem 3.3.1 of \cite{Ch1995}, where the vector field belongs to some H\"older space. We shall present
the detailed proof of Proposition \ref{S9prop1} in Section \ref{Sect5}. As an application of Proposition \ref{S9prop1}, we shall prove the following estimate
of $\bigl\|\na u_{X_\la}\bigr\|_{L^p}.$

\begin{col}\label{S9col1}
{\sl Let $(\r,u)$ be a smooth enough solution of \eqref{INS} on $[0,T^\ast[$ and $\bigl(X_\la(t)\bigr)_{\la\in\Lambda}$ be a non-degenerate family of vector fields
 where $X_\la$ satisfies \eqref{Xtrans} with initial data $X_\la$ for each $\la\in\La.$ For a given positive constant $\frak{M}$ and for some $\e_1$ sufficiently small, we define
\beq \label{S9eq11as}
T^\star\eqdefa \sup\bigl\{ T<T^\ast:\ \sup_{\la\in\La}\|\na X_\la\|_{L^\infty_T(L^p)}\leq \frak{M} \andf \|\mu(\r_0)-1\|_{L^\infty}\f{\frak{M}}{I(X)}\leq \e_1\ \bigr\}.
\eeq
 Then under the assumption of \eqref{S1eq2} and \beq\label{S9eq11af}
 \sup_{\la\in\La}\bigl\|\p_{X_\la}\mu(\r_0)\bigr\|_{L^\infty}\leq C_{\mu,X},\eeq
 for any $t\in [0,T^\star],$
 we have
\beq \label{S9eq11}
\begin{split}
\sup_{\la\in\La}\|\na u_{X_\la}(t)\|_{L^p}\leq &  C\Bigl\{\|\na u\|_{L^p}\Bigl[1+\f{\sup_{\la\in\La}\|X_\la(t)\|_{L^\infty}^{1-\f2p}}{I(X(t))}\Bigr]\\
&+\bigl[\|\na u\|_{L^p}+\|D_t u\|_{L^2}^{\f2p}\|\na D_t u\|_{L^2}^{1-\f2p}\bigr]
\sup_{\la\in\La}\|X_\la(t)\|_{L^\infty}\\
&+\bigl(\|\na u\|_{L^p}
+\|\na u\|_{L^\infty}+\|D_tu\|_{L^2}\bigr)\sup_{\la\in\La}\|\na X_\la(t)\|_{L^p}\Bigr\},
\end{split}
\eeq for some positive constant $C$ depending on $C_{\mu,X}.$}
\end{col}

In Section \ref{Sect6}, we shall first present the estimates of $\|X_\la(t)\|_{L^\infty}$ and $\|\na X_\la(t)\|_{L^p}.$

\begin{prop}\label{S6prop1}
{\sl Let $X_0\in L^\infty\cap \dot W^{1,p}$ and $u\in L^1_t(Lip(\R^2))$ be a smooth solenoidal vector field. Let  $X(t,\cdot)$ be a smooth enough solution of
\eqref{Xtrans}. Then for any $r\in [1,\infty],$ we have
\beq\label{S6eq1}
\begin{split}
&\|X_0\|_{L^r}\exp\bigl(-\|\na u\|_{L^1_t(L^\infty)}\bigr)\leq \|X(t)\|_{L^r}\leq \|X_0\|_{L^r}\exp\bigl(\|\na u\|_{L^1_t(L^\infty)}\bigr).
\end{split}
\eeq
Moreover, if $\left(X_\la(0)\right)_{\la\in\La}$ be a non-degenerate family of vector fields and $X_\la(t)$ be the corresponding solution
  of \eqref{Xtrans} with initial data $X_\la(0).$ Then under the assumptions of Proposition \ref{S2prop2} and Corollary \ref{S9col1}, for $t\leq T^\star$ given by \eqref{S9eq11as}, one has
\beq \label{S6eq2}
\sup_{\la\in\La}\|\nabla X_\la(t)\|_{L^p}\leq C_0\Bigl(1+\sup_{\la\in\La}\|\nabla X_\la(0)\|_{L^p}+\Bigl[1+
\f{1}{I(X(0))^{\f{p}{p-2}}}\Bigr]\sup_{\la\in\La}\|X_\la(0)\|_{L^\infty}\Bigr){\rm V}(t).
\eeq Here and in what follows, we always denote ${\rm V}(t)=\exp\left(C\|\na u\|_{L^1_t(L^\infty)}\right).$ }
\end{prop}

By virtue of Propositions \ref{S9prop1} and \ref{S6prop1}, we prove  the following key ingredient used in the proof of Theorem \ref{thm1} in Section \ref{Sect6}:

\begin{prop} \label{S9lem3}
{\sl Let $u$ be a smooth enough solution of \eqref{INS} on $[0,T^\ast[.$  Let $\left(X_\la(0)\right)_{\la\in\La}$ be
a non-degenerate family of vector fields. Then under the assumptions of Theorem \ref{thm1}, there exist a positive constant  $\frak{N}$  and some small enough constant $\e_0,$
which depend only on $C_0, I(X(0)), \sup_{\la\in\La}\|X_\la(0)\|_{L^\infty},$  $\sup_{\la\in\La}\|\na X_\la(0)\|_{L^p}$ and $C_{\mu,X}$
 so that under the smallness condition \eqref{S1eq2}, there holds
\beq \label{S6eq4}
\begin{split}
\|\na u\|_{L^1_t(L^\infty)} \leq \frak{N}\quad\mbox{for any}\ \ t\in [0,T^\ast[.
\end{split}
\eeq}
\end{prop}

Now we are in a position to complete the proof of Theorem \ref{thm1}.

\begin{proof}[Proof of Theorem \ref{thm1}]
We denote
$\r_{0,\eta}\eqdefa\r_0\ast j_\eta,$ $u_{0,\eta}\eqdefa u_0\ast
j_\eta,$  and $\mu_\eta=\mu\ast j_\eta,$ where
$j_\eta(|x|)=\eta^{-2}j(|x|/\eta)$ is the standard Friedrich's
mollifier. Then along the same line to the proof of Propositions \ref{S2prop2}, \ref{S3prop1}, \ref{S6prop1}
and \ref{S9lem3}, we can prove that the inequalities (\ref{ener2}-\ref{S3eq6}) and
(\ref{S6eq1}-\ref{S6eq4}) hold for the solution,  $(\r_\eta,u_\eta, \na\Pi_\eta),$
of  \eqref{INS} with viscous coefficient $\mu_\eta$ and
with initial data $(\r_{0,\eta},u_{0,\eta})$ on $[0,T_\eta^\ast[$ for some maximal time of existence $T^\ast_\eta$ provided that
$\eta$ is sufficiently small.
Correspondingly, for any $\la\in\La,$ \eqref{Xtrans} with initial data $X_\la$ has a unique solution $X_{\la,\eta}(t)$ for any $t<T^\ast_\eta$
which satisfies the estimates \eqref{S6eq1} and \eqref{S6eq2}.

By virtue of \eqref{S6eq4} and Theorem 4.1 of \cite{AZ1}, we conclude that $T^\ast_\eta=\infty.$ Then a standard
compactness argument yields the existence part of  Theorem \ref{thm1}. Moreover, by virtue of
\eqref{S6eq1} and \eqref{S6eq2}, the limit set $\left(X_\la(t)\right)_{\la\in\La}$ of   $\left(X_{\la,\eta}(t)\right)_{\la\in\La}$
as $\eta$ goes to $0$ is a non-degenerate family of vector fields and $X_\la(t)\in L^\infty(\R^+; W^{1,p}(\R^2))$ for each $\la\in\La.$
The uniqueness part of Theorem \ref{thm1} follows along the same line to the proof of Theorem 1.2 of \cite{AZ1}. We skip the details here.
\end{proof}

To prove Theorem \ref{thm2}, as in \cite{LZ1,LZ2}, we need to propagate striated regularity of the convection velocity field.

\begin{prop}\label{S3col1}
{\sl Let $\d\in ]0,1/2[,$  $p\in \left]2,2+2/\d\right[,$ and $X_0\in L^\infty\cap\dot{W}^{1,p}(\R^2).$ Let $(\r,u,X)$ be a smooth enough solution
 of the coupled system \eqref{INS} with \eqref{Xtrans} on $[0,T^\ast[.$ Then under the assumptions of Proposition \ref{S2prop2} and $\bigl\|\p_{X_0}\r_0\bigr\|_{L^\infty}\leq C_{X,\r_0},$  for any $t<T^\ast,$  we have
\beq \label{S3eq14}
\|\w{t}^{\d_-}u_X(t)\|_{L^2}^2+\|\w{t}^{\d_-}\na u_X\|_{L^2_t(L^2)}^2\leq C_1(1+\|X\|_{L^\infty_t(L^\infty)}^2+\|\na X\|_{L^\infty(L^p)}^2\bigr){\rm V}(t),
\eeq Here and in all that follows, we always denote $C_1$ to be a positive constant depending on $C_0,$ $\|X_0\|_{ L^\infty\cap\dot{W}^{1,p}}$ and $C_{X,\r_0}.$}
\end{prop}

\begin{prop}\label{S4prop1}
{\sl Let $p\in \left]4,2+2/\d\right[.$ Then under the assumptions of Proposition \ref{S3col1} and \eqref{S1eq2}, for any $t<T^\ast,$ one has
\beq \label{S4eq13}
\begin{split}
\bigl\|&\sqrt{t}\w{t}^{\d_-}\na u_X(t)\bigr\|_{L^2}^2+\bigl\|\sqrt{t}\w{t}^{\d_-}D_t u_X(t)\bigr\|_{L^2}^2\\
&\qquad\qquad\leq C_1\bigl(1+\|X\|_{L^\infty_t(L^\infty)}^2+\|\na X\|_{L^\infty_t(L^p)}^2\bigr){\rm V}(t).
\end{split}
\eeq If we assume moreover that $\p_{X_0}u_0\in \dot{H}^1(\R^2),$ we have
\beq \label{S4eq13op}
\begin{split}
\bigl\|&\w{t}^{\left(\d+\f12\right)_-}\na u_X(t)\bigr\|_{L^2}^2+\bigl\|\w{t}^{\left(\d+\f12\right)_-}D_t u_X(t)\bigr\|_{L^2}^2\\
&\qquad\qquad\leq C_1\bigl(1+\bigl\|\na\p_{X_0}u_0\bigr\|_{L^2}+\|X\|_{L^\infty_t(L^\infty)}^2+\|\na X\|_{L^\infty_t(L^p)}^2\bigr){\rm V}(t).
\end{split}
\eeq
}
\end{prop}

Propositions \ref{S3col1} and \ref{S4prop1} will be proved in Sections \ref{Sect7} and \ref{Sect8} respectively. Then we shall prove
in Section \ref{Sect9} the following propositions:

\begin{prop}\label{S0prop1}
{\sl Let $\d\in ]1/3,1/2[.$
We assume that $\p_{X_0}\mu(\r_0)\in L^6(\R^2)\cap L^\infty(\R^2)$ and $\p_{X_0}^2\mu(\r_0)\in L^2(\R^2).$
Then under the assumptions of Proposition \ref{S4prop1} and \eqref{S1eq2}, for any $t<T^\ast,$ one has
\beq \label{S0eq1}
\begin{split}
\|\na \p_XX(t)\|_{L^2}\leq C_2\bigl(1+\|X\|_{L^\infty_t(L^\infty)}&+\|\na X\|_{L^\infty_t(L^4)}+\|\na X\|_{L^\infty_t(L^8)}\bigr)^3\\
&\qquad\ \ \times{\rm V}(t)
\exp\bigl(C_2\|\Pi_2\|_{L^1_t(L^\infty)}\bigr).
\end{split}
\eeq
where $\Pi_2$ is given by \eqref{b.4a} and  $C_2$ is a positive constant depending on $C_1,$  $\|\p_{X_0}\mu(\r_0)\|_{L^6\cap L^\infty}$ and $\|\p_{X_0}^2\mu(\r_0)\|_{L^2}.$}
\end{prop}

Now we present the proof of Theorem \ref{thm2}.

\begin{proof}[Proof of Theorem \ref{thm2}]  It is easy to observe that the family of vector fields $X_\la$, $\la=0,1,2,$
given by \eqref{S1eq8} is a non-degenerate family of vector fields. Moreover,   we deduce from
\eqref{S1eq8}-\eqref{S1eq6} that
\beq\label{S0eq2}
\begin{split}
& \dive X_\la=0,\quad X_\la\in (H^2)^2 \andf\min\left\{\eta_1,\eta_2\right\}\leq \r_0
\leq \max\left\{\eta_1,\eta_2\right\},\quad
\p_{X_\la}^\ell\r_0\equiv 0 \end{split} \eeq for $\la=0, 1, 2$ and $\ell=1, 2.$
  Then  the initial  conditions   in Theorem \ref{thm1} are   fulfilled.
By virtue of Theorem \ref{thm1}, under the smallness condition \eqref{S1eq2}, the coupled system \eqref{INS}-\eqref{Xtrans} has a unique global solution $\left(\r, u, \na\Pi, X_\la\right)$
so that the bounds \eqref{S1eq3}  and \eqref{S6eq4} are valid for any $t>0.$ Moreover, Proposition \ref{S6prop1} ensure that for any $p\in ]2,\infty[$ and $\La=\{0,1,2\},$
\beq\label{S0eq3}
\begin{split}
&\left(\sup_{\la\in\La}\|X_\la(t)\|_{L^\infty}+\sup_{\la\in\La}\|\nabla X_\la(t)\|_{L^p}\right)\leq \frak{L} \andf
I(X(t))\geq I(X(0))\exp\left(-\frak{N}\right).
\end{split}
\eeq
Here and in what follows, we always denote $\frak{L}$ to be a positive constant
depending on  $I(X(0)),$ $C_0,$ $\frak{N},$ $\|X_\la\|_{L^\infty}$ and $\|\na X_\la\|_{L^p}.$

In view of \eqref{S0eq3}, we deduce from
Corollary \ref{S9col1} that for any $p\in [3,\infty[$
\beq  \label{S0eq4} \int_0^\infty\sup_{\la\in\La}\bigl\|\na u_{X_{\la}(t)}(t)\bigr\|_{L^p}\,dt\leq \frak{L}.
\eeq

On the other hand,  since $\dive u=0,$ we write
\beno
\Pi_2=-(-\D)^{-1}\dive\dive\bigl(2(\mu(\r)-1)\cM(u)\bigr) \andf \frak{B}(\r)\eqdefa \|\mu(\r)-1\|_{L^\infty},
\eeno
so that for any $s\in ]2/p.1[,$  applying Proposition \ref{S9prop1} gives rise to
\beq\label{S0eq5}
\begin{split}
\|\Pi_2\|_{L^\infty}
\leq & C_s\Bigl\{\frak{B}(\r)\bigl(\|\na u\|_{L^p}+\|\na u\|_{L^\infty}\bigr)+\f1{I(X)}\sup_{\la\in\La}\Bigl(\bigl(\frak{B}(\r)|X_\la\|_{L^\infty}\|\na u\|_{L^p}\bigr)^{1-\f2p}\\
&\times\bigl(\|\p_{X_\la}\mu(\r)\|_{L^\infty}\|\na u\|_{L^p}+\frak{B}(\r)
\bigl(\|\na X_\la\|_{L^p}\|\na u\|_{L^\infty}
+\|\na u_{X_\la}\|_{L^p}\bigr)\bigr)^{\f2p}\\
&\qquad\qquad\qquad+\frak{B}(\r)\bigl(\|X_\la\|_{L^\infty}\|\na u\|_{L^p}\bigr)^{1-\f2{ps}}\bigl(
\|\na X_\la\|_{L^p}\|\na u\|_{L^\infty}
\bigr)^{\f2{ps}}\Bigr)\Bigr\}.
\end{split}
\eeq
Along with \eqref{S1eq3},  \eqref{S0eq3} and \eqref{S0eq4},  for $p\in [3,\infty[$ in \eqref{S0eq5},
we infer
\beno
\int_0^\infty \|\Pi_2(t')\|_{L^\infty}\,dt'\leq \frak{L}.
\eeno
Then we deduce from Proposition \ref{S0prop1} that
\beq\label{S7eq5}
\sup_{\la\in\La}\|\na\p_{X_\la(t)}X_\la(t)\|_{L^2}\leq \frak{L}\quad\forall\ t\in [0,\infty[.
\eeq

Let us denote by $\psi(t,\cdot)$  the flow associated with the
 vector field $u,$ that is
\begin{equation*}
 \left\{\begin{array}{l}
\displaystyle \f{d}{dt}\psi(t,x)=u(t,\psi(t,x)),\\
\displaystyle \psi(0,x)=x.
\end{array}\right.
\end{equation*}
Then it follows from   the standard theory of transport equation and \eqref{S6eq4} that
  \beq\label{S1eq19}
\begin{split}
\|\na\psi(t,\cdot)-Id\|_{L^\infty}
\leq \frak{N}\exp\left(\frak{N}\right)\quad\forall\  t\in [0,\infty[.
\end{split}
\eeq

Let $\Om(t)=\psi(t,\Omega_0)$, with the boundary parametrization $\psi(t,\gamma_0(\cdot)): \mathbb{S}^1\mapsto \p\Omega(t)$.
 \eqref{S1eq19} ensures  that $\p\Om(t)$ is of  $W^{1,\infty}$ class. Furthermore,
 we deduce from the transport equation of \eqref{INS} that
 \beno
  \r(t,x)={\eta_1}{\bf 1}_{\Om(t)}(x)+\eta_2{\bf
1}_{\Om(t)^c}(x). \eeno
Next we are going to prove that
$\p\Omega(t)$ belongs to the class of  ${H}^{3}$.
Notice that the boundary  $\p\Omega(t)$
is the level surface of $f(t,\cdot)$ with $f$ being transported by the $\psi$-flow:
\beno
 \left\{\begin{array}{l}
\displaystyle \p_tf+u\cdot\na f=0,\\
\displaystyle f(0,x)=f_0(x).
\end{array}\right. \eeno
So that the vector field $X_0(t,\cdot)\eqdefa\na^\perp f(t,\cdot)$ verifies the equation \eqref{Xtrans} with initial data $X_0=\na^\perp f_0$ and
\beq \label{S1eq11}
\begin{split} &X_0(t,\psi(t,x))=X_0(x)\cdot\nabla {\psi(t,x)} .
\end{split} \eeq
Then   in view of  \eqref{S1eq5}, we find
\beno
\begin{split}
\p_sX_0(t,\psi(t,\gamma_0(s)))=&\p_s\psi(t,\gamma_0(s))\cdot\na X_0(t,\psi(t,\gamma_0(s)))\\
=&\left(X_0(\gamma_0(s))\cdot\nabla {\psi(t,\gamma_0(s))}\right)\cdot\na X_0(t,\psi(t,\gamma_0(s)))\\
=&\left(X_0\cdot\na X_0\right)(t,\psi(t,\gamma_0(s))),
\end{split}
\eeno
which together with \eqref{S7eq5} ensures that
\beq\label{S1eq13}
X_0(t,\psi(t,\gamma_0(s)))\in
L^\infty(\R^+; \dot H^2({\Bbb S}^1)).
\eeq
As a result, it comes out
\beq\label{S1eq12a}
 \p_s(\psi(t,\gamma_0(s)))=X_0(\gamma_0(s))
 \cdot\nabla{ \psi(t,\gamma_0(s))}
= X_0(t,\psi(t,\gamma_0(s)))\in L^\infty(\R^+; \dot H^2({\Bbb S}^1)).\eeq Hence $\p\Omega(t)$ belongs to the class of  $H^{3}$ for any $t>0.$ This
completes the proof of Theorem \ref{thm2}.
\end{proof}

Finally in the Appendix \ref{Sappb}, we shall present a commutative estimate, which will be used frequently in the whole context;
while in Appendix \ref{append}, we shall  generalize Proposition \ref{S9prop1} for elliptic equation of divergence form
 with bounded coefficients which may have a small gap across a  surface.

\setcounter{equation}{0}
\section{The basic energy estimate}\label{Sect3}

The goal of this section is to prove Propositions \ref{S2prop1} and \ref{S2prop2}. Toward this and also for the
 $\dot H^1$ estimate of the tangential derivative of the velocity field, we first present the energy
estimate for the linearized equation of \eqref{INS}.

\begin{lem}\label{S2lem1}
{\sl Let $u$ be a smooth solenoidal vector field and $D_t\eqdefa \p_t+u\cdot\na.$    Let $(\rho, v, \na \pi)$ be a smooth enough solution to the following
system on $[0,T]:$ \beq\label{LINS} \left\{
\begin{array}
{l} \displaystyle D_t\rho  = 0  \qquad (t,x)\in [0,T]\times\Bbb{R}^2, \\
\displaystyle \rho D_t v  - \dive\left(2\mu(\rho)\cM(v)\right) + \nabla \pi = F, \\
\displaystyle \dive v = \hbar, \\
\displaystyle (\rho, v)|_{t=0}=(\rho_0,v_0).
\end{array}
\right. \eeq Then under the assumption of \eqref{S1eq7} and \eqref{S1eq2},  for any $t\leq T,$ we have
 \beq \label{ener4}
\begin{split}\f{d}{dt} \int_{\R^2}\mu(\rho)&|\cM(v)|^2\,dx +
\int_{\R^2}|\sqrt{\rho}D_tv|^2\,dx=\int_{\R^2}(F-\na\pi) | D_tv\,dx\\
&-\int_{\R^2}\mu(\r)\left((\p_1v^1)^2\p_1 u^1+(\p_1v^2+\p_2v^1)\bigl(\p_1u^1\p_1v^2+\p_2u^2\p_2v^1\bigr)\right.\\
&\qquad\qquad\ \  \left. +(\p_2v^2)^2\p_2 u^2+ (\p_1u^2+\p_2u^1)\bigl(
\p_1v^1\p_2v^1+\p_1v^2\p_2v^2\bigr)\right)\,dx.  \end{split}\eeq
Moreover, for
any $p\in [2, \infty[,$ we have
\beq \label{S2eq1} \|\na
v\|_{L^p}\leq  C \bigl(\|\hbar\|_{L^p}+\|\na v\|_{L^2}^{\f2p}\|(\rho D_tv -F)\|_{L^2}^{1-\f2p}\bigr).\eeq
}
\end{lem}

\begin{proof}   We first deduce from \eqref{S1eq2} and the transport equation of
  \eqref{LINS} that
  \beq \label{ener4a}
 m\leq  \r(t,\cdot)\leq M \andf \mu(\r(t,\cdot))-1=\mu(\r_0(\cdot))-1.
 \eeq
  While by taking $L^2$
inner product of the momentum equation of \eqref{LINS} with
$D_tv$ and then using integration by parts, we obtain \beno
\int_{\R^2}{\rho}|D_t v|^2\,dx  + 2\int_{\R^2} \mu(\rho) \cM(v) :\cM(D_t v)\,dx=\int_{\R^2}(F-\na\pi) | D_tv\,dx.
\eeno Note that $D_t\mu(\r)=0,$ we have \beno \begin{split}
2\int_{\R^2} \mu(\rho)\cM(v)
:\cM(D_tv)\,dx = & \f{d}{dt} \int_{\R^2}\mu(\rho)|\cM(u)|^2\,dx+2\int_{\R^2}\mu(\rho)\cM_{ij}(v)\p_iu\cdot\na v^j\,dx,
\end{split} \eeno
which together with the fact that
\beno
\begin{split}
\int_{\R^2}&\mu(\rho)\cM_{ij}(v)\p_iu\cdot\na v^j\,dx
=\int_{\R^2}\mu(\r)\Bigl(\p_1u^1\bigl((\p_1v^1)^2+\f12(\p_1v^2+\p_2v^1)\p_1v^2\bigr)\\
&+\p_2u^2\bigl((\p_2v^2)^2+\f12(\p_1v^2+\p_2v^1)\p_2v^1\bigr)+\f12\bigl(\p_1v^1\p_2v^1+\p_1v^2\p_2v^2\bigr)(\p_1u^2+\p_2u^1)\Bigr)\,dx,\end{split}
\eeno
gives rise to \eqref{ener4}.

Due to $\dive v=\hbar,$ let $w\eqdefa v+\na(-\D)^{-1}\hbar,$ we have $\dive w=0.$  Then it is easy to Observe that
\beno -\D w = \dive\bigl((\mu(\rho)-1)\cM(w)\bigr)-
\dive\bigl(\mu(\rho)\cM(w)\bigr). \eeno By taking the Leray projection operator, $\PP\eqdefa I+\na(-\D)^{-1}\dive,$ to the above
equation, we obtain \beq \label{S1eq1ag}
\na w =
\na(-\D)^{-1}\PP\dive\bigl((\mu(\rho)-1)\cM(w)\bigr)-\na(-\D)^{-1}\PP\dive\bigl(\mu(\rho)\cM(w)\bigr),
\eeq
so that by virtue of \eqref{ener4a}, we get, by applying
Gagliardo-Nirenberg inequality in 2-D, that for any $p\in [2,\infty[,$ \beq \label{S2eq1a}\begin{split} \|\na w\|_{L^p}
\leq &C\Bigl(\|(\mu(\rho)-1)\cM(w)\|_{L^p}
+ \|\na(-\D)^{-1}\PP\dive\bigl(\mu(\rho)\cM(v+\na(-\D)^{-1}\hbar)\bigr)\|_{L^p}\Bigr)\\
\leq &C\Bigl(\|\mu(\rho_0)-1\|_{L^\infty}\|\na
w\|_{L^p}+\|\hbar\|_{L^p}+
\|\na v\|_{L^2}^{\f2p}\|\PP\dive\left(\mu(\rho)\cM(v)\right)\|_{L^2}^{1-\f2p}\Bigr),
\end{split}\eeq
which together with \eqref{S1eq2} ensures that
\beno
\|\na w\|_{L^p} \leq C\Bigl(\|\hbar\|_{L^p}+
\|\na v\|_{L^2}^{\f2p}\|\PP\dive\left(\mu(\rho)\cM(v)\right)\|_{L^2}^{1-\f2p}\Bigr). \eeno
We thus deduce from the momentum equation of \eqref{LINS}  that
  \beno \|\na
v\|_{L^p}\leq \|\na w\|_{L^p}+C\|\hbar\|_{L^p}\leq C \bigl(\|\hbar\|_{L^p}+\|\na v\|_{L^2}^{\f2p}\|(\rho D_tv-F)\|_{L^2}^{1-\f2p}\bigr).\eeno
This proves \eqref{S2eq1} and the lemma.
\end{proof}

\begin{rmk} Compared with the ``pseudo-energy" method introduced by Desjardin \cite{desj} (see also \cite{AZ1, AZ2}),
here we take the $L^2$ inner product of the momentum equation \eqref{LINS} with $D_tv$ instead of $v_t,$ which
is simpler due to $D_t\mu(\r)=0.$
\end{rmk}

Let us now outline the proof of Propositions \ref{S2prop1} and \ref{S2prop2}.

\begin{proof}[Proof of Proposition \ref{S2prop1}]  We first get, by taking the $L^2$ inner product of the momentum equation of \eqref{INS}  with $u,$ that \beno
\f12\f{d}{dt}\int_{\R^2}|\sqrt{\rho}u|^2\,dx +2\int_{\R^2}\mu(\rho)|\cM(u)|^2\,dx =0.
 \eeno Integrating  the above equality over $[0,t]$ and using \eqref{ener4a} yields \eqref{ener2}.

While due to $\dive u=0,$ we get, by applying Lemma \ref{S2lem1}, that
\beq \label{S2eq2}
\begin{split}
\f{d}{dt} \int_{\R^2}\mu(\rho)|\cM(u)|^2\,dx +
\int_{\R^2}|\sqrt{\rho}D_tu|^2\,dx=&-\int_{\R^2}\na\Pi | D_tu\,dx\\
=&\int_{\R^2}\Pi | \p_iu^j\p_ju^i \,dx.
\end{split}
\eeq
We  deduce from \eqref{b.4a} that
$$
\begin{aligned}
\Bigl|&\int_{\R^2}
\Pi\partial_iu_j\partial^ju^i\,dx\Bigr| \lesssim \|\nabla
u\|_{L^2}\|\nabla u\|_{L^4}^2
+\|(-\Delta)^{-1}\dive(\rho D_tu
)\|_{BMO}
\bigl\|\na u_i\cdot\partial_iu\bigr\|_{\mathscr{H}^1},
\end{aligned}
$$ where $\|f\|_{\mathscr{H}^1}$ denotes the Hardy norm of $f.$
Whereas as  $\dive\,u=0,$  it follows from Theorem II.1 of \cite{CLMS} that \beno
\bigl\|\na u^i\cdot\partial_iu\bigr\|_{\mathscr{H}^1}\lesssim
\|\na u\|_{L^2}^2,\eeno from which,  and the fact that: $\|f\|_{BMO(\R^2)}\lesssim\|\nabla
f\|_{L^2(\R^2)},$ we infer
\beno
\Bigl|\int_{\R^2}
\Pi\partial_iu^j\partial_ju^i\,dx\Bigr|
\leq& C\|\nabla
u\|_{L^2}\bigl(\|\nabla u\|_{L^4}^2
 +\|\rho D_tu \|_{L^2}
\|\nabla u\|_{L^2}\bigr).
\eeno
While it follow from \eqref{S2eq1} that
\beq\label{B.4b}
\|\na u\|_{L^4}\leq C\|\na u\|_{L^2}^{\f12}\|\sqrt{\r} D_tu\|_{L^2}^{\f12}.
\eeq
As a result, it comes out
\beq \label{b.4af}
\Bigl|\int_{\R^2}
\Pi\partial_iu^j\partial_ju^i\,dx\Bigr|
\leq C\|\sqrt{\rho} D_tu \|_{L^2}
\|\nabla u\|_{L^2}^2.
\eeq
Inserting the above estimate into \eqref{S2eq2} gives rise to \beq\label{ener3}
\begin{split}&\f{d}{dt}\int_{\R^2}\mu(\rho)|\cM(u)|^2\,dx +\int_{\R^2}|\sqrt{\rho}D_tu|^2\,dx\leq C\|\na u\|_{L^2}^4. \end{split}\eeq
Due to \eqref{S1eq2} and \eqref{ener4a},  we have
\beno
\|\na u\|_{L^2}^2\leq C\int_{\R^2}|\cM(u)|^2\,dx\leq C\int_{\R^2}\mu(\rho)|\cM(u)|^2\,dx,
\eeno so that applying Gronwall's inequality to \eqref{ener3} yields
\beq
\label{ener3a}
\begin{split}
\|\na u\|_{L^\infty_t(L^2)}^2 +\int_0^t\int_{\R^2}|\sqrt{\rho}D_tu|^2\,dx\,dt'\leq & C\|\na u_0\|_{L^2}^2\exp\left(C\|\na u\|_{L^2(L^2)}^2\right)\\
\leq &C\|\na u_0\|_{L^2}^2\exp\left(C\|u_0\|_{L^2}^2\right), \end{split}\eeq
which together with \eqref{ener2} and \eqref{B.4b} implies that
\beq\label{S2eq3}
\begin{split}
\|u_t\|_{L^2_t(L^2)}\leq & \|D_tu\|_{L^2_t(L^2)}+\|u\|_{L^\infty_t(L^4)}\|\na u\|_{L^2_t(L^4)}\\
\leq & \|D_tu\|_{L^2_t(L^2)}+C\|u\|_{L^\infty_t(L^2)}^{\f12}\|\na u\|_{L^\infty_t(L^2)}^{\f12}\|\na u\|_{L^2_t(L^2)}^{\f12}\|D_t u\|_{L^2_t(L^2)}^{\f12}\\
\leq &C(1+\|u_0\|_{L^2})\|\na u_0\|_{L^2}\exp\left(C\|u_0\|_{L^2}^2\right).
\end{split}\eeq

On the other hand, by applying the Leray projection operator, $\PP,$  and the operator $\Q$ on the momentum equation of
\eqref{LINS}, we write \beno \PP\dive\bigl(\mu(\rho)\cM(u)\bigr)&=&
\PP(\rho D_tu),\\
\Q\dive\bigl(\mu(\rho)\cM(u)\bigr)-\na\Pi&=&\Q(\rho D_t u), \eeno
from which and  \eqref{ener3a}, we infer
\beno
\begin{split}
\int_0^t\int_{\R^2}\bigl(|\PP\dive\bigl(\mu(\rho)\cM(u)\bigr)|^2+|\Q\dive\bigl(\mu(\rho)\cM(u)\bigr)&-\na\Pi|^2\bigr) \,dx\,dt'\\
\leq & C\|\na u_0\|_{L^2}^2\exp\left(C\|u_0\|_{L^2}^2\right),
\end{split}
\eeno
which together with \eqref{ener3a} and \eqref{S2eq3} leads to \eqref{ener1}.
This completes the proof of Proposition \ref{S2prop1}.
\end{proof}

 \begin{proof}[Outline of the proof of Proposition \ref{S2prop2}] With Proposition \ref{S2prop1}, we deduce from Proposition 2.2 of \cite{AZ1} that \eqref{p.2a}
 holds except the estimate of $\|\w{t}^{(\f{1}2+\d)_-}D_tu\|_{L^2_t(L^2)}.$ Indeed It is easy to observe from \eqref{B.4b} that
\beno
\begin{split}
\|D_tu\|_{L^2}\leq &\|u_t\|_{L^2}+\|u\|_{L^4}\|\na u\|_{L^4}\\
\leq &\|u_t\|_{L^2}+C\|u\|_{L^2}^{\f12}\|\na u\|_{L^2}\|D_tu\|_{L^2}^{\f12},
\end{split}
\eeno
which implies
\beq \label{S3eq7y}
\|D_tu\|_{L^2}\leq C\bigl(\|u_t\|_{L^2}+\|u\|_{L^2}\|\na u\|_{L^2}^2\bigr).
\eeq
Then  applying Proposition \ref{S2prop2} gives rise to
\beq \label{S3eq7yp}
\begin{split}
\|\w{t}^{(\f{1}2+\d)_-}D_tu\|_{L^2_t(L^2)}^2\leq C\Bigl(\|\w{t}^{(\f{1}2+\d)_-}u_t\|_{L^2_t(L^2)}^2+C_0\int_0^t\w{t'}^{-\left(1+4\d\right)_-}\,dt'\Bigr)\leq C_0
\end{split}
\eeq This concludes the proof of proposition \ref{S2prop2}.
 \end{proof}

\setcounter{equation}{0}
\section{The energy estimate of $D_tu$}\label{Sect4}

The goal of this section is to prove Proposition \ref{S3prop1}.

\begin{lem}\label{S3lem1}
{\sl Let $(\r,u)$ be a smooth enough solution of \eqref{INS} on $[0,T^\ast[$ with $\mu(\r_0)\geq\f34.$ Then for any $t<T^\ast,$ we have
\beq\label{S3eq7v}
\begin{split}
\f{d}{dt}\Bigl(\f12\int_{\R^2}\r|D_tu|^2\,dx-\int_{\R^2}\Pi\p_iu^j\p_ju^i\,dx\Bigr)+\|\na D_tu\|_{L^2}^2\leq C\|\na u\|_{L^2}^2\|\sqrt{\r}D_tu\|_{L^2}^2.
\end{split}
\eeq}
\end{lem}

\begin{proof}
We first get, by applying the operator $D_t$ to the momentum equation of \eqref{INS}, that
\beq \label{S3eq1}
\begin{split}
\r D_t^2u^j-&2\p_i\left(\mu(\r)\cM_{ij}(D_tu)\right)+\p_jD_t\Pi\\
=&2[D_t;\p_i](\mu(\r)\cM_{ij}(u))-\p_i\left(\mu(\r)\bigl(\p_iu\cdot\na u^j+\p_ju\cdot\na u^i\bigr)\right)
+[\p_j;D_t]\Pi. \end{split}
\eeq
Taking $L^2$ inner product of the above equation with $D_tu,$ we write
\beq \label{S3eq2}
\begin{split}
\f12\f{d}{dt}&\int_{\R^2}\r|D_tu|^2\,dx+2\int_{\R^2}\mu(\r)|\cM(D_tu)|^2\,dx\\
=&-(\na D_t\Pi | D_tu)+2\bigl([D_t;\p_i](\mu(\r)\cM_{ij}(u)) | D_tu^j\bigr)\\
&+\bigl(\mu(\r)\bigl(\p_iu\cdot\na u^j+\p_ju\cdot\na u^i\bigr) | \p_i D_tu^j\bigr)+
\bigl([\p_j;D_t]\Pi | D_tu^j\bigr).
\end{split}
\eeq
We now handle term by term above.

\no$\bullet$ \underline{The estimate of $(\na D_t\Pi | D_tu)$}

Due to $\dive u=0,$ by using integration by parts,  we find
\beq\label{S3eq2a}
\begin{split}
-&(\na D_t\Pi | D_tu)=\int_{\R^2}D_t\Pi | \dive D_tu\,dx=\int_{\R^2}D_t\Pi | \p_iu^j\p_ju^i\,dx\\
=&\f{d}{dt}\int_{\R^2} \Pi \p_iu^j\p_ju^i\,dx-2\int_{\R^2}\Pi D_t\p_iu^j\p_ju^i\,dx\\
=&\f{d}{dt}\int_{\R^2} \Pi \p_iu^j\p_ju^i\,dx-2\int_{\R^2}\Pi \p_i D_t u^j\p_ju^i\,dx+\int_{\R^2}\Pi \p_iu\cdot\na u^j\p_ju^i\,dx.
\end{split}
\eeq
In view of \eqref{b.4a} and \eqref{B.4b}, we deduce from a similar derivation of   \eqref{b.4af} that
\beno
\begin{split}
\bigl|\int_{\R^2}\Pi \p_i D_t u^j\p_ju^i\,dx\bigr|
\leq &C\Bigl(\|\na u\|_{L^4}^2\|\na D_tu\|_{L^2}\\
&+
\|(-\Delta)^{-1}\dive(\rho D_tu
)\|_{BMO}
\bigl\|\na D_t u^j\cdot\partial_ju\bigr\|_{\mathscr{H}^1}\Bigr)\\
\leq &C\bigl(\|\na u\|_{L^4}^2+\|\na u\|_{L^2}\|\sqrt{\r}D_tu\|_{L^2}\bigr)\|\na D_tu\|_{L^2}\\
\leq & C\|\na u\|_{L^2}^2\|\sqrt{\r}D_tu\|_{L^2}^2+\f18\|\na D_tu\|_{L^2}^2.
\end{split}
\eeno
Observing that due to $\dive u=0,$ $\p_iu\cdot\na u^j\p_ju^i=0.$ Then by inserting the above estimate into \eqref{S3eq2a}, we obtain
$$
\bigl|(\na D_t\Pi | D_tu)+\f{d}{dt}\int_{\R^2} \Pi \p_iu^j\p_ju^i\,dx\bigr|
\leq \f18\|\na D_tu\|_{L^2_t(L^2)}^2
+ C\|\na u\|_{L^2}^2\|\sqrt{\r}D_tu\|_{L^2}^2. $$

\no$\bullet$ \underline{The estimate of $\bigl([D_t;\p_i](\mu(\r)\cM_{ij}(u)) | D_tu^j\bigr)$}

Note that $[D_t;\p_i]f=-\p_iu\cdot\na f,$ we write
\beno
\begin{split}
\bigl([D_t;\p_i](\mu(\r)\cM_{ij}(u)) | D_tu^j\bigr)=&-\bigl(\p_iu\cdot\na(\mu(\r)\cM_{ij}(u)) | D_tu^j\bigr)\\
=&\bigl(\mu(\r)\cM_{ij}(u) | \p_iu\cdot\na D_tu^j\bigr),
\end{split}
\eeno
which together with \eqref{B.4b} implies that
\beno \begin{split}
\bigl|\bigl([D_t;\p_i](\mu(\r)\cM_{ij}(u)) | D_tu^j\bigr)\bigr|\leq &C\|\na u\|_{L^4}^2\|\na D_tu\|_{L^2}\\
\leq &C\|\na u\|_{L^2}^2\|\sqrt{\r}D_tu\|_{L^2}^2+\f18\|\na D_tu\|_{L^2}^2.
\end{split}
\eeno

\no$\bullet$ \underline{The estimate of $\bigl(\mu(\r)\bigl(\p_iu\cdot\na u^j+\p_ju\cdot\na u^i\bigr) | \p_i D_tu^j\bigr)$}

We deduce from \eqref{B.4b} that
\beno
\begin{split}
\bigl|\bigl(\mu(\r)\bigl(\p_iu\cdot\na u^j+\p_ju\cdot\na u^i\bigr) | \p_i D_tu^j\bigr) \bigr|
\leq& C\|\na u\|_{L^4}^2\|\na D_tu\|_{L^2}\\
\leq& C\|\na u\|_{L^2}^2\|\sqrt{\r}D_tu\|_{L^2}^2+\f18\|\na D_tu\|_{L^2}.
\end{split}
\eeno

\no$\bullet$ \underline{The estimate of $\bigl([\p_j;D_t]\Pi | D_tu^j\bigr)$}

Thanks to \eqref{b.4a}, we write
\beno
\begin{split}
\bigl([\p_j;D_t]\Pi | D_tu^j\bigr)=&\int_{\R^2}\p_ju\cdot\na\Pi | D_tu^j\,dx\\
=&\int_{\R^2}\p_ju\cdot\na(-\D)^{-1}\dive(\r D_tu) | D_tu^j\,dx\\
&+\int_{\R^2}(-\D)^{-1}\dive\dive(2\mu(\r)\cM(u)) | \p_ju\cdot\na D_tu^j\,dx,
\end{split}
\eeno
from which and \eqref{B.4b}, we infer
\beno
\begin{split}
\bigl|\bigl([\p_j;D_t]\Pi | D_tu^j\bigr)\bigr|\leq &C\bigl(\|\na u\|_{L^2}\|D_tu\|_{L^4}^2
+\|\na u\|_{L^4}^2\|\na D_tu\|_{L^2}\bigr)\\
\leq &C\|\na u\|_{L^2}^2\|\sqrt{\r}D_tu\|_{L^2}^2+\f18\|\na D_tu\|_{L^2}^2.
\end{split}
\eeno

On the other hand,  since $\mu(\r_0)\geq\f34,$  we get, by using  \eqref{ener4a} and integration by parts, that
\beno
\begin{split}
2\int_{\R^2}\mu(\r)|\cM(D_tu)|^2\,dx\geq &\f32\int_{\R^2}|\cM(D_tu)|^2\,dx\\
= &\f32\int_{\R^2}\bigl(|\na D_tu|^2+\p_iD_tu^j\p_jD_tu^i\bigr)\,dx\\
= &\f32\int_{\R^2}\bigl(|\na D_tu|^2+(\dive D_tu)^2\bigr)\,dx\geq \f32\int_{\R^2}|\na D_tu|^2\,dx.
\end{split}
\eeno
Substituting the above estimates into \eqref{S3eq2} leads to \eqref{S3eq7v}. \end{proof}

\begin{proof}[Proof of Proposition \ref{S3prop1}]
By multiplying \eqref{S3eq7v} by $t-t_0$ and then integrating the resulting inequality over $[t_0,t],$ we obtain
\beq \label{S3eq7q}
\begin{split}
(t-t_0)\int_{\R^2}&\r|D_tu|^2(t)\,dx\leq 2(t-t_0)\bigl|\int_{\R^2}\Pi\p_iu^j\p_ju^i\,dx\bigr|+\int_{t_0}^t\int_{\R^2}\r|D_tu|^2\,dx\,dt'\\
&\quad+2\bigl|\int_{t_0}^t\int_{\R^2}\Pi\p_iu^j\p_ju^i\,dx\,dt'\bigr|+C\int_{t_0}^t(t'-t_0)\|\na u\|_{L^2}^2\|\sqrt{\r}D_tu\|_{L^2}^2\,dt'.
\end{split}
\eeq
In what follows, we  take $t_0=\f{t}2$ in  the above inequality.

By applying \eqref{b.4af}
and Proposition \ref{S2prop2}, we get
\beno
\begin{split}
t\bigl|\int_{\R^2}\Pi\p_iu^j\p_ju^i\,dx\bigr|\leq & C\sqrt{t}\|\na u\|_{L^2}^2\sqrt{t}\|\sqrt{\r}D_tu\|_{L^2}\\
\leq &\f{t}4\|\sqrt{\r}D_tu\|_{L^2}^2+Ct\|\na u\|_{L^2}^4\\
\leq&\f{t}4\|\sqrt{\r}D_tu\|_{L^2}+C_0\w{t}^{-(1+4\d)_-},
\end{split}
\eeno
and
\beno
\begin{split}
\bigl|\int_{\f{t}2}^t\int_{\R^2}\Pi\p_iu^j\p_ju^i\,dx\,dt'\bigr|\leq & C\int_{\f{t}2}^t\|\na u(t')\|_{L^2}^2\|D_tu(t')\|_{L^2}\,dt'\\
\leq &C_0\int_{\f{t}2}^t\w{t'}^{-3\left(\f12+\d\right)_-}\w{t'}^{\left(\f12+\d\right)_-}\|D_tu(t')\|_{L^2}\,dt'\\
\leq & C_0\bigl\|\w{t}^{\left(\f12+\d\right)_-}D_tu\bigr\|_{L^2_t(L^2)}\w{t}^{-(1+3\d)_-}
\leq C_0\w{t}^{-(1+3\d)_-}.
\end{split}
\eeno
Similarly,   we have
\beno
\begin{split}
\int_{\f{t}2}^tt'\|\na u\|_{L^2}^2\|\sqrt{\r}D_tu\|_{L^2}^2\,dt'\leq&
C_0\int_{\f{t}2}^t\w{t'}^{-\left(1+4\d\right)_-}\bigl\|\w{t'}^{\left(\f12+\d\right)_-}D_tu(t')\bigr\|_{L^2}^2\,dt'\\
\leq &C_0\bigl\|\w{t}^{\left(\f12+\d\right)_-}D_tu\bigr\|_{L^2_t(L^2)}^2\w{t}^{-\left(1+4\delta\right)_-}
\leq C_0\w{t}^{-\left(1+4\delta\right)_-},
\end{split}
\eeno
and
\beno\begin{split}
\int_{\f{t}2}^t\r\|D_t u\|_{L^2}^2\,dt'\leq &C\int_{\f{t}2}^t\w{t'}^{-\left(1+2\d\right)_-}\bigl\|\w{t'}^{\left(\f12+\d\right)_-}D_tu(t')\bigr\|_{L^2}^2\,dt'\\
\leq &C_0\bigl\|\w{t}^{\left(\f12+\d\right)_-}D_tu\bigr\|_{L^2_t(L^2)}^2\w{t}^{-\left(1+2\delta\right)_-}
\leq C_0\w{t}^{-\left(1+2\delta\right)_-}.\end{split}
\eeno
Inserting the above estimates into \eqref{S3eq7q} gives rise to
\beq \label{S3eq7u}
t\|D_tu(t)\|_{L^2}^2\leq C_0\w{t}^{-\left(1+2\delta\right)_-}. \eeq

On the other hand, by multiplying \eqref{S3eq7v} by $t\w{t}^{\left(1+2\delta\right)_-}$ and then integrating the resulting inequality over $[0,t],$  we find
\beq \label{S3eq7qa}
\begin{split}
\bigl\|\sqrt{t}\w{t}^{\left(\f12+\delta\right)_-}\na D_tu\|_{L^2_t(L^2)}^2\leq  & 2t\w{t}^{\left(1+2\delta\right)_-}\bigl|\int_{\R^2}\Pi\p_iu^j\p_ju^i\,dx\bigr|\\
&+2(1+\d)\int_{0}^t\w{t'}^{(1+2\d)_-}\int_{\R^2}\r|D_tu|^2\,dx\,dt'\\
&+4(1+\d)\bigl|\int_{0}^t\w{t'}^{(1+2\d)_-}\int_{\R^2}\Pi\p_iu^j\p_ju^i\,dx\,dt'\bigr|\\
&+C\int_{0}^tt'\w{t'}^{\left(1+2\delta\right)_-}\|\na u\|_{L^2}^2\|\sqrt{\r}D_tu\|_{L^2}^2\,dt'.
\end{split}
\eeq
We first deduce from Proposition \ref{S2prop2} that
\beno
\int_{0}^t\w{t'}^{(1+2\d)_-}\int_{\R^2}\r|D_tu|^2\,dx\,dt'\leq \bigl\|\w{t}^{\left(\f12+\d\right)_-}D_tu\bigr\|_{L^2_t(L^2)}^2\leq C_0.
\eeno
While it is easy to observe from \eqref{b.4af} and \eqref{S3eq7u} that
\beno
\begin{split}
2t\w{t}^{\left(1+2\delta\right)_-}\bigl|\int_{\R^2}\Pi\p_iu^j\p_ju^i(t)\,dx\bigr|\leq& Ct\w{t}^{\left(1+2\delta\right)_-}\|D_tu(t)\|_{L^2}\|\na u(t)\|_{L^2}^2\\
\leq &C_0\w{t}^{-\d_-},
\end{split}
\eeno
and
\beno
\begin{split}
\bigl|\int_{0}^t\w{t'}^{(1+2\d)_-}&\int_{\R^2}\Pi\p_iu^j\p_ju^i\,dx\,dt'\bigr|
\leq \int_0^t\w{t'}^{(1+2\d)_-}\|\na u\|_{L^2}^2\|D_tu\|_{L^2}\,dt'\\
&\leq  \|\na u\|_{L^2_t(L^2)}\bigl\|\w{t}^{\left(\f12+\d\right)_-}\na u\bigr\|_{L^\infty_t(L^2)} \bigl\|\w{t}^{\left(\f12+\d\right)_-}D_tu\bigr\|_{L^2_t(L^2)}
\leq C_0,
\end{split}
\eeno
and
\beno
\int_{0}^tt'\w{t'}^{\left(1+2\delta\right)_-}\|\na u\|_{L^2}^2\|\sqrt{\r}D_tu\|_{L^2}^2\,dt'\leq C_0\int_0^t\|\na u\|_{L^2}^2\,dt'\leq C_0.
\eeno
Substituting the above estimates into \eqref{S3eq7qa} yields
\beno
\bigl\|\sqrt{t}\w{t}^{\left(\f12+\delta\right)_-}\na D_tu\|_{L^2_t(L^2)}^2\leq C_0,\eeno
which together with  \eqref{S3eq7u} ensures \eqref{S3eq6}. This completes the proof of Proposition \ref{S3prop1}.
\end{proof}

\setcounter{equation}{0}

\section{$L^\infty$  estimate of $\na^2\D^{-1}g$}\label{Sect5}

In this section, we shall use some basic facts on Littlewood-Paley theory. Let us recall from \cite{BCD, Gra} that
\begin{defi}
\label {S9def1} {\sl Consider a smooth radial function~$\vf $
on~$\R,$ supported in~$[3/4,8/3]$ such that
$$  \sum_{j\in\Z}\varphi(2^{-j}\tau)=1
 \andf \chi(\tau)\eqdefa 1 - \sum_{j\geq 0}\varphi(2^{-j}\tau) \in C^\infty_0([0,4/3]).
$$ for any $\tau>0$.
We denote
$$
\Delta_ja=\cF^{-1}(\varphi(2^{-j}|\xi|)\widehat{a}(\xi))  \andf S_ja=\cF^{-1}(\chi(2^{-j}|\xi|)\widehat{a}(\xi)), \quad j \in \Z.
$$
Let $p \in ]1,+\infty[$ and~$s \in \R$. The Soblev norms are defined as
$$
\|a\|_{\dot W^{s,p}}\eqdefa\big\|\big(2^{js}\Delta_j
a\bigr)_{\ell^{2}(\Z)}\bigr\|_{L^p}. $$
}
\end{defi}
When $p=2,$  the Sobolev
spaces~$\dot W^{s,p}$ coincide with the classical homogeneous
Sobolev spaces $\dH^s$.

\begin{defi} \label{S9def2}
{\sl Let $f$ be a locally integrable function. We define the maximal function ${\rm M}f(x)$ as
\beno
{\rm M}f(x)=\sup_{r>0}\f1{|B(x,r)|}\int_{B(x,r)}|f(y)|\,dy
\eeno
where $B(x,r)$ denotes the ball with center $x$ and radius $r,$ and $|B(x,r)|$ the volume of ball $B(x,r)$.}
\end{defi}

Let $\phi_\e(x)=\f1{\e^2}\phi\left(\f{x}\e\right)$ for $\e>0.$ Let $\psi(x)=\sup_{|y|\geq |x|}|\phi(y)|$ satisfy
$\int_{\R^2}\psi(x)\,dx=A<\infty.$ Then it follows Proposition 1.16 and Remark 1.17 of \cite{BCD} that
\beq \label{S9eq0}
\sup_{\e>0}|(f\ast \phi_\e)(x)|\leq A {\rm M}f(x).
\eeq

\begin{lem}\label{S9lem1}
{\sl  Let $p\in ]1,\infty[,$  $X=(X^1,X^2)$ be a solenoidal vector field with $\na X\in L^p$ and $g\in L^\infty$ with $\p_Xg\in L^p.$
Let $\s(\xi)$ be an infinitely differentiable function with $\mbox{supp}\s\subset \{ \xi\in \R^2:\ |\xi|\geq \f12\ \},$ we assume that
 for all $\xi\in\R^2$ with $|\xi|\geq 1,$ there holds $\s(\la\xi)=\la^{-1}\s(\xi).$
Then we have
\beq \label{S9eq1}
\begin{split}
\|\p_X\s(D)g\|_{L^\infty}\leq  &
C\Bigl(\|g\|_{L^\infty}+\f1{I(X)}\bigl(\|X\|_{L^\infty}\|g\|_{L^p}\bigr)^{1-\f2p}\\
&\times\bigl(\|\na X\|_{L^p}\|g\|_{L^\infty}+\|\p_{X}g\|_{L^p}\bigr)^{\f2p}\ln\Bigl(e+\f{\|\na X\|_{L^p}\|g\|_{L^\infty}}{\|X\|_{L^\infty}\|g\|_{L^p}}\Bigr)\Bigr).
\end{split}
\eeq
}
\end{lem}

\begin{proof} We first get, by using Bony's decomposition \cite{Bo81} and a commutator's process, that
\beq \label{S9eq2}
\begin{split}
\p_X\s(D)g=&\s(D)\p_Xg+[T_{X^k}; \s(D)\p_k]g-\s(D)T_{\p_kg}X^k\\
&-\s(D)R(X^k,\p_kg)+T_{\s(D)\p_kg}X^k+R(X^k,\s(D)\p_kg).
\end{split}
\eeq
It is easy to observe that for any $p\in ]1,\infty[,$
\beno
\|\s(D)\p_Xg\|_{\dot W^{1,p}}\leq C\|\p_Xg\|_{L^p}.
\eeno
Considering the support properties to the Fourier transform of the terms in $T_{\p_kg}X^k,$ for each $x\in \R^2,$ we deduce from \eqref{S9eq0} that
\beno
\begin{split}
\bigl|\D_j\bigl(T_{\p_kg}X^k\bigr)(x)\bigr|\leq &C\sum_{|j'-j|\leq 4}{\rm M}\left(S_{j'-1}\p_kg\D_{j'}X^k\right)(x)\\
\leq &C\sum_{|j'-j|\leq 4}\|S_{j'-1}\p_kg\|_{L^\infty}{\rm M}\left(\D_{j'}X^k\right)(x)\\
\leq &C\|g\|_{L^\infty}\sum_{|j'-j|\leq 4}{\rm M}\left(2^{j'}\D_{j'}X^k\right)(x),
\end{split}
\eeno
so that
we have
\beno
\begin{split}
\|\s(D)T_{\p_kg}X^k\|_{\dot W^{1,p}}\leq& C\|g\|_{L^\infty}\Bigl\|\Bigl(\sum_{j\in\Z}\bigl(\sum_{|j'-j|\leq 4}{\rm M}(2^{j'}\D_{j'}X^k)(x)\bigr)^2\Bigr)^{\f12}\Bigr\|_{L^p}\\
\leq& C\|g\|_{L^\infty}\Bigl\|\Bigl(\sum_{j\in\Z}\bigl(2^{j}{\rm M}(\D_{j}X)(x)\bigr)^2\Bigr)^{\f12}\Bigr\|_{L^p}\\
\leq& C\|g\|_{L^\infty}\Bigl\|\Bigl(\sum_{j\in\Z}\bigl(2^{j}\D_{j}X(x)\bigr)^2\Bigr)^{\f12}\Bigr\|_{L^p}\leq C\|g\|_{L^\infty}\|\na X\|_{L^p}.
\end{split}
\eeno
Similarly, due to $\dive X=0$ and considering the support properties to the Fourier transform of the terms in $R(X^k,\p_kg),$ for each $x\in\R^2,$ we have
\beno
\begin{split}
|\D_j R(X^k,\p_kg)(x)|=&\bigl|\p_k\D_j\bigl(\sum_{j'\geq j-5}\D_{j'}X^k\wt{\D}_{j'}g\bigr)(x)\bigr|\\
\leq &C2^j{\rm M}\Bigl(\sum_{j'\geq j-5}\D_{j'}X\wt{\D}_{j'}g\Bigr)(x)\\
\leq &C\|g\|_{L^\infty}\Bigl(\sum_{j'\geq j-5}2^{j-j'}{\rm M}(2^{j'}\D_{j'}X)(x)\Bigr),
\end{split}
\eeno where $\wt{\D}_{j}g\eqdefa \sum_{|j'-j|\leq 1}\D_{j'}g.$
Applying Young inequality gives rise to
\beno
\begin{split}
\bigl\|\s(D)R(X^k,\p_kg)\|_{\dot W^{1,p}}\leq &\|R(X^k,\p_kg)\|_{L^{p}}\\
\leq &C\|g\|_{L^\infty}\Bigl(\sum_{\ell\geq -5}2^{-\ell}\Bigr)\Bigl\|\bigl(\sum_{j\in\Z}\bigl({\rm M}(2^{j}\D_{j}X)(x)\bigr)^2\bigr)^{\f12}
\Bigr\|_{L^p}\\
\leq &C\|g\|_{L^\infty}\|\na X\|_{L^p}.
\end{split}
\eeno
Exactly along the same line, for each $x\in\R^2,$ there holds
\beno
\begin{split}
\bigl|\D_jR(X^k,\s(D)\p_kg)(x)\bigr|\leq &C\|g\|_{L^\infty}\sum_{j'\geq j-5}{\rm M}(\D_{j'}X)(x)\\
\leq &C\|g\|_{L^\infty}2^{-j}\sum_{j'\geq j-5}2^{j-j'}{\rm M}(2^{j'}\D_{j'}X)(x),
\end{split}
\eeno
from which, we infer
\beno
\|R(X^k,\s(D)\p_kg)\|_{\dot W^{1,p}}\leq C\|g\|_{L^\infty}\|\na X\|_{L^p}.
\eeno
Next, we consider the estimate of the commutator, $[T_{X^k}; \s(D)\p_k]g.$
Let us denote $\th(\xi)\eqdefa-\xi\s(\xi)\varphi(\xi).$  It is easy to observe that
\beq\label{S9eq3}
[T_{X^k}; \s(D)\p_k]g=\sum_{j\in\Z} \bigl(S_{j-1}X^k\D_j\s(D)\p_kg-\s(D)\p_k(S_{j-1}X^k\D_jg)\bigr).
\eeq
Whereas by applying Abel rearrangement techniques, we get
\beno
\begin{split}
S_{j-1}&X^k\D_j\s(D)\p_kg-\s(D)\p_k(S_{j-1}X^k\D_jg)\\
=&\sum_{|j-\ell|\leq 4}\bigl(S_{j-1}X^k\D_j\s(D)\p_k\D_\ell g-\s(D)\p_k\D_\ell(S_{j-1}X^k\D_jg)\bigr)\\
=&\sum_{|j-\ell|\leq 4}\Bigl(S_{j-1}X\cdot\bigl(S_j\theta(2^{-\ell}D)g-S_{j-1}\theta(2^{-\ell}D)g\bigr)\\
&\qquad\qquad-\theta(2^{-\ell}D)\cdot\bigl(S_{j-1}X(S_jg-S_{j-1}g)\bigr)\Bigr)\\
=&-\sum_{|j-\ell|\leq 4}\bigl(\D_jX\cdot S_j\th(2^{-\ell} D)g-\th(2^{-\ell} D)\cdot(\D_j XS_jg)\bigr),
\end{split}
\eeno
which implies for each $x\in\R^2,$
\beq\label{S9eq3qe}
\begin{split}
\bigl(&S_{j-1}X^k\D_j\s(D)\p_kg-\s(D)\p_k(S_{j-1}X^k\D_jg)\bigr)(x)\\
&=\sum_{|j-\ell|\leq 4}2^{2\ell}\int_{\R^2}\check{\th}(2^{\ell}z)\bigl(\D_jX(x-z)-\D_jX(x)\bigr)S_jg(x-z)\,dz\\
&=\sum_{|j-\ell|\leq 4}2^{2\ell}\int_{\R^2}\check{\th}(2^{\ell}z)\int_0^1\D_j\na X(x-\tau z)\cdot
z\,d\tau S_jg(x-z)\,dz.
\end{split}
\eeq
Then for $\Psi(z)=\check{\th}(z)|z|,$ we deduce from \eqref{S9eq0} that
\beno
\begin{split}
\bigl|\bigl(&S_{j-1}X^k\D_j\s(D)\p_kg-\s(D)\p_k(S_{j-1}X^k\D_jg)\bigr)(x)\bigr|\\
&\leq C\|g\|_{L^\infty}\sum_{|j-\ell|\leq 4}2^\ell\int_0^1\int_{\R^2}|\Psi(2^\ell z)||\D_j\na X(x-\tau z)|\,dz\,d\tau\\
&=C2^{-j}\|g\|_{L^\infty}\sum_{|j-\ell|\leq 4}\int_0^1\left(\frac{2^\ell}\tau\right)^2\int_{\R^2}|\Psi(2^\ell z/\tau)||\D_j\na X(x- z)|\,dz\,d\tau\\
&\leq C2^{-j}\|g\|_{L^\infty}{\rm M}(\D_j\na X)(x).
\end{split}
\eeno
Considering the support properties to terms in \eqref{S9eq3},
we conclude that
\beno
\begin{split}
\|[T_{X^k}; \s(D)\p_k]g\|_{\dot W^{1,p}}\leq &C\|g\|_{L^\infty}\bigl\|\bigl(\sum_{j\in \Z}\bigl({\rm M}(\D_j\na X)(x)\bigr)^2\bigr)^{\f12}\bigr\|_{L^p}\\
\leq &C\|g\|_{L^\infty}\|\na X\|_{L^p}.
\end{split}
\eeno
By summing up the estimates, we achieve
\beno
\bigl\|\p_X\s(D)g-T_{\s(D)\p_kg}X^k\bigr\|_{\dot W^{1,p}}\leq C\bigl(\|\p_Xg\|_{L^p}+\|\na X\|_{L^p}\|g\|_{L^\infty}\bigr),
\eeno
from which  and the 2-D interpolation inequality
\beno
\|a\|_{L^\infty}\leq C\|a\|_{L^p}^{1-\f2p}\|\na a\|_{L^p}^{\f2p},
\eeno  we deduce that
\beq\label{S9eq1a}
\begin{split}
\bigl\|\p_X\s(D)g-T_{\s(D)\p_kg}X^k\bigr\|_{L^\infty}
\leq C\bigl(\|X\|_{L^\infty}\|g\|_{L^p}\bigr)^{1-\f2p}\bigl(\|\na X\|_{L^p}\|g\|_{L^\infty}+\|\p_{X}g\|_{L^p}\bigr)^{\f2p}.
\end{split}
\eeq
Here we used the fact that
\beq \label{S9eq1b}
\|\D_j(T_{\s(D)\p_kg}X^k)\|_{L^p}\leq C\|g\|_{L^p}\|X\|_{L^\infty}.
\eeq
Indeed it is easy to observe that
\beno
\begin{split}
\bigl\|\D_jT_{\s(D)\p_kg}X^k)\bigr\|_{L^p}\leq &\sum_{|j'-j|\leq 4}\|S_{j'-1}(\s(D)\p_kg)\|_{L^p}\|\D_{j'}X^k\|_{L^\infty} \andf\\
\|S_{j'-1}(\s(D)\p_kg)\|_{L^p}\leq &  C\|g\|_{L^p},
\end{split}
\eeno
which implies \eqref{S9eq1b}.

While note that $\mbox{supp}\s\subset \{ \xi\in\R^2:\ |\xi|\geq \f12\ \},$  we have
\beno
\|\D_j\s(D)\p_kg\|_{L^\infty}\leq C\|g\|_{L^\infty}\Rightarrow  \|S_{j'-1}(\s(D)\p_kg)\|_{L^\infty}\leq  Cj\|g\|_{L^\infty},
\eeno
so that we find
\beq\label{S9eq1c}
\begin{split}
\bigl\|\D_j(T_{\s(D)\p_kg}X^k)\bigr\|_{L^p}\leq &\sum_{|j'-j|\leq 4}\|S_{j'-1}(\s(D)\p_kg)\|_{L^\infty}\|\D_{j'}X^k\|_{L^p}\\
\leq & Cj2^{-j}\|g\|_{L^\infty}\|\na X\|_{L^p}.
\end{split}
\eeq
Then for any integer $N,$ we deduce from Bernstein Lemma and \eqref{S9eq1b}, \eqref{S9eq1c} that
\beno
\begin{split}
\|T_{\s(D)\p_kg}X^k\|_{L^\infty}
\lesssim &\sum_{j\leq N}2^{\f{2j}p}\|X\|_{L^\infty}\|g\|_{L^p}+\sum_{j>N }j2^{-j\left(1-\f{2}p\right)}\|\na X\|_{L^p}\|g\|_{L^\infty}\\
\lesssim &2^{\f{2N}p}\|X\|_{L^\infty}\|g\|_{L^p}+N2^{-N\left(1-\f{2}p\right)}\|\na X\|_{L^p}\|g\|_{L^\infty}.
\end{split}
\eeno
Taking $N$ so that $2^N\sim \f{\|g\|_{L^\infty}\|\na X\|_{L^p}}{\|g\|_{L^p}\|X\|_{L^\infty}}$ in the above inequality gives rise to
\beq \label{S9eq1d}
\|T_{\s(D)\p_kg}X^k\|_{L^\infty}\leq C\bigl(\|X\|_{L^\infty}\|g\|_{L^p}\bigr)^{1-\f2p}\bigl(\|\na X\|_{L^p}\|g\|_{L^\infty}\bigr)^{\f2p}
\ln\Bigl(e+\f{\|\na X\|_{L^p}\|g\|_{L^\infty}}{\|X\|_{L^\infty}\|g\|_{L^p}}\Bigr).
\eeq Together with \eqref{S9eq1a}, we conclude the proof of \eqref{S9eq1}, and thus  Lemma \ref{S9lem1}.
\end{proof}

\begin{rmk}\label{S9rmk1}
Note for any $s\in ]2/p,1[$ $ N2^{-N\left(1-\f{2}p\right)}\leq C_s2^{-N\left(s-\f{2}p\right)},$ so that we deduce from the proof of \eqref{S9eq1d} that
\beq \label{S9eq1e}
\|T_{\s(D)\p_kg}X^k\|_{L^\infty}\leq C_s\bigl(\|X\|_{L^\infty}\|g\|_{L^p}\bigr)^{1-\f2{ps}}\bigl(\|\na X\|_{L^p}\|g\|_{L^\infty}\bigr)^{\f2{ps}}.
\eeq
\end{rmk}

\begin{proof}[Proof of Proposition \ref{S9prop1}] The proof of this proposition is motivated by that of Theorem 3.3.1 of \cite{Ch1995}.
Indeed for any vector field $X(x)=(X^1(x),X^2(x)),$ there hold
\beno
\begin{split}
&|X(x)|\p_1^2=\f{X^1(x)X(x,D)\p_1-X^2(x)X(x,D)\p_2+(X^2(x))^2\D}{|X(x)|},\\
&|X(x)|\p_2^2=\f{X^2(x)X(x,D)\p_2-X^1(x)X(x,D)\p_1+(X^1(x))^2\D}{|X(x)|},\\
&|X(x)|\p_1\p_2=\f{X^1(x)X(x,D)\p_2+X^1(x)X(x,D)\p_2-X^1(x)X^2(x)\D}{|X(x)|}.
\end{split}
\eeno
Since $|X^i(x)|\leq |X(x)|,$ we infer that for each $x\in\R^2,$ there holds
\beq\label{S9eq5}
|X(x)|\bigl|\na^2(-\D)^{-1}g(x)\bigr|\leq C\bigl(|X(x)|\|g\|_{L^\infty}+\bigl\|X(x,D)\na(-\D)^{-1}g\bigr\|_{L^\infty}\bigr).
\eeq
Let us denote
\beno
U_\la\eqdefa \bigl\{ x\in\R^2:\ |X_\la(x)|\geq I(X)\ \bigr\}.
\eeno
Then we deduce from \eqref{S9eq5} that
\beno
\bigl\|\na^2(-\D)^{-1}g\bigr\|_{L^\infty(U_\la)}\leq C\bigl(\|g\|_{L^\infty}+\f1{I(X)}\bigl\|\p_X\na(-\D)^{-1}g\bigr\|_{L^\infty}\bigr).
\eeno
from which and Definition \ref{S1def1}, we deduce that
\beq\label{S9eq9a}
\|\na^2 \Delta^{-1} g\|_{L^\infty}\leq C\Bigl(\|g\|_{L^\infty}+\f1{I(X)}\sup_{\la\in\Lambda}\|\p_{X_\lambda}\na(-\D)^{-1} g\|_{L^\infty}\Bigr).
\eeq
Let us now split $\na^2 \Delta^{-1}g$ as
\beno
\na^2 \Delta^{-1} g=\na^2 \Delta^{-1}S_0 g+\na^2 \Delta^{-1}(1-S_0) g.
\eeno
It is easy to observe that
\beq
\label{S9eq9b}
\begin{split}
\|\na^2 \Delta^{-1}S_0 g\|_{L^\infty}\leq & C\bigl(\|\na^2 \Delta^{-1}S_0 g\|_{L^p}+\|\na^3 \Delta^{-1}S_0 g\|_{L^p}\bigr)\leq C\|g\|_{L^p}.
\end{split}
\eeq
Whereas it follows from \eqref{S9eq9a} that
\beno
\begin{split}
\|\na^2 \Delta^{-1}(1-S_0) g\|_{L^\infty}\leq &C\Bigl(\|(1-S_0) g\|_{L^\infty}+\f{1}{I(X)}\sup_{\la\in\Lambda}\bigl\|\p_{X_\la}\na \Delta^{-1}(1-S_0) g\bigr\|_{L^\infty}\Bigr)
\end{split}
\eeno
from which,  \eqref{S9eq1} and \eqref{S9eq1e},  for any $s\in ]2/p,1[,$ we infer that
\beno
\begin{split}
\|\na^2 \Delta^{-1}(1-S_0) g&\|_{L^\infty}\leq  C_s\Bigl\{\|g\|_{L^\infty}+\f1{I(X)}\sup_{\la\in\La}\Bigl(\bigl(\|\na X_\la\|_{L^p}\|g\|_{L^\infty}+\|\p_{X_\la}g\|_{L^p}\bigr)^{\f2p}
\\
&\times\bigl(\|X_\la\|_{L^\infty}\|g\|_{L^p}\bigr)^{1-\f2p}+\bigl(\|\na X_\la\|_{L^p}\|g\|_{L^\infty}\bigr)^{\f2{ps}}\bigl(\|X_\la\|_{L^\infty}\|g\|_{L^p}\bigr)^{1-\f2{ps}}\Bigr)\Bigr\}.
\end{split}
\eeno
Along with \eqref{S9eq9b}, we complete the proof of \eqref{S9eq4}.
\end{proof}

As an application of Proposition \ref{S9prop1}, let us now present the proof of Corollary \ref{S9col1}.

\begin{proof}[Proof of Corollary \ref{S9col1}] Let $X$ be a smooth enough solution of \eqref{Xtrans}, we first deduce
from the transport equation of \eqref{S3eq6a} that for any $\ell\in \N,$
\beno
\p_t\p_X^\ell\mu(\r)+u\cdot\na\p_X^\ell\mu(\r)=0,
\eeno
which \eqref{S9eq11af}  implies
\beq\label{S9eq20}
\|\p_X^\ell\mu(\r(t))\|_{L^p}=\|\p_{X_0}^\ell\mu(\r_0)\|_{L^p}\ \ \forall p\in [1,\infty] \andf \sup_{\la\in\La}\|\p_{X_\la(t)}\mu(\r(t))\|_{L^\infty}\leq C_{X,\mu}.
\eeq
In what follows, let us take any $X(t)\eqdefa \left(X_\la(t)\right)_{\la\in\La}$ where $X_\la(t)$ solves \eqref{Xtrans} with initial data $X_\la$ for any fixed $\la\in\La.$
We denote $\hbar\eqdefa -\na(-\D)^{-1}\mbox{Tr}(\na X\na u).$ Then in view of \eqref{S1eq1ag},we write
\beno
\na(u_X-\hbar)  =
\na(-\D)^{-1}\PP\dive\bigl((\mu(\rho)-1)\cM(u_X-\hbar)\bigr)-\na(-\D)^{-1}\PP\dive\bigl(\mu(\rho)\cM(u_X-\hbar)\bigr),
\eeno
from which and\eqref{XLINS}, we infer
\beq \label{S4eq3p}
\begin{split}
\na u_X=&\na \hbar+\na(-\D)^{-1}\PP\dive\bigl(\mu(\rho)\cM(\hbar)\bigr)-\na(-\D)^{-1}\PP\bigl(\r D_tu_X-G)\\
&\qquad\qquad\qquad\qquad\qquad+\na(-\D)^{-1}\PP\dive\bigl((\mu(\rho)-1)\cM(u_X-\hbar)\bigr).
\end{split}
\eeq
Yet it follows from \eqref{S3eq7} that
\beq \label{S4eq365}
\begin{split}
\r D_tu_X^j-G^j=&\p_X(\r D_tu^j)-\p_jX\cdot\na\Pi-\p_i\left(2\p_X\mu(\r)\cM_{ij}(u)\right)\\
&+\p_i\left(\mu(\r)\bigl(\p_iX\cdot\na u^j+\p_jX\cdot\na u^i\bigr)\right)
+\p_iX\cdot\na\left( 2\mu(\r)\cM_{ij}(u)\right),
\end{split}
\eeq
we then deduce from \eqref{S4eq3p} that
\beno
\begin{split}
\|\na u_X\|_{L^{p}}\leq & C\Bigl(\|\na \hbar\|_{L^{p}}+\|\na u\|_{L^{p}}+\|\mu(\r)-1\|_{L^\infty}\|\na u_X\|_{L^{p}}\\
&+\|\na u\|_{L^{\infty}}\|\na X\|_{L^p}+\bigl\|\na(-\D)^{-1}\PP\bigl(\p_X(\r D_tu)-\na X^k\p_k\Pi\bigr)\bigr\|_{L^{p}}\Bigr).
\end{split}
\eeno
Then under the condition \eqref{S1eq2},  we deduce from \eqref{ener4a} that
\beq\label{S9eq12}
\begin{split}
\|\na u_X\|_{L^p}\leq & C\Bigl(\|\na \hbar\|_{L^p}+\|\na u\|_{L^p}+\|\na u\|_{L^{\infty}}\|\na X\|_{L^{p}}\\
&\quad +\bigl\|\na(-\D)^{-1}\PP\bigl(\p_X(\r D_tu)-\na X^k\p_k\Pi\bigr)\bigr\|_{L^{p}}\Bigr).
\end{split}
\eeq
Due to $\dive X=0,$ we find
\beno
\begin{split}
\|\na(-\D)^{-1}{\Bbb P}\p_X(\r D_tu)\|_{L^{p}}= &\|\na(-\D)^{-1}{\Bbb P}\p_k(X^k\r D_tu)\|_{L^{p}} \\
\leq &C\|X\|_{L^\infty}\|D_tu\|_{L^p}
\leq C\|X\|_{L^\infty}\|D_tu\|_{L^2}^{\f2p}\|\na D_tu\|_{L^2}^{1-\f2p}.
\end{split}
\eeno
Whereas according to \eqref{b.4a}, we infer
\beno
\begin{split}
\bigl\|\na(-\D)^{-1}\PP(\na X^k\p_k\Pi_1)\bigr\|_{L^{p}}\leq &C\|\na X\|_{L^p}\|\na \Pi_1\|_{L^2}\\
\leq &C\|\na X\|_{L^p}\| D_tu\|_{L^2},
\end{split}
\eeno
and
\beno
\begin{split}
\bigl\|\na(-\D)^{-1}\PP(\na &X^k\p_k\Pi_2)\bigr\|_{L^{p}}\leq C\|\na X\|_{L^p}\|\Pi_2\|_{L^\infty}.
\end{split}
\eeno
For any $s\in ]2/p,1[,$ by substituting the above estimates into \eqref{S9eq12} and using \eqref{S0eq5}, we obtain
\beno
\begin{split}
\|\na u_{X}(t)\|_{L^p}\leq & C_s\Bigl\{\|\na u\|_{L^{p}}+\|D_tu\|_{L^2}^{\f2p}\|\na D_tu\|_{L^2}^{1-\f2p}\|X\|_{L^\infty}+\f{\sup_{\la\in\La}\|X_\la(t)\|_{L^\infty}^{1-\f2p}}{I(X(t))}\|\na u\|_{L^p}\\
&+\bigl(\|\na u\|_{L^p}+\|\na u\|_{L^\infty}+\|D_tu\|_{L^2}\bigr)\|\na X\|_{L^p}\\
&+\|\mu(\r_0)-1\|_{L^\infty}\f{\|\na X\|_{L^p}}{I(X(t))}\sup_{\la\in\La}\Bigl(\bigl(\|\na X_\la\|_{L^p}\|\na u\|_{L^\infty}
+\|\na u_{X_\la}\|_{L^p}\bigr)^{\f2p}
\\
&\times\bigl(\|X_\la\|_{L^\infty}\|\na u\|_{L^p}\bigr)^{1-\f2p}+\bigl(
\|\na X_\la\|_{L^p}\|\na u\|_{L^\infty}
\bigr)^{\f2{ps}}\bigl(\|X_\la\|_{L^\infty}\|\na u\|_{L^p}\bigr)^{1-\f2{ps}}\Bigr)\Bigr\}.
\end{split}
\eeno
Then under the assumption of \eqref{S9eq11as}, by taking $X=X_\la$ in the above inequality and take supremum of the resulting inequality for $\la\in\La,$
we achieve
\beno
\begin{split}
\sup_{\la\in\La}\|\na u_{X_\la(t)}(t)\|_{L^p}\leq  C\Bigl(&\bigl[\|\na u\|_{L^p}+\|D_tu\|_{L^2}^{\f2p}\|\na D_tu\|_{L^2}^{1-\f2p}\bigr]\sup_{\la\in\La}\|X_\la(t)\|_{L^\infty}\\
&+\bigl(\|\na u\|_{L^p}+\|\na u\|_{L^\infty}+\|D_tu\|_{L^2}\bigr)\sup_{\la\in\La}\|\na X_\la(t)\|_{L^p}\\
&+\Bigl[1+\f{\sup_{\la\in\La}\|X_\la(t)\|_{L^\infty}^{1-\f2p}}{I(X(t))}\Bigr]\|\na u\|_{L^p} \Bigr)+
\f12\sup_{\la\in\La}\|\na u_{X_\la(t)}(t)\|_{L^p},
\end{split}
\eeno
which implies \eqref{S9eq11}.
This completes the proof of the Corollary \ref{S9col1}.
\end{proof}

\setcounter{equation}{0}
\section{The $L^1_t({\rm Lip})$ estimate of the velocity field}\label{Sect6}

Let us first present the proof of Proposition \ref{S6prop1}.

\begin{proof}[Proof of Proposition \ref{S6prop1}] Due to  $\dive u=0,$ for any $r\in [1,\infty[,$ we get, by taking $L^2$ inner product of \eqref{Xtrans} with $|X|^{r-2}X,$ that
 \beno
 \begin{split}
 \bigl|\f{d}{dt}\|X(t)\|_{L^r}^r\bigr|= &r\bigl|\bigl(X\cdot\na u | |X|^{r-2}X\bigr)\bigr|\\
 \leq &r\|\na u\|_{L^\infty}\|X\|_{L^r}^r,
 \end{split}
 \eeno
 which implies the first inequality of \eqref{S6eq1} $r\in [1,\infty[$. The case for $r=\infty$ can be proved similarly.

 Since $\left(X_\la(0)\right)_{\la\in\La}$ is a non-degenerate family of vector fields, we deduce from \eqref{S6eq1} that
 $\left(X_\la(t)\right)_{\la\in\La}$ is also a non-degenerate family of vector fields. Moreover, there holds
 \beq \label{S6eq0}
 I(X(0))\exp\bigl(-\|\na u\|_{L^1_t(L^\infty)}\bigr)\leq I(X(t)) \leq I(X(0))\exp\bigl(\|\na u\|_{L^1_t(L^\infty)}\bigr).
 \eeq

Let us now turn to the proof of \eqref{S6eq2}.
We  first take  $\p_i$ to the System \eqref{Xtrans} to get
$$\pa_t \pa_i X^j+u\cdot\nabla \pa_i X^j=-\pa_i u\cdot\nabla X^j+\pa_i  u^j_X.$$
By multiplying the above equation by $(\p_iX^j)^{p-1}$ and integrating the resulting equality over $\R^2,$ we obtain
$$\frac{d}{dt}\|\nabla X(t)\|_{L^p}\leq \|\nabla u\|_{L^\infty}\|\nabla X\|_{L^p}+\|\nabla u_X \|_{L^p}.$$
Thanks to Corollary \ref{S9col1}, by taking $X(t)=X_\la(t)$ in the above inequality and integrating the inequality over $[0,t],$ and then inserting \eqref{S9eq11} to the resulting inequality,
we achieve
\beq \label{S6eq3}
\begin{split}
\sup_{\la\in\La}\|\na X_\la(t)\|_{L^p}\leq & \sup_{\la\in\La}
\|\na X_\la(0)\|_{L^p}+C\int_0^t\|\na u\|_{L^p}\Bigl(1+\f{\sup_{\la\in\La}\|X_\la(t')\|_{L^\infty}^{1-\f2p}}{I(X(t'))}\Bigr)\,dt'\\
&+ C\int_0^t
\bigl(\|\na u\|_{L^p}+\|D_t u\|_{L^2}^{\f2p}\|\na D_t u\|_{L^2}^{1-\f2p}\bigr)
\sup_{\la\in\La}\|X_\la(t')\|_{L^\infty}\,dt'\\
&+C\int_0^t\bigl(\|\na u\|_{L^p}
+\|\na u\|_{L^\infty}+\|D_tu\|_{L^2}\bigr)\sup_{\la\in\La}\|\na X_\la(t')\|_{L^p} \,dt'.
\end{split}
\eeq
It is easy to observe from \eqref{S2eq1} and Propositions \ref{S2prop2} and \ref{S3prop1} that for any $r\in [2,\infty[$
\beq \label{S4eq15}
\|\na u(t)\|_{L^r}\leq C\|\na u(t)\|_{L^2}^{\f2r}\|D_t u(t)\|_{L^2}^{1-\f2r}\leq C_1t^{\f1r-\f12}\w{t}^{-\left(\f12+\d\right)_-}.
\eeq
By virtue of \eqref{S4eq15}, as long as $\d>\f1p,$ we have
\beno
\int_0^t\|\na u(t')\|_{L^p}\,dt'\leq C_0\int_0^t(t')^{\f1p-\f12}\w{t'}^{-\left(\f12+\d\right)_-}\,dt'\leq C_0.
\eeno
Whereas it follows from Proposition \ref{S2prop2} that
\beq\label{S4eq15re}
\begin{split}
\int_0^t\|D_tu(t')\|_{L^2}\,dt'= &\int_0^t\w{t'}^{-\left(\f12+\d\right)_-}\bigl\|\w{t'}^{\left(\f12+\d\right)_-}D_tu(t')\|_{L^2}\,dt'\\
\leq &C\|\w{t'}^{\left(\f12+\d\right)_-}D_tu\|_{L^2_t(L^2)}\leq C_0,
\end{split}
\eeq
and
\beno
\begin{split}
\int_0^t&\|D_t u(t')\|_{L^2}^{\f2p}\|\na D_t u(t')\|_{L^2}^{1-\f2p}\,dt'\\
=&\int_0^t(t')^{-\f12\left(1-\f2p\right)}\w{t'}^{-\left(\f12+\d\right)_-}
\bigl\|\w{t'}^{\left(\f12+\d\right)_-}D_t u(t')\bigr\|_{L^2}^{\f2p}\bigl\|\sqrt{t'}\w{t'}^{\left(\f12+\d\right)_-}\na D_t u(t')\bigr\|_{L^2}^{1-\f2p}\,dt'\\
\leq &C\bigl\|\w{t'}^{\left(\f12+\d\right)_-}D_t u\bigr\|_{L^2_t(L^2)}^{\f2p}\bigl\|\sqrt{t'}\w{t'}^{\left(\f12+\d\right)_-}\na D_t u\bigr\|_{L^2_t(L^2)}^{1-\f2p}\leq C_0.
\end{split}
\eeno
Hence by applying Gronwall's inequality  to \eqref{S6eq3}, we arrive at
\beno
\begin{split}
\sup_{\la\in\La}\|\na X_\la(t)\|_{L^p}\leq C_0\Bigl(1+\sup_{\la\in\La}\|\na X_\la(0)\|_{L^p}&+\sup_{\la\in\La}\|X_\la\|_{L^\infty_t(L^\infty)}\\
&+\sup_{\tau\in [0,t]}\f{\sup_{\la\in\La}\|X_\la(\tau)\|_{L^\infty}^{1-\f2p}}{I(X(\tau))}\Bigr){\rm V}(t).
\end{split}
\eeno
Inserting the Inequality \eqref{S6eq0} to the above inequality leads to \eqref{S6eq2}.
This completes the proof of the proposition.
\end{proof}

\begin{proof}[Proof of Proposition \ref{S9lem3}]  Let  $S_0$ be the partial sum operator introduced in Definition \ref{S9def1}. We denote
\beq\label{S6eq8a}
\frak{A}\eqdefa
(-\D)^{-1}(I-S_0)\dive\bigl((\mu(\rho)-1)\cM(u)\bigr) \andf \frak{B}(\r)\eqdefa \|\mu(\r)-1\|_{L^\infty}.\eeq
Due to $\dive u=0,$ according to the momentum equation of \eqref{INS}, we write
\beno
 -\Delta u = -\PP(\rho D_tu)
+\PP\dive\bigl((\mu(\rho)-1)\cM(u)\bigr),\eeno
 which implies
\beq\label{S6eq7} \begin{split} (I-S_0)\na u =&- \na
(-\D)^{-1}\PP(I-S_0)(\rho D_tu)  +\na\PP\frak{A}. \end{split}\eeq

By virtue of the the 2-D interpolation inequality that
\beno
\|f\|_{L^\infty}\leq C\|f\|_{L^4}^{\f12}\|\na f\|_{L^4}^{\f12}\leq C\|f\|_{H^1}^{\f12}\|\na f\|_{L^4}^{\f12},
\eeno
and that
\beno
\bigl\|\na
(-\D)^{-1}\PP(I-S_0)(\rho D_tu)\bigr\|_{H^1}\leq C\|\r D_tu\|_{L^2},\eeno
we infer
\beq \label{S6eq8}
\begin{split}
\bigl\|\na
(-\D)^{-1}\PP(I-S_0)(\rho D_tu)\bigr\|_{L^\infty}\leq &C\|\r D_tu\|_{L^2}^{\f12}\|\r D_tu
\|_{L^4}^{\f12}\\
\leq &C\|D_tu\|_{L^2}^{\f34}\|\na D_tu
\|_{L^2}^{\f14}.
\end{split}
\eeq

On the other hand, let $X_\la(t)$ be the corresponding solution of \eqref{Xtrans} with initial data $X_\la(0).$
Thanks to \eqref{S6eq0},  $\left(X_\la(t)\right)_{\la\in\La}$ is a non-degenerate family of vector fields.
 Recall that $\PP=Id+\na(-\D)^{-1}\dive,$  for any $s\in ]2/p,1[,$ we
get, by applying Proposition \ref{S9prop1}, that
\beq\label{S6eq9}
\begin{split}
\bigl\|\na\PP\frak{A}\bigr\|_{L^\infty}
\lesssim &\frak{B}(\r)\|\na u\|_{L^p}+ \|\na\frak{A}\|_{L^\infty} +\f1{I(X)}\sup_{\la\in\Lambda}
\Bigl\{\bigl(\frak{B}(\r)\|X_\la\|_{L^\infty}\|\na u\|_{L^p}\bigr)^{1-\f2p}\\
&\qquad\qquad\times \bigl(\|\na X_\la\|_{L^p}\|\na\frak{A}\|_{L^\infty}+\bigl\|\p_{X_\la}\na\frak{A}\bigr\|_{L^p}\bigr)^{\f2p}\\
&\qquad\quad+\bigl(\frak{B}(\r)\|X_\la\|_{L^\infty}\|\na u\|_{L^p}\bigr)^{1-\f2{ps}} \Bigl(\|\na X_\la\|_{L^p}\|\na\frak{A}\|_{L^\infty}\Bigr)^{\f2{ps}}\Bigr\}.
\end{split}
\eeq
It follows from  Proposition \ref{Saprop1}  that
$$\bigl\|\p_{X_\la}\na\frak{A}\bigr\|_{L^{p}}
\lesssim \frak{B}(\r)\bigl(\|\na X_\la\|_{L^p}\|\na u\|_{L^\infty}+\|\na u_{X_\la}\|_{L^p}\bigr)+\|\p_{X_\la}\mu(\r)\|_{L^\infty}\|\na u\|_{L^p},
 $$
so that in view of  \eqref{S9eq20}, we get, by applying Young's inequality, that
\beno
\begin{split}
&\f1{I(X)}\bigl(\frak{B}(\r)\|X_\la\|_{L^\infty}\|\na u\|_{L^p}\bigr)^{1-\f2p}\bigl\|\p_{X_\la}\na\frak{A}\bigr\|_{L^{p}}^{\f2p}\\
&\lesssim \frak{B}(\r)\Bigl(I(X)^{-\f{p}{p-2}}\|\na u\|_{L^p}\|X_\la\|_{L^\infty}+\|\na u\|_{L^\infty}\|\na X_\la\|_{L^p}+\|\na u_{X_\la}\|_{L^p}\Bigr)
+\|\na u\|_{L^p}.
\end{split}
\eeno
While we deduce from Proposition \ref{S9prop1} that for any $s\in ]2/p,1[,$
\beno
\begin{split}
\|\na\frak{A}\|_{L^\infty}\lesssim &
\frak{B}(\r)\bigl(\|\na u\|_{L^p}+\|\na u\|_{L^\infty}\bigr)+\f1{I(X)}\sup_{\la\in\Lambda}\Bigl\{\bigl(\frak{B}(\r)\|X_\la\|_{L^\infty}\|\na u\|_{L^p}\bigr)^{1-\f2p}\\
&\qquad\times \Bigl(\frak{B}(\r)\bigl(\|\na u\|_{L^\infty}\|\na X_\la\|_{L^p}+\|\na u_{X_\la}\|_{L^p}\bigr)+\|\p_{X_\la}\mu(\r)\|_{L^\infty}\|\na u\|_{L^p}\Bigr)^{\f2p}\\
&\qquad+\frak{B}(\r)\bigl(\|X_\la\|_{L^\infty}\|\na u\|_{L^p}\bigr)^{1-\f2{ps}}\bigl(\|\na u\|_{L^\infty}\|\na X_\la\|_{L^p}\bigr)^{\f2{ps}}\Bigr\}.
\end{split}
\eeno
Applying Young's inequality yields
\beno
\begin{split}
\|\na\frak{A}\|_{L^\infty}\lesssim & \|\na u\|_{L^p}+\frak{B}(\r)\Bigl(\|\na u\|_{L^\infty}\\
&+\sup_{\la\in\La}\bigl(C(s,p,X)\|\na u\|_{L^p}\|X_\la\|_{L^\infty}+\|\na u\|_{L^\infty}\|\na X_\la\|_{L^p}+\|\na u_{X_\la}\|_{L^p}\bigr)\Bigr).
\end{split}
\eeno
 Here and in what follows, we always denote
\beq
\label{S6eq9a}
C(s,p,X)\eqdefa I(X)^{-\f{p}{p-2}}+I(X)^{-\f{ps}{ps-2}}. \eeq
The same estimate holds for
\beno
\begin{split}
\f1{I(X)}\sup_{\la\in\Lambda}&\bigl(\frak{B}(\r)\|X_\la\|_{L^\infty}\|\na u\|_{L^p}\bigr)^{1-\f2p}
\bigl(\|\na X_\la\|_{L^p}\|\na\frak{A}\|_{L^\infty}\bigr)^{\f2p}.
\end{split}
\eeno
Inserting the above inequalities into \eqref{S6eq9} gives rise to
\beq\label{S6eq10}
\begin{split}
\|\na\PP\frak{A}\|_{L^\infty}\leq  C\Bigl\{\|\na u\|_{L^p}+\frak{B}(\r)\Bigl(&\|\na u\|_{L^\infty}+\sup_{\la\in\La}\bigl(C(s,p,X)\|\na u\|_{L^p}\|X_\la\|_{L^\infty}\\
&\qquad+\|\na u\|_{L^\infty}\|\na X_\la\|_{L^p}+\|\na u_{X_\la}\|_{L^p}\bigr)\Bigr)\Bigr\}.
\end{split}
\eeq
By virtue of \eqref{S6eq7}, \eqref{S6eq8} and \eqref{S6eq10}, we achieve
\beno
\begin{split}
\|\na u\|_{L^\infty}\leq & \|S_0\na u\|_{L^\infty}+\|(I-S_0)\na u\|_{L^\infty}\\
\leq &C\Bigl\{\|\na u\|_{L^p}+\|D_tu\|_{L^2}^{\f34}\|\na D_tu\|_{L^2}^{\f14}+\frak{B}(\r)\Bigl(\|\na u\|_{L^\infty}\\
&+\sup_{\la\in\La}\bigl(C(s,p,X)\|\na u\|_{L^p}\|X_\la\|_{L^\infty}+\|\na u\|_{L^\infty}\|\na X_\la\|_{L^p}+\|\na u_{X_\la}\|_{L^p}\bigr)\Bigr)\Bigr\}.
\end{split}
\eeno
 Whereas for some $\frak{M},$ which we shall fix later on, we define $T^\star$ by \eqref{S9eq11as}. Then  under the smallness condition \eqref{S1eq2}, for any  $t\in [0,T^\star],$ we get, by inserting the estimate
\eqref{S9eq11} to the above inequality, that
\beq\label{S6eq11}
\begin{split}
\|\na & u(t)\|_{L^\infty}
\leq C\Bigl\{\|\na u\|_{L^p}+\|D_tu\|_{L^2}^{\f34}\|\na D_tu\|_{L^2}^{\f14}\\
&+\frak{B}(\r)\Bigl(\bigl(\bigl(1+C(s,p,X)\bigr)\|\na u\|_{L^p}+\|D_tu\|_{L^2}^{\f2p}\|\na D_tu\|_{L^2}^{1-\f2p}\bigr)\sup_{\la\in\La}\|X_\la\|_{L^\infty}\\
&+\|\na u\|_{L^p}\bigl(1+C(s,p,X)\bigr)+\bigl(\|\na u\|_{L^\infty}+\|\na u\|_{L^p}+\|D_tu\|_{L^2}\bigr)\sup_{\la\in\La}\|\na X_\la\|_{L^p}\Bigr)\Bigr\}.
\end{split}
\eeq
Yet it follows from the estimates following \eqref{S6eq3} that there exists a positive constant $\frak{N},$
which depends on $C_0$ and $C_{\mu,X},$ so that
\beno
C\int_0^{T^\ast}\Bigl(\na u\|_{L^p}+\|D_tu\|_{L^2}+\|D_tu\|_{L^2}^{\f34}\|\na D_tu\|_{L^2}^{\f14}+\|D_tu\|_{L^2}^{\f2p}\|\na D_tu\|_{L^2}^{1-\f2p}\Bigr)\,dt'
\leq \f{\frak{N}}2.
\eeno
Then
we deduce from  \eqref{S6eq1}, \eqref{S6eq2}, \eqref{ener4a} and \eqref{S6eq11} that
\beq\label{S6eq12}
\begin{split}
\|\na u\|_{L^1_t(L^\infty)}\leq & \f{\frak{N}}2\left(1+\frak{K}\frak{B}(\r_0)\exp\bigl(C\|\na u\|_{L^1_t(L^\infty)}\bigr)\right)\with\\
\frak{K}\eqdefa &C_0\Bigl(1+\bigl(1+C(s,p,X_0)\bigr)\sup_{\la\in\La}\|X_\la(0)\|_{L^\infty}+\sup_{\la\in\La}
\|\na X_\la(0)\|_{L^p}\Bigr).
\end{split}
\eeq
In particular, if $\e_0$ in \eqref{S1eq2} is so small that
\beq \label{S6eq13}
\e_0\leq \frac12\frak{K}^{-1}\exp\bigl(-C\frak{N}\bigr),
\eeq
we conclude that
\beq \label{S6eq16}
\begin{split}
\|\na u(t)\|_{L^1_t(L^\infty)} \leq \f{3\frak{N}}4\quad\mbox{for any}\ \ t\in [0,T^\star[.
\end{split}
\eeq
We now take
\beq \label{S6eq14}
\frak{M}\eqdefa 2C_0\Bigl(1+\bigl(1+C(s,p,X_0)\bigr)\sup_{\la\in\La}\|X_\la(0)\|_{L^\infty}+\sup_{\la\in\La}
\|\na X_\la(0)\|_{L^p}\Bigr)\exp\left(C\frak{N}\right).
\eeq
We claim that $T^\star=T^\ast.$ Otherwise, if $T^\star<T^\ast,$
we deduce from \eqref{S6eq2} that
\beq \label{S6eq17}
\|\na X_\la\|_{L^\infty_t(L^p)}\leq \frak{M}/2\quad\mbox{for any}\quad t\in [0,T^\star[,
\eeq
under the assumption that
\beq \label{S6eq15}
\e_0\leq \min\Bigl\{\f{\e_1 I(X_0)}{\frak{M}}\exp\bigl(-\frak{N}\bigr), \frac12\frak{K}^{-1}\exp\bigl(-C\frak{N}\bigr)\Bigr).
\eeq
This shows that under the assumption \eqref{S6eq15}, \eqref{S6eq16} and \eqref{S6eq17} hold on $[0, T^\star[,$ which contradicts with the definition of $T^\star$
given by \eqref{S9eq11as}.  This in turn shows that
$T^\star=T^\ast.$ We complete the proof of the Proposition \ref{S9lem3}.
\end{proof}

\setcounter{equation}{0}
\section{The energy estimate of $u_X$}\label{Sect7}

\begin{lem}\label{energyprop}
{\sl Let \beq \label{S3eq12}
\cY(t)\eqdefa \f12\int_{\R^2}\r|u_X|^2\,dx+\int_{\R^2}\na(-\D)^{-1}\dive(\r u_X) | u\cdot\na X\,dx.
\eeq
 Then under the assumptions of Proposition  \ref{S3col1}, for any $p\in ]2,\infty[$ and any $t<T^\ast, $ there holds
\beq\label{S3eq13}
\begin{split}
\f{d}{dt}\cY(t)+\|\na u_X\|_{L^2}^2\leq &C\bigl(\|\na u\|_{L^\infty}+\|D_tu\|_{L^2}+\w{t}^{-1_+}\bigr)\cY(t)\\
&+C\Bigl(\w{t}^{1_+}\|D_tu\|_{L^2}^2+\|\na u\|_{L^2}^2+\bigl(\|D_tu\|_{L^2}^{1+\f2p}+\|\na u\|_{L^{\f{2p}{p-2}}}^2\\
&\qquad+(\|\na u\|_{L^\infty}+\|D_tu\|_{L^2}+\w{t}^{-1_+})\|u\|_{L^{\f{2p}{p-2}}}^2\bigr)\|\na X\|_{L^p}^2\Bigr).\end{split}
\eeq
 }
\end{lem}

\begin{proof} We first get, by taking $L^2$ inner product of $u_X$ with the momentum equation of \eqref{XLINS}, that
\beq \label{S3eq9}
\f12\f{d}{dt}\int_{\R^2}\r|u_X(t)|^2\,dx+2\int_{\R^2}\mu(\r)|\cM(u_X)|^2\,dx=\int_{\R^2}(G-\na \Pi_X) | u_X\,dx.
\eeq
To handle the term $\Pi_X$, we  take the divergence operator to the momentum equation
of \eqref{XLINS} to get
\beq \label{S3eq10}
\Pi_X=(-\D)^{-1}\dive\bigl(\r D_tu_X-\dive(2\mu(\r)\cM(u_X)\bigr)-\dive(-\D)^{-1}G.
\eeq
Then we have
\beq \label{S3eq10a}
\begin{split}
\int_{\R^2}(G-\na \Pi_X) | u_X\,dx=&\int_{\R^2}\PP G | u_X\,dx\\
&-\int_{\R^2}\na(-\D)^{-1}\dive\bigl(\r D_t u_X-\dive(2\mu(\r)\cM(u_X))\bigr) | u_X\,dx.
\end{split}
\eeq
Next let us handle term by term above.  By using integration by parts and $D_t\r=0,$ we find
\beq\label{S3eq11}
\begin{split}
\int_{\R^2}&\na(-\D)^{-1}\dive(\r D_t u_X) | u_X\,dx=\f{d}{dt}\int_{\R^2}\na(-\D)^{-1}\dive(\r u_X) | u_X\,dx\\
&+\int_{\R^2}[\na(-\D)^{-1}\dive; u\cdot\na](\r u_X) | u_X\,dx-\int_{\R^2}\na(-\D)^{-1}\dive(\r u_X) | D_tu_X\,dx.
\end{split}
\eeq
Due to $\dive u_X=\dive(u\cdot\na X),$ one has
\beno
\begin{split}
\int_{\R^2}\na(-\D)^{-1}\dive(\r u_X) | u_X\,dx=&-\int_{\R^2}(-\D)^{-1}\dive(\r u_X) | \dive u_X\,dx\\
= &\int_{\R^2}\na(-\D)^{-1}\dive(\r u_X) | u\cdot\na X\,dx.
\end{split}
\eeno
While it follows from commutator's estimate that
\beno
\bigl|\int_{\R^2}[\na(-\D)^{-1}\dive; u\cdot\na](\r u_X) | u_X\,dx\bigr|\leq C\|\na u\|_{L^\infty}\|u_X\|_{L^2}^2.
\eeno
Whereas note that $\dive X=0$ and $\dive u_X=\p_iX^k\p_ku^i,$ we deduce   from \eqref{Xtrans} that
\beq \label{S4eq2}
\begin{split}
\dive D_tu_X=&\p_t\dive u_X+\dive(u\cdot\na u_X)\\
=&D_t(\p_iX^k\p_ku^i)+\p_iu\cdot\na u_X^i\\
=&[D_t; \p_i]X^k\p_ku^i+\p_iD_tX^k\p_ku^i+\p_iX^kD_t\p_ku^i+\p_iu\cdot\na u_X^i\\
=&-\p_iu^\ell\p_\ell X^k\p_ku^i(=0)+\p_iu^k_X\p_ku^i+\p_iX^kD_t\p_ku^i+\p_iu\cdot\na u_X^i\\
=&2\p_iu_X^k\p_ku^i+\p_iX^k\p_kD_tu^i,
\end{split}
\eeq
so that, by using integration by parts, we achieve
\beno
\begin{split}
-\int_{\R^2}\na(-\D)^{-1}\dive(\r u_X) | D_tu_X\,dx=&\int_{\R^2}(-\D)^{-1}\dive(\r u_X) | \dive D_tu_X\,dx\\
=&-\int_{\R^2}\na(-\D)^{-1}\dive(\r u_X) | (2u_X\cdot\na u+D_tu\cdot\na X)\,dx,
\end{split}
\eeno
from which and the following interpolation inequality
\beno
\|a\|_{L^q}\leq C\|a\|_{L^2}^{\f2q}\|\na a\|_{L^2}^{1-\f2q}\quad \forall \ q\in ]2,\infty[,
\eeno
 we infer for any $\eta>0,$
\beno
\begin{split}
\bigl|\int_{\R^2}&\na(-\D)^{-1}\dive(\r u_X) | D_tu_X\,dx\bigr|\\
\leq & C\bigl(\|\na u\|_{L^\infty}\|u_X\|_{L^2}^2+\|u_X\|_{L^{\f{2p}{p-2}}}\|D_tu\|_{L^2}\|\na X\|_{L^p}\bigr)\\
\leq & C\bigl(\|\na u\|_{L^\infty}\|u_X\|_{L^2}^2+\|u_X\|_{L^2}^{1-\f2p}\|\na u_X\|_{L^2}^{\f2p}\|D_tu\|_{L^2}\|\na X\|_{L^p}\bigr)\\
\leq &\eta\|\na u_X\|_{L^2}^2+C_\eta\bigl((\|\na u\|_{L^\infty}+\|D_tu\|_{L^2})\|u_X\|_{L^2}^2+\|D_tu\|_{L^2}^{1+\f2p}\|\na X\|_{L^p}^2\bigr).
\end{split}
\eeno
Substituting the above estimates into \eqref{S3eq11} yields
\beq \label{S3eq11a}
\begin{split}
\bigl|\int_{\R^2}&\na(-\D)^{-1}\dive(\r D_t u_X) | u_X\,dx-\f{d}{dt}\int_{\R^2}\na(-\D)^{-1}\dive(\r u_X) | u\cdot\na X\,dx\bigr|\\
&\leq \eta\|\na u_X\|_{L^2}^2+C_\eta\bigl((\|\na u\|_{L^\infty}+\|D_tu\|_{L^2})\|u_X\|_{L^2}^2+\|D_tu\|_{L^2}^{1+\f2p}\|\na X\|_{L^p}^2\bigr).
\end{split}
\eeq

To handle the last term in \eqref{S3eq10a}, by using integration by parts and $\dive u_X=\mbox{tr}(\na X\na u),$ we obtain
\beq \label{S3eq11b}
\begin{split}
\bigl|\int_{\R^2}&\na(-\D)^{-1}\dive\dive(2\mu(\r)\cM(u_X)) | u_X\,dx\bigr|\\
=&\bigl|\int_{\R^2}(-\D)^{-1}\dive\dive(2\mu(\r)\cM(u_X)) | \mbox{tr}(\na X\na u)\,dx\bigr|\\
\leq &C\|\mu(\r)\|_{L^\infty}\|\na u_X\|_{L^2}\|\na X\|_{L^p}\|\na u\|_{L^{\f{2p}{p-2}}}\\
\leq &\eta\|\na u_X\|_{L^2}^2+C_\eta\|\na u\|_{L^{\f{2p}{p-2}}}^2\|\na X\|_{L^p}^2.
\end{split}
\eeq

It remains to deal with the term $\int_{\R^2}\PP G | u_X\,dx.$ Indeed it follows from the transport equation of \eqref{XLINS} that
\beno
\begin{split}
\bigl|\int_{\R^2}\PP(\r_XD_tu)  | u_X\,dx\bigr|\leq& C\|\r_X\|_{L^\infty}\|D_tu\|_{L^2}\|u_X\|_{L^2}\\
\leq &C\w{t}^{-1_+}\|u_X\|_{L^2}^2+C\w{t}^{1_+}\|D_tu\|_{L^2}^2,\\
\bigl|\int_{\R^2}\PP\dive(2\p_X\mu(\r)\cM(u))  | u_X\,dx\bigr|\leq& C\|\p_X\mu(\r)\|_{L^\infty}\|\na u\|_{L^2}\|\na u_X\|_{L^2}\\
\leq &\eta\|\na u_X\|_{L^2}^2+C_\eta\|\na u\|_{L^2}^2.
\end{split}
\eeno
While in view of \eqref{b.4a} and $\dive X=0,$ we write
\beno
\begin{split}
\int_{\R^2}\PP\bigl(\na X^k\p_k\Pi\bigr) | u_X\,dx=&\int_{\R^2}\PP\bigl(\na X^k\p_k(-\D)^{-1}\dive(\r D_tu)\bigr) |  u_X\,dx\\
&+\int_{\R^2}\PP\bigl(\na X^k(-\D)^{-1}\dive\dive(2\mu(\r)\cM(u))\bigr) | \p_k u_X\,dx,
\end{split}
\eeno
so that we obtain
\beno
\begin{split}
\bigl|\int_{\R^2}\PP\bigl(\na X^k\p_k \Pi\bigr) | u_X^j\,dx\bigr|&\leq C\|\na X\|_{L^p}\bigl(\|D_tu\|_{L^2}\|u_X\|_{L^{\f{2p}{p-2}}}
+\|\na u\|_{L^{\f{2p}{p-2}}}\|\na u_X\|_{L^2}\bigr)\\
&\leq \eta\|\na u_X\|_{L^2}^2+C_\eta\Bigl(\|D_tu\|_{L^2}\|u_X\|_{L^2}^2\\
&\qquad\qquad\qquad\qquad\quad+\bigl(\|\na u\|_{L^{\f{2p}{p-2}}}^2+\|D_tu\|_{L^2}^{1+\f2p}\bigr)\|\na X\|_{L^p}^2\Bigr).
\end{split}
\eeno
Finally due to the $\dive X=0,$ we get, by using integration by parts, that
\beno
\begin{split}
\bigl|\int_{\R^2}\PP\bigl(\p_iX\cdot\na(\mu(\r)\cM_{ij}(u))\bigr)  | u_X\,dx\big|\leq &C\|\na X\|_{L^p}\|\na u\|_{L^{\f{2p}{p-2}}}\|\na u_X\|_{L^2}\\
\leq &\eta\|\na u_X\|_{L^2}^2+C_\eta\|\na u\|_{L^{\f{2p}{p-2}}}^2\|\na X\|_{L^p}^2,\\
\bigl|\int_{\R^2}\PP\p_i\bigl(\mu(\r)(\p_iX\cdot\na u+\na X\cdot\na u^i)\bigr) | u_X\,dx\bigr|\leq &C\|\na X\|_{L^p}\|\na u\|_{L^{\f{2p}{p-2}}}\|\na u_X\|_{L^2}\\
\leq &\eta\|\na u_X\|_{L^2}^2+C_\eta\|\na u\|_{L^{\f{2p}{p-2}}}^2\|\na X\|_{L^p}^2.
\end{split}
\eeno
As a result, it comes out
\beq \label{S3eq11c}
\begin{split}
\bigl|\int_{\R^2}\PP G |& u_X\,dx\bigr|\leq 4\eta \|\na u_X\|_{L^2}^2+ C_\eta\bigl(\|D_tu\|_{L^2}+\w{t}^{-1_+}\bigr)\|u_X\|_{L^2}^2
\\
&+C_\eta\Bigl(\w{t}^{1_+}\|D_tu\|_{L^2}^2+\|\na u\|_{L^2}^2+\bigl(\|D_tu\|_{L^2}^{1+\f2p}+\|\na u\|_{L^{\f{2p}{p-2}}}^2\bigr)\|\na X\|_{L^p}^2\Bigr).\end{split}
\eeq

Notice that
\beno
\bigl|\int_{\R^2}\na(-\D)^{-1}\dive(\r u_X) | u\cdot\na X\,dx\bigr|\leq C\|u_X\|_{L^2}\|u\|_{L^{\f{2p}{p-2}}}\|\na X\|_{L^p},
\eeno
we have
\beno
\cY(t)\geq c\|u_X\|_{L^2}^2-C\|u\|_{L^{\f{2p}{p-2}}}^2\|\na X\|_{L^p}^2.
\eeno
By inserting the  Estimates \eqref{S3eq11a}, \eqref{S3eq11b} and \eqref{S3eq11c}  into \eqref{S3eq10a} and then substituting the resulting
inequality into \eqref{S3eq9} and taking $\eta$ sufficiently small, we achieve \eqref{S3eq13}. This completes the proof of the lemma.
\end{proof}

\begin{proof}[Proof of Proposition \ref{S3col1}] By multiplying \eqref{S3eq13} with $\w{t}^{2\d_-}$, we find
\beq\label{S3eq15}
\begin{split}
\f{d}{dt}&\bigl(\w{t}^{2\d_-}\cY(t)\bigr)+\|\w{t}^{\d_-}\na u_X\|_{L^2}^2\leq  C\bigl(\|\na u\|_{L^\infty}+\|D_tu\|_{L^2}+\w{t}^{-1_+}\bigr)\w{t}^{2\d_-}\cY(t)\\
&+2\d\w{t}^{2\d_--1}\cY(t)
+C\w{t}^{2\d_-}\Bigl(\w{t}^{1_+}\|D_tu\|_{L^2}^2+\|\na u\|_{L^2}^2\\
&+\bigl((\|\na u\|_{L^\infty}+\|D_tu\|_{L^2}+\w{t}^{-1_+})\|u\|_{L^{\f{2p}{p-2}}}^2+\|D_tu\|_{L^2}^{1+\f2p}+\|\na u\|_{L^{\f{2p}{p-2}}}^2\bigr)\|\na X\|_{L^p}^2\Bigr).\end{split}
\eeq
It is easy to observe from Proposition \ref{S2prop2} that
\beno
\begin{split}
\int_0^t\w{t'}^{2\d_--1}\cY(t')\,dt'\leq &\int_0^t\w{t'}^{2\d_--1}\bigl(\|X\|_{L^\infty}^2\|\na u\|_{L^2}^2+\|X\|_{L^\infty}\|\na X\|_{L^p}\|u\|_{L^{\f{2p}{p-2}}}\|\na u\|_{L^2}\bigr)\,dt'\\
\leq &C_0\bigl(\|X\|_{L^\infty_t(L^\infty)}^2+\|\na X\|_{L^\infty(L^p)}^2\bigr),
\end{split}
\eeno
and due to  $p<2\left(1+1/\d\right),$ one has
\beno
\begin{split}
\int_0^t\w{t'}^{2\d_-}\|D_tu(t')\|_{L^2}^{1+\f2p}\,dt'\leq \Bigl(\int_0^t\w{t'}^{2\d_--\f{p+2}{p-2}}\,dt'\Bigr)^{\f{p-2}{2p}}
\bigl\|\w{t}^{\left(\f12+\d\right)_-}D_tu\bigr\|_{L^2_t(L^2)}^{\f{p+2}p}\leq C_0.
\end{split}
\eeno
Similarly it follows from \eqref{S2eq1} and
Proposition \ref{S2prop1} that
\beno
\int_0^t\w{t'}^{2\d_-}\|\na u(t')\|_{L^{\f{2p}{p-2}}}^2\,dt'\leq C\|\w{t}^{\d}\na u\|_{L^2_t(L^2)}^{2\left(1-\f2p\right)}\|\w{t}^{\d_-} D_tu\|_{L^2_t(L^2)}^{\f4p}\leq C_0.
\eeno
Hence thanks to Proposition \ref{S2prop2} and the fact that $|\cY(0)|\leq \bigl(\|X_{0}\|_{L^\infty}^2+\|\na X_0\|_{L^p}^2\bigr)\|u_0\|_{H^1}^2,$  by applying Gronwall 's inequality to \eqref{S3eq15}, we achieve \eqref{S3eq14}.
\end{proof}

\setcounter{equation}{0}
\section{Energy estimate of $\na u_X$}\label{Sect8}

\begin{lem}\label{S4lem1}
{\sl Under the assumptions of Lemma \ref{energyprop}, for any $t<T^\ast,$ one has
\beq \label{S4eq0}
\begin{split}
\f{d}{dt}&\int_{\R^2}\mu(\r)|\cM(u_X)|^2\,dx+\int_{\R^2}\r|D_tu_X|^2\,dx\\
= &\f{d}{dt}\int_{\R^2}\na X^k(-\D)^{-1}\dive\dive\left(2\mu(\r)\cM(u)\right) | \p_k \PP u_X\,dx\\
&-\f{d}{dt}\Bigl(\int_{\R^2}\p_X\mu(\r)\cM(u) :\PP \na  u_X\,dx+\int_{\R^2}\mu(\r)\p_iX^k\cM_{ij}(u) |  \PP^j \p_ku_X^j\,dx\Bigr)\\
&-\f{d}{dt}\int_{\R^2}\mu(\r)\left(\p_iX\cdot\na u+\na X\cdot\na u^i\right)
| \p_i\PP u_X\,dx+{\rm r},
\end{split}
\eeq where $\PP=(\PP^1,\PP^2)=Id+\na(-\D)^{-1}\dive$ denotes the leray projection operator to the solenoidal vector field spaces, and ${\rm r}$ satisfies
\beno
\begin{split}
|{\rm r}|\leq &C\Bigl(\bigl(\|\na u\|_{L^\infty}+\|u\|_{L^\infty}^2\bigr)\|\na u_X\|_{L^2}^2+\|D_tu\|_{L^{\f{2p}{p-2}}}^2\|\na X\|_{L^p}^2+\|D_tu\|_{L^2}^2\Bigr)\\
&+C\bigl((\|\na D_tu\|_{L^2}+\|\na u\|_{L^4}^2)\|\na X\|_{L^p}
+\|\na u\|_{L^p}\|\na u_X\|_{L^2}\bigr)\|\na u_X\|_{L^{\f{2p}{p-2}}}\\
&+C\Bigl(\bigl(\|\na u\|_{L^{\f{4p}{p-2}}}^2+\|\na u\|_{L^\infty}\|\na u\|_{L^{\f{2p}{p-2}}}\bigr)\|\na X\|_{L^p}
\\
&\qquad\qquad\qquad\qquad+\|\na u\|_{L^4}^2+\|\na u\|_{L^2}\|\na u\|_{L^\infty}+\|\na D_tu\|_{L^2}\Bigr)\|\na u_X\|_{L^2}.
\end{split} \eeno
}
\end{lem}

\begin{proof}
In view of \eqref{XLINS}, we get, by applying Lemma \ref{S2lem1}, that
\beno \label{S4eq1}
\begin{split}
\f{d}{dt}&\int_{\R^2}\mu(\r)|\cM(u_X)|^2\,dx+\int_{\R^2}\r|D_tu_X|^2\,dx\\
&\qquad\leq C\|\na u\|_{L^\infty}\|\na u_X\|_{L^2}^2+\int_{\R^2}(G-\na\Pi_X) | D_tu_X\,dx,
\end{split}
\eeno
from which and \eqref{S3eq10}, we infer
\beq \label{S4eq10}
\begin{split}
\f{d}{dt}\int_{\R^2}\mu(\r)|\cM(u_X)|^2\,dx&+\int_{\R^2}\r|D_tu_X|^2\,dx\leq C\|\na u\|_{L^\infty}\|\na u_X\|_{L^2}^2\\
&+\bigl|\int_{\R^2}\na(-\D)^{-1}\dive\dive(2\mu(\r)\cM(u_X))|  D_tu_X\,dx\bigr|\\
&+\bigl|\int_{\R^2}\na(-\D)^{-1}\dive(\r D_tu_X) |  D_tu_X\,dx\bigr|+\bigl|\int_{\R^2}\PP G | D_tu_X\,dx\bigr|.
\end{split}
\eeq
Thanks to \eqref{S4eq2}, we get, by using integration by parts, that
\beno
\begin{split}
\bigl|\int_{\R^2}&\na(-\D)^{-1}\dive(\r D_tu_X) |  D_tu_X\,dx\bigr|\\
=&\bigl|\int_{\R^2}(-\D)^{-1}\dive(\r D_tu_X) |  \dive D_tu_X\,dx\bigr|\\
=&\bigl|\int_{\R^2}\na(-\D)^{-1}\dive(\r D_tu_X) |  \bigl(2u\cdot\na u_X+D_tu\cdot\na X\bigr)\,dx\bigr|\\
\leq &C\|\r D_tu_X\|_{L^2}\bigl(\|u\|_{L^\infty}\|\na u_X\|_{L^2}+\|D_tu\|_{L^{\f{2p}{p-2}}}\|\na X\|_{L^p}\bigr),
\end{split}
\eeno
Similarly, one has
\beno
\begin{split}
&\bigl|\int_{\R^2}\na(-\D)^{-1}\dive\dive(2\mu(\r)\cM(u_X))|  D_tu_X\,dx\bigr|\\
&=\bigl|\int_{\R^2}(-\D)^{-1}\dive\dive(2\mu(\r)\cM(u_X))|  \bigl(2\mbox{tr}(\na u_X\na u)+\mbox{tr}(\na X\na D_tu)\bigr)\,dx\bigr|\\
&\leq C\bigl(\|\na u\|_{L^\infty}\|\na u_X\|_{L^2}^2+\|\na u_X\|_{L^{\f{2p}{p-2}}}\|\na X\|_{L^p}\|\na D_tu\|_{L^2}\bigr).
\end{split}
\eeno

It remains to deal with the last term in \eqref{S4eq10}, which we shall handle term by term below.

\no$\bullet$ \underline{The estimate of $\bigl(\PP (\na X^k\p_k\Pi) | D_t u_X \bigr)$}

It follows from \eqref{b.4a} that
\beno
\begin{split}
\int_{\R^2}\PP(\na X^k\p_k\Pi) | D_t u_X\,dx=&\int_{\R^2}\na X^k\p_k(-\D)^{-1}\dive(\r D_tu)  | \PP D_t u_X\,dx\\
&-\int_{\R^2}\na X^k\p_k(-\D)^{-1}\dive\dive\left(2\mu(\r)\cM(u)\right) | \PP D_t u_X\,dx.
\end{split}
\eeno
It is easy to observe that
\beno
\bigl|\int_{\R^2}\na X^k\p_k(-\D)^{-1}\dive(\r D_tu)  | \PP D_t u_X\,dx\bigr|\leq C\|\na X\|_{L^p}\|D_tu\|_{L^{\f{2p}{p-2}}}\|D_tu_X\|_{L^2}.
\eeno
By using integration by parts, one has
\beno
\begin{split}
&-\int_{\R^2}\na X^k\p_k(-\D)^{-1}\dive\dive\left(2\mu(\r)\cM(u)\right) | \PP D_t u_X\,dx\\
&=\int_{\R^2}\na X^k(-\D)^{-1}\dive\dive\left(2\mu(\r)\cM(u)\right) |\bigl(\PP(\p_ku\cdot\na u_X)+ [\PP;u\cdot\na]\p_ku_X+D_t \PP\p_ku_X\bigr)\,dx.
\end{split}
\eeno
We first deduce that
\beno
\bigl|\int_{\R^2}\na X^k(-\D)^{-1}\dive\dive\left(2\mu(\r)\cM(u)\right) | \PP(\p_ku\cdot\na u_X)\,dx\bigr|
\leq C\|\na u\|_{L^{\f{4p}{p-2}}}^2\|\na X\|_{L^p}\|\na u_X\|_{L^2}.
\eeno
Applying classical commutator's estimate yields
\beno
\begin{split}
\bigl|\int_{\R^2}\na X^k(-\D)^{-1}\dive\dive\left(2\mu(\r)\cM(u)\right)& | [\PP; u\cdot\na]\p_ku_X \,dx\bigr|\\
&\leq C\|\na X\|_{L^p}\|\na u\|_{L^{\f{2p}{p-2}}}\|\na u\|_{L^\infty}\|\na u_X\|_{L^2}.
\end{split}
\eeno
While we get, by using integration by parts, that
\beno
\begin{split}
\int_{\R^2}&\na X^k(-\D)^{-1}\dive\dive\left(2\mu(\r)\cM(u)\right) | D_t\p_k \PP u_X\,dx\\
=&\f{d}{dt}\int_{\R^2}\na X^k(-\D)^{-1}\dive\dive\left(2\mu(\r)\cM(u)\right) | \p_k \PP u_X\,dx\\
&-\int_{\R^2}\bigl(D_t\na X^k(-\D)^{-1}\dive\dive\left(2\mu(\r)\cM(u)\right)\\
&\qquad\quad+ \na X^kD_t(-\D)^{-1}\dive\dive\left(2\mu(\r)\cM(u)\right)\bigr)| \p_k \PP u_X\,dx.
\end{split}
\eeno
By virtue of \eqref{Xtrans}, we find
\beno
\begin{split}
\int_{\R^2}&D_t\na X^k(-\D)^{-1}\dive\dive\left(2\mu(\r)\cM(u)\right)| \p_k \PP u_X\,dx\\
=&\int_{\R^2}\bigl(-\na u\cdot\na X^k(-\D)^{-1}\dive\dive\left(2\mu(\r)\cM(u)\right)\\
&\qquad\qquad+\na u_X^k(-\D)^{-1}\dive\dive\left(2\mu(\r)\cM(u)\right)\bigr) | \p_k \PP u_X\,dx,
\end{split}
\eeno
which gives rise to
$$\longformule{
\bigl|\int_{\R^2}D_t\na X^k(-\D)^{-1}\dive\dive\left(2\mu(\r)\cM(u)\right)| \p_k \PP u_X\,dx\bigr|}{{}
\leq C\bigl(\|\na u\|_{L^{\f{4p}{p-2}}}^2\|\na X\|_{L^p}
+\|\na u\|_{L^p}\|\na u_X\|_{L^{\f{2p}{p-2}}}\bigr)\|\na u_X\|_{L^2}.}
$$
Whereas by using a commutator's argument, we write
\beno
\begin{split}
\int_{\R^2}&\na X^kD_t(-\D)^{-1}\dive\dive\left(2\mu(\r)\cM(u)\right)\bigr)| \p_k \PP u_X\,dx\\
=&\int_{\R^2}\bigl(\na X^k[u\cdot\na; (-\D)^{-1}\dive\dive]\left(2\mu(\r)\cM(u)\right)\\
&\qquad\quad +\na X^k(-\D)^{-1}\dive\dive\left(2\mu(\r)D_t\cM(u)\right)\bigr) | \p_k \PP u_X\,dx,
\end{split}
\eeno
which together with the classical commutator estimate implies
\beno
\begin{split}
\bigl|\int_{\R^2}&\na X^kD_t(-\D)^{-1}\dive\dive\left(2\mu(\r)\cM(u)\right)\bigr)| \p_k \PP u_X\,dx\bigr|\\
\leq &C\|\na X\|_{L^p}\bigl(\|\na u\|_{L^\infty}\|\na u\|_{L^{\f{2p}{p-2}}}\|\na u_X\|_{L^2}+
(\|\na D_tu\|_{L^2}+\|\na u\|_{L^4}^2)\|\na u_X\|_{L^{\f{2p}{p-2}}}\bigr).
\end{split}
\eeno
By summarizing the above estimates, we obtain
\beq\label{S4eq6}
\int_{\R^2}\PP(\na X^k\p_k\Pi) | D_t u_X\,dx=\f{d}{dt}\int_{\R^2}\na X^k(-\D)^{-1}\dive\dive\left(2\mu(\r)\cM(u)\right) | \p_k \PP u_X\,dx+r_1
\eeq
with $r_1$ satisfying
\beno
\begin{split}
|r_1|\leq C\|\na X\|_{L^p}\Bigl(&\|D_tu\|_{L^{\f{2p}{p-2}}}\|D_tu_X\|_{L^2}+
\bigl(\|\na u\|_{L^{\f{4p}{p-2}}}^2+\|\na u\|_{L^\infty}\|\na u\|_{L^{\f{2p}{p-2}}}\bigr)\|\na u_X\|_{L^2}\\
&+
(\|\na D_tu\|_{L^2}+\|\na u\|_{L^4}^2)\|\na u_X\|_{L^{\f{2p}{p-2}}}\Bigr)+C\|\na u\|_{L^p}\|\na u_X\|_{L^{\f{2p}{p-2}}}\|\na u_X\|_{L^2}.
\end{split}
\eeno

\no$\bullet$ \underline{The estimate of $\bigl(\PP \dive\left(\p_X\mu(\r)\cM(u)\right) | D_t u_X \bigr)$}

By using integration by parts, we write
\beno
\begin{split}
\int_{\R^2}&\PP\dive\left(\p_X\mu(\r)\cM(u)\right) | D_t u_X\,dx=-\int_{\R^2}\p_X\mu(\r)\cM(u) : \na\PP D_t u_X\,dx\\
=&-\int_{\R^2}\p_X\mu(\r)\cM(u): \bigl(\PP(\na u\cdot\na u_X)+ [\PP;u\cdot\na]\na u_X+D_t\PP\na  u_X\bigr)\,dx,
\end{split}
\eeno
and
\beno
\begin{split}
\int_{\R^2}\p_X\mu(\r)\cM(u) : D_t\PP\na  u_X\,dx=&\f{d}{dt}\int_{\R^2}\p_X\mu(\r)\cM(u) : \PP\na  u_X\,dx\\
&-\int_{\R^2}\p_X\mu(\r)D_t\cM(u) : \na  u_X\,dx,
\end{split}
\eeno
This leads to
\beq\label{S4eq7}
\int_{\R^2}\PP\dive\left(\p_X\mu(\r)\cM(u)\right) | D_t u_X\,dx=-\f{d}{dt}\int_{\R^2}\p_X\mu(\r)\cM(u) :\PP \na  u_X\,dx+r_2,
\eeq
with $r_2$ satisfying
\beno
|r_2|\leq C\bigl(\|\na u\|_{L^4}^2+\|\na u\|_{L^2}\|\na u\|_{L^\infty}+\|\na D_tu\|_{L^2}\bigr)\|\na u_X\|_{L^2}.
\eeno

\no$\bullet$ \underline{The estimate of $\bigl(\PP \bigl(\p_i\left(\mu(\r)\bigl(\p_iX\cdot\na u+\na X\cdot\na u^i\bigr)\right)\bigr) | D_t u_X\bigr)$}

Again by using integration by parts, we write
\beno
\begin{split}
\int_{\R^2}&\PP\p_i\left(\mu(\r)\bigl(\p_iX\cdot\na u+\na X\cdot\na u^i\bigr)\right)
| D_t u_X\,dx\\
=&-\int_{\R^2}\mu(\r)\bigl(\p_iX\cdot\na u+\na X\cdot\na u^i\bigr)
| \PP\p_iD_t u_X\,dx\\
=&-\int_{\R^2}\mu(\r)\bigl(\p_iX\cdot\na u+\na X\cdot\na u^i\bigr)
| \bigl(\PP(\p_iu\cdot\na u_X)+[\PP;u\cdot\na]\p_iu_X+D_t\p_i \PP u_X\bigr)\,dx,
\end{split}
\eeno
and
\beno
\begin{split}
\int_{\R^2}&\mu(\r)\left(\p_iX\cdot\na u+\na X\cdot\na u^i\right)
| D_t\p_i\PP u_X\,dx\\
=&\f{d}{dt}\int_{\R^2}\mu(\r)\left(\p_iX\cdot\na u+\na X\cdot\na u^i\right)
| \p_i \PP u_X\,dx\\
&-\int_{\R^2}\mu(\r)\bigl(D_t\p_iX\cdot\na u+D_t\na X\cdot\na u^i\bigr)
| \p_i\PP u_X\,dx\\
&-\int_{\R^2}\mu(\r)\bigl(\p_iX\cdot D_t\na u+\na X\cdot D_t \na u^i\bigr)
| \p_i\PP u_X\,dx.
\end{split}
\eeno
Note from \eqref{Xtrans} that $D_t\p_iX=-\p_iu\cdot\na X+\p_iu_X$ and $ D_t\na u=-\na u\cdot\na u+\na D_tu ,$ we deduce that
\beq \label{S4eq8}
\begin{split}
\int_{\R^2}&\PP\p_i\left(\mu(\r)\bigl(\p_iX\cdot\na u+\na X\cdot\na u^i\bigr)\right)
| D_t u_X\,dx\\
=&-\f{d}{dt}\int_{\R^2}\mu(\r)\left(\p_iX\cdot\na u+\na X\cdot\na u^i\right)
| \p_i\PP u_X\,dx+r_3, \end{split}
\eeq
with $r_3$ satisfying
\beno
\begin{split}
|r_3|\leq C\Bigl(&\bigl(\|\na X\|_{L^p}(\|\na u\|_{L^{\f{4p}{p-2}}}^2+\|\na u\|_{L^{\f{2p}{p-2}}}\|\na u\|_{L^\infty})\\
&+\|\na u\|_{L^\infty} \|\na u_X\|_{L^2}\bigr)\|\na u_X\|_{L^2}+\|\na X\|_{L^p}\|\na D_tu\|_{L^2}\|\na u_X\|_{L^{\f{2p}{p-2}}}\Bigr).
\end{split}
\eeno

\no$\bullet$ \underline{The estimate of $\bigl(\PP^j \bigl(\p_iX\cdot\na\left(\mu(\r)\cM_{ij}(u)\right)\bigr) | D_t u_X^j \bigr)$}

Along the same line to the proof of \eqref{S4eq8}, we write
\beno
\begin{split}
&\int_{\R^2}\PP^j \bigl(\p_iX\cdot\na\left(\mu(\r)\cM_{ij}(u)\right)\bigr) | D_t u_X^j\,dx\\
&=-\int_{\R^2}\mu(\r)\p_iX^k\cM_{ij}(u) |\bigl(\PP^j(\p_ku\cdot\na u_X^j)+[\PP^j;u\cdot\na]\p_ku^j_X+ D_t\PP^j \p_ku_X^j\bigr)\,dx,
\end{split}
\eeno
and
\beno
\begin{split}
\int_{\R^2}&\mu(\r)\p_iX^k\cM_{ij}(u) |  D_t\PP^j \p_ku_X^j\,dx=\f{d}{dt}\int_{\R^2}\mu(\r)\p_iX^k\cM_{ij}(u) |  \PP^j \p_ku_X^j\,dx\\
&-\int_{\R^2}\mu(\r) D_t\p_iX^k\cM_{ij}(u) |  \PP^j \p_ku_X^j\,dx-\int_{\R^2}\mu(\r) \p_iX^kD_t\cM_{ij}(u) |  \PP^j \p_ku_X^j\,dx.
\end{split}
\eeno
Hence we obtain
\beq\label{S4eq8a} \begin{split}
\int_{\R^2}\PP^j \bigl(\p_iX\cdot\na&\left(\mu(\r)\cM_{ij}(u)\right)\bigr) | D_t u_X^j\,dx\\
&=-\f{d}{dt}\int_{\R^2}\mu(\r)\p_iX^k\cM_{ij}(u) |  \PP^j \p_ku_X^j\,dx
+r_4,
\end{split}
\eeq
with $r_4$ shares the same estimate as $r_3.$

Finally it is trivial to note that
\beno
\bigl|\int_{\R^2}\PP(\r_XD_tu) | D_t u_X\,dx\bigr|\leq \|\r_X\|_{L^\infty}\|D_tu\|_{L^2}\|D_tu_X\|_{L^2}.
\eeno
Inserting the above estimate and (\ref{S4eq6}-\ref{S4eq8a}) into \eqref{S4eq10} leads to \eqref{S4eq0}. This completes the proof of the lemma.
\end{proof}

\begin{lem}\label{S4lem2}
{\sl Let $p\in \bigl]4,2\left(1+1/\d\right)\bigr[.$ Then under the assumptions of Lemma \ref{energyprop} and \eqref{S1eq2},  for any $t<T^\ast,$ we have
\beq\label{S4eq3}
\begin{split}
\|\na u_X\|_{L^{\f{2p}{p-2}}}\leq &  C\Bigl(\|\na X\|_{L^p}\bigl(\|\na u\|_{L^{\f{2p}{p-4}}}+\|D_tu\|_{L^{\f{2p}{p-4}}}\bigr)\\
&+\|\na u\|_{L^{\f{2p}{p-2}}}+\bigl(\|D_tu\|_{L^2}+\|D_tu_X\|_{L^2}\bigr)^{\f2p}\\
&\times \bigl(\|\na u\|_{L^2}+\|\na u_X\|_{L^2}+\|\na X\|_{L^p}\|\na u\|_{L^{\f{2p}{p-2}}}+\|X\|_{L^\infty}\|D_t u\|_{L^2}\bigr)^{1-\f2p}\Bigr).
\end{split}
\eeq}
\end{lem}

\begin{proof} We first get, by a similar derivation of \eqref{S9eq12}, that
\beq\label{S4eq6qw}
\begin{split}
\|\na u_X\|_{L^{\f{2p}{p-2}}}\leq & C\Bigl(\|\na \hbar\|_{L^{\f{2p}{p-2}}}+\|\na u\|_{L^{\f{2p}{p-2}}}+\|\na u\|_{L^{\f{2p}{p-4}}}\|\na X\|_{L^{p}}\\
&\quad +\bigl\|\na(-\D)^{-1}\PP\bigl(\p_X(\r D_tu)-\na X^k\p_k\Pi\bigr)\bigr\|_{L^{\f{2p}{p-2}}}\Bigr).
\end{split}
\eeq
Recall that $\hbar\eqdefa -\na(-\D)^{-1}\mbox{Tr}(\na X\na u),$ one has
\beno
\|\na \hbar\|_{L^{\f{2p}{p-2}}}\leq C\|\na X\|_{L^p}\|\na u\|_{L^{\f{2p}{p-4}}}.
\eeno
Whereas note from the momentum equation of \eqref{INS} that
\beno
\begin{split}
{\Bbb P}\bigl(&\p_X(\r D_tu)-\na X^k\p_k\Pi\bigr)={\Bbb P}\p_X\dive(\mu(\r)\cM(u))\\
&={\Bbb P}\p_k[X^k;\p_i](\mu(\r)\cM_{ij}(u))+{\Bbb P}\dive\left(\p_X\mu(\r)\cM(u)+\mu(\r)\p_X\cM(u)\right),
\end{split}
\eeno
so that
\beno
\bigl\|\na(-\D)^{-1}{\Bbb P}\bigl(\p_X(\r D_tu)-\na X^k\p_k\Pi\bigr)\bigr\|_{L^2}\leq C\bigl(\|\na u\|_{L^{\f{2p}{p-2}}}\|\na X\|_{L^p}+\|\na u\|_{L^2}+\|\na u_X\|_{L^2}\bigr).
\eeno
While thanks to \eqref{b.4a}, we have
\beno
\begin{split}
\|\na(-\D)^{-1}{\Bbb P}(\na X^k\p_k\Pi)\|_{L^2}
\leq &\bigl\|\na(-\D)^{-1}{\Bbb P}\p_k\bigl( X^k\na(-\D)^{-1}\dive(\r D_tu)\bigr)\bigr\|_{L^2}\\
&+\bigl\|\na(-\D)^{-1}{\Bbb P}\p_k\bigl(\na X^k(-\D)^{-1}\dive\dive(2\mu(\r)\cM(u))\bigr)\bigr\|_{L^2}\\
\leq& C\bigl(\|X\|_{L^\infty}\|D_tu\|_{L^2}+\|\na X\|_{L^p}\|\na u\|_{L^{\f{2p}{p-2}}}\bigr).
\end{split}
\eeno
As a result, it comes out
\beq \label{S4eq3pq}
\begin{split}
\|\na(-\D)^{-1}{\Bbb P}\p_X(\r D_tu)\|_{L^2}\leq C\bigl(&\|\na u\|_{L^2}+\|\na u_X\|_{L^2}\\
&+\|\na X\|_{L^p}\|\na u\|_{L^{\f{2p}{p-2}}}+\|X\|_{L^\infty}\|D_tu\|_{L^2}\bigr).
\end{split}
\eeq
Then we get, by using 2-D interpolation inequality,  that
\beno
\begin{split}
\|\na(-\D)^{-1}{\Bbb P}\p_X(\r D_tu)\|_{L^{\f{2p}{p-2}}}\leq & C\|\na(-\D)^{-1}{\Bbb P}\p_X(\r D_tu)\|_{L^2}^{1-\f2p}\|\p_X(\r D_tu)\|_{L^2}^{\f2p}\\
\leq &C\bigl(\|D_tu\|_{L^2}+\|D_tu_X\|_{L^2}\bigr)^{\f2p}\bigl(\|\na u\|_{L^2}+\|\na u_X\|_{L^2}\\
&\qquad+\|\na X\|_{L^p}\|\na u\|_{L^{\f{2p}{p-2}}}+\|X\|_{L^\infty}\|D_tu\|_{L^2}\bigr)^{1-\f2p}.
\end{split}
\eeno
It remains to handle the estimate of $\bigl\|\na(-\D)^{-1}\PP(\na X^k\p_k\Pi)\bigr\|_{L^{\f{2p}{p-2}}}.$ Indeed it is easy to observe that
\beno
\begin{split}
\|\na&(-\D)^{-1}{\Bbb P}\na X^k\p_k(-\D)^{-1}\dive(\r D_tu)\|_{L^{\f{2p}{p-2}}}
\leq C\|\na X\|_{L^p}\|D_tu\|_{L^{\f{2p}{p-4}}}.
\end{split}
\eeno
Along the same line, we have
\beno
\begin{split}
\bigl\|\na(-\D)^{-1}\PP\p_k\bigl(\na X_k (-\D)^{-1}\dive\dive(2\mu(\r)\cM(u))\bigr)\bigr\|_{L^{\f{2p}{p-2}}}
\leq C\|\na X\|_{L^p}\|\na u\|_{L^{\f{2p}{p-4}}},
\end{split}
\eeno
so thanks to \eqref{b.4a}, we obtain
\beno
\bigl\|\na(-\D)^{-1}\PP(\na X^k\p_k\Pi)\bigr\|_{L^{\f{2p}{p-2}}}\leq C\|\na X\|_{L^p}\bigl(\|D_tu\|_{L^{\f{2p}{p-4}}}+\|\na u\|_{L^{\f{2p}{p-4}}}\bigr).
\eeno
By substituting the above inequalities into \eqref{S4eq6qw} leads to the Estimate \eqref{S4eq3}.
\end{proof}

We are now in a position to present the proof of Proposition \ref{S4prop1}.

\begin{proof}[Proof of Proposition \ref{S4prop1}] We get, by first multiplying \eqref{S4eq0} by $(t-t_0)$ and then integrating the resulting inequality over $[0,t],$
that
\beq\label{S4eq14}
\begin{split}
(t-t_0)\|\na& u_X(t)\|_{L^2}^2+\int_{t_0}^t(t-t_0)\int_{\R^2}\r|D_tu_X|^2\,dx\,dt'\leq C\Bigl(\int_{t_0}^t\|\na u_X(t')\|_{L^2}^2\,dt'\\
&+(t-t_0)\bigl(\|\na u\|_{L^{\f{2p}{p-2}}}\|\na X\|_{L^p}+\|\na u\|_{L^2}\bigr)\|\na u_X\|_{L^2}
+\int_{t_0}^t(t-t_0){\rm r}(t')\,dt'\Bigr).
\end{split}
\eeq
However, by virtue of Lemmas \ref{S4lem1} and \ref{S4lem2}, we write
\beq \label{S4eq16}
\begin{split}
{\rm r}\leq & \f14\|\sqrt{\r}D_tu_X\|_{L^2}^2+C\Bigl(\|D_tu\|_{L^2}^2+\bigl(\|\na u\|_{L^2}^2+\|\na u\|_{L^{\f{2p}{p-2}}}^2\bigr)\|\na u\|_{L^\infty}\\
&+\|\na u\|_{L^p}\|D_tu\|_{L^2}^{\f2p}\|\na u_X\|_{L^2}^{2\left(1-\f1p\right)}+f_1\|\na u_X\|_{L^2}^2+f_2\|\na u_X\|_{L^2}^{\f{p}{p-1}}\\
&+f_3\|\na u_X\|_{L^2}+f_4\|\na u_X\|_{L^2}^{\f{p-2}{p-1}}+f_5\|\na u_X\|_{L^2}^{1-\f{2}{p}}+g_3\|\na X\|_{L^p}^{\f{p}{p-1}}\\
&+g_1\bigl(\|\na X\|_{L^p}^2+\|X
\|_{L^\infty}^2\bigr)+g_2 \bigl(\|\na X\|_{L^p}^{2\left(1-\f1p\right)}+\|X
\|_{L^\infty}^{2\left(1-\f1p\right)}\bigr)+g_4\|\na X\|_{L^p}\Bigr),
\end{split}
\eeq
where
\beno
\begin{split}
f_1\eqdefa&\|\na u\|_{L^p}^{\f{p}{p-1}}+\|u\|_{L^\infty}^2+\|\na u\|_{L^\infty},\\
f_2\eqdefa &\|\na u\|_{L^p}^{\f{p}{p-1}}\bigl(\|\na u\|_{L^2}+\|\na u\|_{L^{\f{2p}{p-2}}}\|\na X\|_{L^p}
+\|D_t u\|_{L^{2}}\|X\|_{L^\infty}\bigr)^{\f{p-2}{p-1}},\\
f_3\eqdefa&\|\na D_tu\|_{L^2}+\|\na u\|_{L^4}^2+\|\na u\|_{L^p}\|\na u\|_{L^{\f{2p}{p-2}}}\\
&+\bigl(\|\na u\|_{L^{\f{4p}{p-2}}}^2
+\bigl(\|\na u\|_{L^{\f{2p}{p-4}}}+\|D_t u\|_{L^{\f{2p}{p-4}}}\bigr)\|\na u\|_{L^p}\bigr)\|\na X\|_{L^p}\\
&+\|\na u\|_{L^p}\|D_tu\|_{L^2}^{\f2p}\bigl(\|\na u\|_{L^2}+\|\na u\|_{L^{\f{2p}{p-2}}}\|\na X\|_{L^p}
+\|D_tu\|_{L^2}\|X\|_{L^\infty}\bigr)^{1-\f2p},\\
f_4\eqdefa&\bigl(\|\na D_t u\|_{L^{2}}+\|\na u\|_{L^4}^{2}\bigr)^{\f{p}{p-1}}\|\na X\|_{L^p}^{\f{p}{p-1}},\\
f_5\eqdefa&\bigl(\|\na D_t u\|_{L^{2}}+\|\na u\|_{L^4}^{2}\bigr)\|D_t u\|_{L^{2}}^{\f2p}\|\na X\|_{L^p},
\end{split}
\eeno
and
\beno
\begin{split}
g_1\eqdefa&\bigl(\|\na D_tu\|_{L^2}+\|\na u\|_{L^4}^2\bigr)\bigl(\|\na u\|_{L^{\f{2p}{p-2}}}+\|D_t u\|_{L^{\f{2p}{p-4}}}\bigr)\\
&+\bigl(\|\na D_t u\|_{L^{2}}+\|\na u\|_{L^4}^{2}\bigr)^{\f{p}{p-1}}
\bigl(\|\na u\|_{L^{\f{2p}{p-2}}}+\|D_t u\|_{L^2}\bigr)^{\f{p-2}{p-1}}+\|D_tu\|_{L^{\f{2p}{p-2}}}^2,\\
g_2\eqdefa&\bigl(\|\na D_tu\|_{L^2}+\|\na u\|_{L^4}^2\bigr)\|D_tu\|_{L^2}^{\f2p}\bigl(\|D_tu\|_{L^2}+\|\na u\|_{L^{\f{2p}{p-2}}}\bigr)^{1-\f2p},\\
g_3\eqdefa&\bigl(\|\na D_t u\|_{L^{2}}+\|\na u\|_{L^4}^{2}\bigr)^{\f{p}{p-1}}\|\na u\|_{L^2}^{\f{p-2}{p-1}},\\
g_4\eqdefa&\bigl(\|\na D_tu\|_{L^2}+\|\na u\|_{L^4}^2\bigr)\bigl(\|\na u\|_{L^{\f{2p}{p-2}}}+\|\na u\|_{L^2}^{1-\f2p}\|D_tu\|_{L^2}^{\f2p}\bigr).
\end{split}
\eeno
In what follows, we shall take $t_0=\f{t}2$ in \eqref{S4eq14}. Next let us handle the right-hand side of \eqref{S4eq14}.

By virtue of  \eqref{S4eq15}, we find
\beno
\begin{split}
Ct\bigl(&\|\na u\|_{L^{\f{2p}{p-2}}}\|\na X\|_{L^p}+\|\na u\|_{L^2}\bigr)\|\na u_X\|_{L^2}\\
\leq& \f{t}4\|\na u_X\|_{L^2}^2+Ct\bigl(\|\na u\|_{L^{\f{2p}{p-2}}}^2\|\na X\|_{L^p}^2+\|\na u\|_{L^2}^2\bigr)\\
\leq& \f{t}4\|\na u_X\|_{L^2}^2+C_1\w{t}^{-2\d_-}\bigl(1+\|\na X\|_{L^\infty_t(L^p)}^2\bigr).
\end{split}
\eeno
Whereas by integrating \eqref{S3eq13} over $\left[t/2, t\right]$ and using Propositions \ref{S2prop2} and \ref{S3col1}, we arrive at
\beq\label{S4eq17}
\begin{split}
C\int_{\f{t}2}^t\|\na u_X\|_{L^2}^2\,dt'\leq &\|\cY(t/2)\|_{L^2}^2+C_1\w{t}^{-2\d_-}\bigl(1+\|X\|_{L^\infty_t(L^\infty)}^2
+\|\na X\|_{L^\infty_t(L^p)}^2\bigr){\rm V}(t)\\
\leq &C_1\w{t}^{-2\d_-}\bigl(1+\|X\|_{L^\infty_t(L^\infty)}^2
+\|\na X\|_{L^\infty_t(L^p)}^2\bigr){\rm V}(t).
\end{split}
\eeq
It remains to handle the last term in \eqref{S4eq14}, which we shall deal with term by term below.
By applying Proposition  \ref{S2prop2}, \eqref{S4eq15} and \eqref{S4eq17}, we obtain
\beno
\begin{split}
{\Large \bullet}&\int_{\f{t}2}^t\left(t'-{t}/2\right)\|\na u\|_{L^p}\|D_tu\|_{L^2}^{\f2p}\|\na u_X\|_{L^2}^{2\left(1-\f1p\right)}\,dt'\\
&\leq Ct\Bigl(\int_{\f{t}2}^t\|\na u_X(t')\|_{L^2}^2\,dt'\Bigr)^{1-\f1p}\Bigl(\int_{\f{t}2}^t\|\na u(t')\|_{L^p}^p\|D_tu(t')\|_{L^2}^{2}\,dt'\Bigr)^{\f1p}\\
&\leq C_1\w{t}^{-2\d_-}\bigl(1+\|X\|_{L^\infty_t(L^\infty)}^2
+\|\na X\|_{L^\infty_t(L^p)}^2\bigr){\rm V}(t).
\end{split}
\eeno
Similarly, we find
\beno
\begin{split}
{\Large \bullet}&\int_{\f{t}2}^t\left(t'-{t}/2\right)f_2(t')\|\na u_X(t')\|_{L^2}^{\f{p}{p-1}}\,dt'\\
&\leq Ct\Bigl(\int_{\f{t}2}^t\|\na u_X(t')\|_{L^2}^2\,dt'\Bigr)^{\f{p}{2(p-1)}}\Bigl(\int_{\f{t}2}^tf_2(t')^{\f{2(p-1)}{p-2}}\,dt'\Bigr)^{\f{p-2}{2(p-1)}}\\
&\leq C_1\w{t}^{-2\d_-}\bigl(1+\|X\|_{L^\infty_t(L^\infty)}^2
+\|\na X\|_{L^\infty_t(L^p)}^2\bigr){\rm V}(t),
\end{split}
\eeno
and
\beno
\begin{split}
{\Large \bullet}\int_{\f{t}2}^t\left(t'-{t}/2\right)f_3(t')\|\na u_X(t')\|_{L^2}\,dt'
\leq & Ct\Bigl(\int_{\f{t}2}^t\|\na u_X(t')\|_{L^2}^2\,dt'\Bigr)^{\f12}\Bigl(\int_{\f{t}2}^tf_3(t')^2\,dt'\Bigr)^{\f12}\\
\leq & C_1\w{t}^{-2\d_-}\bigl(1+\|X\|_{L^\infty_t(L^\infty)}
+\|\na X\|_{L^\infty_t(L^p)}\bigr){\rm V}(t).
\end{split}
\eeno
Note that
\beno
\begin{split}
\int_{\f{t}2}^tf_4(t')^{\f{2(p-1)}{p}}\,dt'\leq & C\|\na X\|_{L^\infty_t(L^p)}^2\Bigl(\int_{\f{t}2}^t\bigl(\|\na D_tu\|_{L^2}^2+\|\na u\|_{L^4}^4\bigr)\,dt'\\
\leq &C_1t^{-1}\w{t}^{-2\left(\f12+\d\right)_-}\bigl(1+\bigl\|\sqrt{t}\w{t}^{\left(\f12+\d\right)_-}\na D_tu\bigr\|_{L^2_t(L^2)}^2\bigr),
\end{split}
\eeno
we deduce from Proposition \ref{S3prop1} and \eqref{S4eq17} that
\beno
\begin{split}
{\Large \bullet}\int_{\f{t}2}^t&\left(t'-{t}/2\right)f_4(t')\|\na u_X(t')\|_{L^2}^{\f{p-2}{p-1}}\,dt'\\
\leq & Ct \Bigl(\int_{\f{t}2}^tf_4(t')^{\f{2(p-1)}{p}}\,dt'\Bigr)^{\f{p}{2(p-1)}}\Bigl(\int_{\f{t}2}^t\|\na u_X(t')\|_{L^2}^{2}\,dt'\Bigr)^{\f{p-2}{2(p-1)}}\\
\leq & C_1\w{t}^{-2\d_-}\bigl(1+\|X\|_{L^\infty_t(L^\infty)}^2
+\|\na X\|_{L^\infty_t(L^p)}^2\bigr){\rm V}(t),
\end{split}
\eeno
and
\beno
\begin{split}
{\Large \bullet}\int_{\f{t}2}^t&\left(t'-{t}/2\right)f_5(t')\|\na u_X(t')\|_{L^2}^{\f{p-2}{p}}\,dt'\\
\leq &\|\na X\|_{L^\infty_t(L^p)}\Bigl(\int_{\f{t}2}^t(t'-t/2)^2\bigl(\|\na D_tu\|_{L^2}^2+\|\na u\|_{L^4}^4\bigr)\,dt'\Bigr)^{\f{p}{p+2}}\\
&\times \Bigl(\int_{\f{t}2}^t\|D_tu\|_{L^2}^2\,dt'\Bigr)^{\f{2}{p+2}}\Bigl(\int_{\f{t}2}^t\|\na u_X(t')\|_{L^2}^{2}\,dt'\Bigr)^{\f{p-2}{2p}}\\
\leq & C_1\w{t}^{-2\d_-}\bigl(1+\|X\|_{L^\infty_t(L^\infty)}^2
+\|\na X\|_{L^\infty_t(L^p)}^2\bigr){\rm V}(t).
\end{split}
\eeno
Observing that
\beno
\begin{split}
\int_{\f{t}2}^t\left(t'-{t}/2\right)\|D_tu\|_{L^{\f{2p}{p-2}}}^2\,dt'\leq& Ct
\Bigl(\int_{\f{t}2}^t\|D_tu\|_{L^{2}}^2\,dt'\Bigr)^{1-\f2p}\Bigl(\int_{\f{t}2}^t\|\na D_tu\|_{L^{2}}^2\,dt'\Bigr)^{\f2p}\\
\leq &C \w{t}^{-2\left(\f1p+\d\right)_-}\bigr\|\w{t}^{\left(\f12+\d\right)_-}D_tu\bigr\|_{L^2_t(L^2)}^{1-\f2p}
\bigr\|\sqrt{t}\w{t}^{\left(\f12+\d\right)_-}\na D_tu\bigr\|_{L^2_t(L^2)}^{\f2p}.
\end{split}
\eeno
The same  estimate holds for the other terms in $g_1(t).$ As a result, it comes out
\beno
\begin{split}
{\Large \bullet}&\int_{\f{t}2}^t\left(t'-{t}/2\right)
g_1(t')\bigl(\|\na X(t')\|_{L^p}^2+\|X
(t')\|_{L^\infty}^2\bigr)\,dt'\\
&\qquad\qquad\qquad\qquad\leq  C_1\w{t}^{-2\d_-}\bigl(1+\|X\|_{L^\infty_t(L^\infty)}^2
+\|\na X\|_{L^\infty_t(L^p)}^2\bigr).
\end{split}
\eeno
Along the same line, we deduce from Propositions \ref{S2prop2} and \ref{S3prop1} that
\beno
\begin{split}
{\Large \bullet}&\int_{\f{t}2}^t\left(t'-{t}/2\right) g_2(t') \bigl(\|\na X(t')\|_{L^p}^{2\left(1-\f1p\right)}+\|X(t')
\|_{L^\infty}^{2\left(1-\f1p\right)}\bigr)
\,dt'\\
&\leq  Ct\bigl(\|\na X\|_{L^\infty_t(L^p)}+\|X
\|_{L^\infty_t(L^\infty)}\bigr)^{2\left(1-\f1p\right)}\Bigl(\int_{\f{t}2}^t\bigl(\|\na D_tu\|_{L^2}^2+\|\na u\|_{L^4}^4\bigr)\,dt'\Bigr)^{\f12}\\
&\qquad\times\Bigl(\int_{\f{t}2}^t\|D_tu\|_{L^2}^2\,dt'\Bigr)^{\f{1}{p}}
\Bigl(\int_{\f{t}2}^t\bigl(\|D_tu\|_{L^2}^2+\|\na u\|_{L^{\f{2p}{p-2}}}^2\bigr)\,dt'\Bigr)^{\f{p-2}{2p}}
 \\
&\leq  C_1\w{t}^{-2\d_-}\bigl(\|\na X\|_{L^\infty_t(L^p)}+\|X
\|_{L^\infty_t(L^\infty)}\bigr)^{2\left(1-\f1p\right)},
\end{split}
\eeno
and
\beno
\begin{split}
{\Large \bullet}&\int_{\f{t}2}^t\left(t'-{t}/2\right)g_3(t')\|\na X(t')\|_{L^p}^{\f{p}{p-1}}\,dt'\\
&\quad\leq Ct\|\na X\|_{L^\infty_t(L^p)}^{\f{p}{p-1}}
\Bigl(\int_{\f{t}2}^t\bigl(\|\na D_tu\|_{L^2}^2+\|\na u\|_{L^4}^4\bigr)\,dt'\Bigr)^{\f{p}{2(p-1)}}\Bigl(\int_{\f{t}2}^t\|\na u(t')\|_{L^2}^2\,dt'\Bigr)^{\f{p-2}{2(p-1)}}\\
&\quad\leq  C_1\w{t}^{-\f1{p-1}-2\d_-}\|\na X\|_{L^\infty_t(L^p)}^{\f{p}{p-1}},
\end{split}
\eeno
and
\beno
\begin{split}
{\Large \bullet}&\int_{\f{t}2}^t\left(t'-{t}/2\right)g_4(t')\|\na X(t')\|_{L^p}\,dt'\leq Ct\|\na X\|_{L^\infty_t(L^p)}
\Bigl(\int_{\f{t}2}^t\bigl(\|\na D_tu\|_{L^2}^2+\|\na u\|_{L^4}^4\bigr)\,dt'\Bigr)^{\f12}\\
&\qquad\times \Bigl(\Bigl(\int_{\f{t}2}^t \|\na u\|_{L^{\f{2p}{p-2}}}^2\,dt'\Bigr)^{\f12}
+\Bigl(\int_{\f{t}2}^t\|\na u(t')\|_{L^2}^2\,dt'\Bigr)^{\f{p-2}{2p}} \Bigl(\int_{\f{t}2}^t\|D_t u(t')\|_{L^2}^2\,dt'\Bigr)^{\f{1}{p}}\Bigr)\\
&\quad\ \leq  C_1\w{t}^{-\f1p-2\d_-}\|\na X\|_{L^\infty_t(L^p)},
\end{split}
\eeno
and finally
\beno
\begin{split}
{\Large \bullet}&\int_{\f{t}2}^t\left(t'-{t}/2\right)
\Bigl(\|D_tu\|_{L^2}^2+\bigl(\|\na u\|_{L^2}^2+\|\na u\|_{L^{\f{2p}{p-2}}}^2\bigr)\|\na u\|_{L^\infty}\Bigr)\,dt'\\
&\quad \leq \w{t}^{-2\d_-}\Bigl(\bigl\|\w{t}^{\left(\f12+\d\right)_-}D_tu\|_{L^2_t(L^2)}^2+C_1\|\na u\|_{L^1_t(L^\infty)}\Bigr)\\
&\quad\leq C_1\w{t}^{-2\d_-}\bigl(1+\|\na u\|_{L^1_t(L^\infty)}\bigr).
\end{split}
\eeno
Consequently,  we get, by applying Gronwall's inequality to \eqref{S4eq14} and then inserting the above estimates to the resulting inequality, that
\beq\label{S4eq18}
\begin{split}
t\|\na u_X(t)\|_{L^2}^2\leq C_1\w{t}^{-2\d_-}\bigl(1+\|X\|_{L^\infty_t(L^\infty)}^2
+\|\na X\|_{L^\infty_t(L^p)}^2\bigr){\rm V}(t).
\end{split}
\eeq
With the above estimate, by repeating the argument following \eqref{S3eq7u}, we can prove the estimate of  $\bigl\|\sqrt{t}\w{t}^{\d_-}D_tu\bigr\|_{L^2_t(L^2)}$
in \eqref{S4eq13}.

Under the additional assumption that $\p_{X_0}u_0\in H^1(\R^2),$ we get by a similar derivation of \eqref{S4eq18} that
\beno
\|\na u_X\|_{L^\infty_t(L^2)}^2+\|D_tu\|_{L^2_t(L^2)}\leq C_1\bigl(1+\bigl\|\na\p_{X_0}u_0\bigr\|_{L^2}^2 +\|X\|_{L^\infty_t(L^\infty)}^2
+\|\na X\|_{L^\infty_t(L^p)}^2\bigr){\rm V}(t),
\eeno
which together with \eqref{S4eq13} ensures \eqref{S4eq13op}.
This completes the proof of Proposition \ref{S4prop1}.
\end{proof}

\setcounter{equation}{0}
\section{The estimate of $\|\na \p_XX\|_{L^2}$}\label{Sect9}

\begin{lem}\label{S7prop1}
{\sl Let $(\r,u, X)$ be a smooth enough solution of the coupled system of \eqref{INS} with \eqref{Xtrans}.
We assume that $\p_{X_0}\mu(\r_0)\in L^6(\R^2)\cap L^\infty(\R^2)$ and $\p_{X_0}^2\mu(\r_0)\in L^2(\R^2).$
Then under the smallness assumption \eqref{S1eq2}, we have
\beq \label{S7eq4}
\begin{split}
\|\na \p_X^2u\|_{L^2}\leq C\Bigl(&\bigl(\|\Pi_2\|_{L^\infty}+\|\na u\|_{L^\infty}\bigr)\|\na\p_XX\|_{L^2}+\|\na u\|_{L^4}\bigl(\|\na X\|_{L^4}\\
&+\|\na X\|_{L^8}^2\bigr)+\|\na u_X\|_{L^3}\bigl(\|\na X\|_{L^6}+\|\p_X\mu(\r_0)\|_{L^6}\bigr)\\
&+\|\na u\|_{L^\infty}\bigl(\|\p_X^2\mu(\r_0)\|_{L^2}+\|\na X\|_{L^4}^2\bigr)\\
&+\|X\|_{L^\infty}\bigl(\|D_tu\|_{L^2}+\|D_tu_X\|_{L^2}+\|\na X\|_{L^4}\|D_tu\|_{L^4}\bigr)\Bigr),
\end{split}
\eeq where $\Pi_2$ is given by \eqref{b.4a}.}
\end{lem}

\begin{proof}
According to \eqref{S4eq3p}, we write
\beq \label{S7eq1}
\begin{split}
\na \p^2_Xu=&[\na; \p_X]u_X+\p_X\na \hbar+\p_X\na(-\D)^{-1}\PP\dive\bigl(\mu(\rho)\cM(\hbar)\bigr)\\
&-\p_X\na(-\D)^{-1}\PP\bigl(\r D_tu_X-G)+\p_X\na(-\D)^{-1}\PP\dive\bigl((\mu(\rho)-1)\cM(u_X-\hbar)\bigr).
\end{split}
\eeq
In what follows, we shall handle the estimates of the above terms.

\no$\bullet$ \underline{The estimate of $\p_X\na \hbar$}

Recall that $\hbar\eqdefa -\na(-\D)^{-1}\mbox{Tr}(\na X\na u),$ we write
\beno
\p_X\na \hbar=-[\p_X; \na^2(-\D)^{-1}]{\rm Tr}(\bigl(\na X\na u\bigr)-\na^2(-\D)^{-1}{\rm Tr}\bigl(\p_X\na X\na u+\na X\p_X\na u\bigr),
\eeno
from which and the classical commutator's estimate, we infer
\beq \label{S7eq2}
\|\p_X\na \hbar\|_{L^2}\lesssim \bigl(\|\na X\|_{L^4}^2+\|\na\p_XX\|_{L^2}\bigr)\|\na u\|_{L^\infty}+\|\na X\|_{L^6}\|\na u_X\|_{L^3}.
\eeq

\no$\bullet$ \underline{The estimate of $\p_X\na(-\D)^{-1}\PP\dive\bigl(\mu(\rho)\cM(\hbar)\bigr)$}

We first write
\beno
\begin{split}
\p_X\na(-\D)^{-1}\PP\dive\bigl(\mu(\rho)\cM(\hbar)\bigr)=&[\p_X; \na(-\D)^{-1}\PP\dive]\bigl(\mu(\rho)\cM(\hbar)\bigr)\\
&+\na(-\D)^{-1}\PP\dive\bigl(\p_X\mu(\rho)\cM(\hbar)+\mu(\rho)\p_X\cM(\hbar)\bigr),
\end{split}
\eeno
from which and \eqref{S7eq2}, we deduce from  the classical commutator's estimate that
\beno
\begin{split}
\bigl\|\p_X\na(-\D)^{-1}\PP\dive\bigl(&\mu(\rho)\cM(\hbar)\bigr)\bigr\|_{L^2}\lesssim \|\na X\|_{L^4}\|\na\hbar\|_{L^4}
+\|\na\hbar\|_{L^2}+\|\p_X\na\hbar\|_{L^2}\\
\lesssim &\bigl(\|\na X\|_{L^2}+\|\na X\|_{L^4}^2+\|\na\p_XX\|_{L^2}\bigr)\|\na u\|_{L^\infty}+\|\na X\|_{L^6}\|\na u_X\|_{L^3}.
\end{split}
\eeno

\no$\bullet$ \underline{The estimate of $
\p_X\na(-\D)^{-1}\PP\dive\bigl((\mu(\rho)-1)\cM(u_X-\hbar)\bigr)$}

We first write
\beno
\begin{split}
\p_X&\na(-\D)^{-1}\PP\dive\bigl((\mu(\rho)-1)\cM(u_X-\hbar)\bigr)\\
=&[\p_X;\na(-\D)^{-1}\PP\dive]\bigl((\mu(\rho)-1)\cM(u_X-\hbar)\bigr)\\
&+\na(-\D)^{-1}\PP\dive\bigl((\mu(\r)-1)\p_X\cM(u_X-\hbar)+\p_X\mu(\r)\cM(u_X-\hbar)\bigr),
\end{split}
\eeno
from which, and the commutator's estimate, we infer
\beno
\begin{split}
\bigl\|&\p_X\na(-\D)^{-1}\PP\dive\bigl((\mu(\rho)-1)\cM(u_X-\hbar)\bigr)\bigr\|_{L^2}\\
&\lesssim \|\p_X\mu(\r)\|_{L^6} \|\na u_X\|_{L^3}+\|\p_X \mu(\r)\|_{L^\infty}\|\na\hbar\|_{L^2}+\|\mu(\r)-1\|_{L^\infty}\\
 &\quad \times \bigl(\|\na \hbar\|_{L^4}\|\na X\|_{L^4}+\|\na u_X\|_{L^3}\|\na X\|_{L^6}+\|\na \p_X^2u\|_{L^2}+\|\p_X\na\hbar\|_{L^2}\bigr).
\end{split}
\eeno
Then by virtue of \eqref{S9eq20} and \eqref{S7eq2}, we achieve
\beno
\begin{split}
\bigl\|&\p_X\na(-\D)^{-1}\PP\dive\bigl((\mu(\rho)-1)\cM(u_X-\hbar)\bigr)\bigr\|_{L^2}\\
&\lesssim \|\p_{X_0}\mu(\r_0)\|_{L^6}\|\na u_X\|_{L^3}+\|\na u\|_{L^\infty}\|\na X\|_{L^2}+\|\mu(\r)-1\|_{L^\infty}\\
 &\quad\times \bigl(\|\na \p_X^2u\|_{L^2}+\bigl(\|\na X\|_{L^4}^2+\|\na \p_XX\|_{L^2}\bigr)\|\na u\|_{L^\infty}+\|\na X\|_{L^6}\|\na u_X\|_{L^3}\bigr).
\end{split}
\eeno

\no$\bullet$ \underline{The estimate of $\p_X\na(-\D)^{-1}\PP\bigl(\r D_tu_X-G)$}

It is easy to observe that
\beno
\bigl\|\p_X\na(-\D)^{-1}\PP\p_X(\r D_tu)\bigr\|_{L^2}\lesssim \|X\|_{L^\infty}\bigl(\|D_tu\|_{L^2}+\|D_tu_X\|_{L^2}\bigr).
\eeno
While in view of \eqref{b.4a}, we have
\beno
\begin{split}
\bigl\|\p_X\na(-\D)^{-1}\PP(\p_jX\cdot\na\Pi_1)\|_{L^2}
\lesssim &\|X\|_{L^\infty}\|\na X\|_{L^4}\|\na\Pi_1\|_{L^4}\\
\lesssim &\|X\|_{L^\infty}\|\na X\|_{L^4}\|D_tu\|_{L^4}.
\end{split}
\eeno
To estimate $\p_X\na(-\D)^{-1}\PP(\p_jX\cdot\na\Pi_2),$ we write
\beno
\p_X\na(-\D)^{-1}\PP(\p_jX\cdot\na\Pi_2)=[\p_X;\na(-\D)^{-1}\PP\dive](\p_jX\Pi_2)+\na(-\D)^{-1}\PP\dive\p_X(\p_jX\Pi_2),
\eeno
which implies
\beq\label{S7eq3}
\bigl\|\p_X\na(-\D)^{-1}\PP(\p_jX\cdot\na\Pi_2)\bigr\|_{L^2}\lesssim \|\na X\|_{L^8}^2\|\Pi_2\|_{L^4}+\|\na\p_XX\|_{L^2}\|\Pi_2\|_{L^\infty}
+\|\p_jX\p_X\Pi_2\|_{L^2}.
\eeq
Yet note from \eqref{b.4a} that
\beno
\begin{split}
\p_X\Pi_2=&[\p_X; (-\D)^{-1}\dive\dive](2\mu(\r)\cM(u))\\
&+(-\D)^{-1}\dive\dive\bigl(2\p_X\mu(\r)\cM(u)+2\mu(\r)\p_X\cM(u)\bigr),
\end{split}
\eeno
from which, we infer
\beno
\|\p_jX\p_X\Pi_2\|_{L^2}\lesssim \bigl(\|\na X\|_{L^8}^2+\|\na X\|_{L^4}\bigr)\na u\|_{L^4}+\|\na X\|_{L^6}\|\na u_X\|_{L^3}.
\eeno
Inserting the above estimates into \eqref{S7eq3} gives rise to
\beno
\begin{split}
\bigl\|\p_X\na(-\D)^{-1}\PP(\p_jX\cdot\na\Pi_2)\bigr\|_{L^2}\lesssim &\bigl(\|\na X\|_{L^8}^2+\|\na X\|_{L^4}\bigr)\na u\|_{L^4}\\
&+\|\na X\|_{L^6}\|\na u_X\|_{L^3}
+\|\na\p_XX\|_{L^2}\|\Pi_2\|_{L^\infty}.
\end{split}
\eeno
Once again through a commutative argument, we write
$$ \longformule{
\p_X\na(-\D)^{-1}\PP\p_i\bigl(2\p_X\mu(\r)\cM_{ij}(u)\bigr)=[\p_X; \na(-\D)^{-1}\PP\p_i]\bigl(2\p_X\mu(\r)\cM_{ij}(u)\bigr)}{{}
+\na(-\D)^{-1}\PP\p_i\bigl(2\p_X^2\mu(\r)\cM_{ij}(u)+2\p_X\mu(\r)\p_X\cM_{ij}(u)\bigr),} $$
from which, we infer
\beno
\begin{split}
\bigl\|\p_X\na(-\D)^{-1}&\PP\p_i\bigl(2\p_X\mu(\r)\cM_{ij}(u)\bigr)\bigr\|_{L^2}\\
&\lesssim \|\na X\|_{L^4}\|\na u\|_{L^4}+\|\p_X^2\mu(\r)\|_{L^2}\|\na u\|_{L^\infty}
+\|\p_X\mu(\r)\|_{L^6}\|\na u_X\|_{L^3}.
\end{split}
\eeno
To deal with the remaining terms in \eqref{S4eq365}, we write
\beno
\begin{split}
\p_X\na(-\D)^{-1}&\PP\p_i\left(\mu(\r)\bigl(\p_iX\cdot\na u^j+\p_jX\cdot\na u^i\bigr)\right)\\
&=[\p_X;\na(-\D)^{-1}\PP\p_i]\left(\mu(\r)\bigl(\p_iX\cdot\na u^j+\p_jX\cdot\na u^i\bigr)\right)\\
&\quad +\na(-\D)^{-1}\PP\p_i\left(\p_X\mu(\r)\bigl(\p_iX\cdot\na u^j+\p_jX\cdot\na u^i\bigr)\right)\\
&\quad +\na(-\D)^{-1}\PP\p_i\left(\mu(\r)\p_X\bigl(\p_iX\cdot\na u^j+\p_jX\cdot\na u^i\bigr)\right),
\end{split}
\eeno
and
\beno
\begin{split}
\p_X\na(-\D)^{-1}\PP\left(\p_iX\cdot\na\bigl(2\mu(\r)\cM_{ij}(u)\bigr)\right)
=&[\p_X;\na(-\D)^{-1}\PP\dive]\left(\p_iX\bigl(2\mu(\r)\cM_{ij}(u)\bigr)\right)\\
&+\na(-\D)^{-1}\PP\dive\p_X\left(\p_iX\bigl(2\mu(\r)\cM_{ij}(u)\bigr)\right),
\end{split}
\eeno
from which, we deduce that
$$\longformule{
\bigl\|\p_X\na(-\D)^{-1}\PP\p_i\left(\mu(\r)\bigl(\p_iX\cdot\na u^j+\p_jX\cdot\na u^i\bigr)\right)\bigr\|_{L^2}}{{}
\lesssim \bigl(\|\na X\|_{L^8}^2+\|\na X\|_{L^4}\bigr)\|\na u\|_{L^4}+\|\na u\|_{L^\infty}\|\na\p_XX\|_{L^2},} $$
and
$$\longformule{
\bigl\|\p_X\na(-\D)^{-1}\PP\left(\p_iX\cdot\na\bigl(2\mu(\r)\cM_{ij}(u)\bigr)\right)\bigr\|_{L^2}}{{}
\lesssim \|\na X\|_{L^8}^2\|\na u\|_{L^4}+\|\na X\|_{L^6}\|\na u_X\|_{L^3}+\|\na u\|_{L^\infty}\|\na\p_XX\|_{L^2}.
}$$
Hence in view of \eqref{S4eq365}, by summarizing the above estimates, we achieve
\beno
\begin{split}
\bigl\|\p_X\na(-\D)^{-1}&\PP\bigl(\r D_tu_X-G)\bigr\|_{L^2}\lesssim  \bigl(\|\Pi_2\|_{L^\infty}+\|\na u\|_{L^\infty}\bigr)\|\na\p_XX\|_{L^2}\\
&+\|\na u\|_{L^4}\bigl(\|\na X\|_{L^4}+\|\na X\|_{L^8}^2\bigr)+ \|\na u\|_{L^\infty}\|\p_X^2\mu(\r)\|_{L^2}\\
&+\|\na u_X\|_{L^3}\bigl(\|\na X\|_{L^6}+\|\p_X\mu(\r)\|_{L^6}\bigr)\\
&+\|X\|_{L^\infty}\bigl(\|D_tu\|_{L^2}+\|D_tu_X\|_{L^2}+\|\na X\|_{L^4}\|D_tu\|_{L^4}\bigr).
\end{split}
\eeno

Thanks to \eqref{S7eq1} and \eqref{S9eq20}, by summarizing the above estimates and using the smallness condition \eqref{S1eq2},
we obtain \eqref{S7eq4}. This completes the proof of Lemma \ref{S7prop1}.
\end{proof}

\begin{proof}[Proof of Proposition \ref{S0prop1}] Taking $\p_X$ to \eqref{Xtrans} yields
\beno
D_t\p_XX=\p_X^2u.
\eeno
Taking $\na$ to the above equation and then taking $L^2$ inner product of the resulting equation with
$\na\p_XX,$ we find
\beq \label{S7eq6}
\f{d}{dt}\|\na\p_XX(t)\|_{L^2}\leq \|\na u\|_{L^\infty}\|\na\p_XX(t)\|_{L^2}+\|\na\p_X^2u\|_{L^2}.
\eeq
On the other hand, by taking $p=6$ in \eqref{S4eq3} and  inserting the resulting inequality to \eqref{S7eq4},
and then substitute it to \eqref{S7eq6}, we achieve
\beno
\begin{split}
\f{d}{dt}&\|\na\p_XX(t)\|_{L^2}\leq C_2\Bigl(\bigl( \|\na u\|_{L^\infty}+\|\Pi_2\|_{L^\infty}\bigr)\|\na\p_XX(t)\|_{L^2}+\|D_tu_X\|_{L^2}\|X\|_{L^\infty}\\
&+\bigl[\bigl(\|D_tu\|_{L^2}+\|D_tu_X\|_{L^2}\bigr)^{\f13}\bigl(\|\na u\|_{L^2}+\|\na u\|_{L^3}+\|D_tu\|_{L^2}+\|\na u_X\|_{L^2}\bigr)^{\f23}\\
&+\|\na u\|_{L^3}+\|\na u\|_{L^\infty}+\|D_tu\|_{L^2}+\|D_tu\|_{L^6}\bigr]\bigl(1+\|X\|_{L^\infty}^2+\|\na X\|_{L^4}^2+\|\na X\|_{L^8}^2\bigr)\Bigr).
\end{split}
\eeno
Note that for $\d\in \left]1/3,1/2\right[,$ we deduce from Propositions \ref{S3col1} and \ref{S4prop1} that
\beno
\begin{split}
\int_0^t\|D_tu_X\|_{L^2}^{\f13}\|\na u_X\|_{L^2}^{\f23}\,dt'\leq &\bigl\|\w{t}^{\left(\d+\f12\right)_-}D_tu_X\|_{L^2_t(L^2)}^{\f13}
\bigl\|\w{t}^{\d_-}\na u_X\bigr\|_{L^2_t(L^2)}^{\f23}\\
&\times\Bigl(\int_0^t\w{t'}^{-\left(2\d+\f13\right)_-}\,dt'\Bigr)^{\f12}\\
\leq &C\Bigl(\bigl\|\w{t}^{\left(\d+\f12\right)_-}D_tu_X\|_{L^2_t(L^2)}+
\bigl\|\w{t}^{\d_-}\na u_X\bigr\|_{L^2_t(L^2)}\Bigr).
\end{split}
\eeno
Therefore, \eqref{S0eq1} follows from Gronwall's inequality,  Propositions \ref{S2prop2} and \ref{S3prop1}, and Propositions \ref{S3col1} and \ref{S4prop1}.
This completes the proof of Proposition \ref{S0prop1}.
\end{proof}

\appendix
\section{The commutative estimate}\label{Sappb}

Let us first recall the following lemma from \cite{FS}.

\begin{lem} \label{Salem1}
{\sl Let $p, q\in ]1,\infty[$ or $p=q=\infty.$ Let $\{f_j\}_{j\in\Z}$ be a sequence of functions in $L^p(\R^d)$ so that
$\left(f_j(x)\right)_{\ell^q(\Z)}\in L^p(\R^d).$ Then there holds
\beno
\bigl\|\left(M(f_j)(x)\right)_{\ell^q}\bigr\|_{L^p}\leq C\bigl\|\left(f_j(x)\right)_{\ell^q}\bigr\|_{L^p}.
\eeno}
\end{lem}

\begin{prop}\label{Saprop1}
{\sl Let $p, r\in ]1,\infty[$ and $q\in ]1, \infty]$ satisfying $\f1r=\f1p+\f1q.$ Let  $X=(X^1,\cdots, X^d)\in \dot{W}^{1,p}(\R^d)$ with $\dive X=0,$
 $g\in L^q(\R^d),$ $R_i=\p_i(-\D)^{-\f12}$
be the Riesz transform. Then one has
\beq \label{Saeq1a}
\bigl\|[\p_X; R_iR_j]g\bigr\|_{L^r}\leq C\|\na X\|_{L^p}\|g\|_{L^q}.
\eeq}
\end{prop}

\begin{proof}
We first get by, applying Bony's decomposition that
\beq\label{Saeq2a}
\begin{split}
[\p_X; R_iR_j]f=[T_{X^k}; R_iR_j]\p_kg&+T_{\p_kR_iR_jg}X^k-R_iR_j(T_{\p_kg}X^k)\\
&+R(X^k,\p_kR_iR_jg)-R_iR_j(R(X^k,\p_kg)).
\end{split}
\eeq
In view of (2.27) of \cite{BCD}, one has
\beq  \label{Saeq3a}
|\D_j(\p_kR_iR_jg)(x)|\leq C2^j {\rm M}g(x)\Longrightarrow
|S_j(\p_kR_iR_jg)(x)|\leq C2^j {\rm M}g(x),
\eeq
from which, we infer
\beno
\begin{split}
\bigl|\D_j\bigl(T_{\p_kR_iR_jg}X^k\bigr)(x)\bigr|\leq &C\sum_{|j'-j|\leq 4}{\rm M}\big(S_{j'-1}(\p_kR_iR_jg)\D_{j'}X^k\bigr)(x)\\
\leq &C\sum_{|j'-j|\leq 4}{\rm M}\big({\rm M}g 2^{j'}\D_{j'}X^k\bigr)(x).
\end{split}
\eeno
As a result, we deduce from Lemma \ref{Salem1} that
\beno
\begin{split}
\bigl\|T_{\p_kR_iR_jg}X^k\bigr\|_{L^r}\leq &C\Bigl\|\Bigl\{\sum_{j\in\Z}\bigl[{\rm M}\big({\rm M}g 2^{j'}\D_{j'}X^k\bigr)\bigr]^2\Bigr\}^{\f12}\Bigr\|_{L^r}\\
\leq &C\Bigl\|{\rm M}g\Bigl\{\sum_{j\in\Z}\bigl[ 2^{j'}\D_{j'}X^k\bigr]^2\Bigr\}^{\f12}\Bigr\|_{L^r}\\
\leq &C\|{\rm M}g\|_{L^q}\Bigl\|\Bigl\{\sum_{j\in\Z}\bigl[ 2^{j'}\D_{j'}X^k\bigr]^2\Bigr\}^{\f12}\Bigr\|_{L^p}\leq C\|g\|_{L^q}\|\na X\|_{L^p}.
\end{split}
\eeno
The same estimate holds for $T_{\p_kg}X^k,$ so that we obtain
\beno
\|R_iR_j(T_{\p_kg}X^k)\|_{L^r}\leq C\|T_{\p_kg}X^k\|_{L^r}\leq C\|g\|_{L^q}\|\na X\|_{L^p}.
\eeno
While due to $\dive X=0$ and \eqref{Saeq3a}, we write
\beno
\begin{split}
|\D_jR(X^k,\p_kR_iR_jg)(x)|=&|\p_k\D_jR(X^k, R_iR_jg)(x)|\\
\leq &C2^j{\rm M}\bigl(\sum_{j'\geq j-5}\D_{j'}X^k\wt{\D}_{j'}R_iR_jg\bigr)(x)\\
\leq &C{\rm M}\bigl(\sum_{j'\geq j-5}2^{j-j'}|2^{j'}\D_{j'}X^k| {\rm M}g\bigr)(x),
\end{split}
\eeno
from which, we infer
\beno
\begin{split}
\bigl\|R(X^k,\p_kR_iR_jg)\bigr\|_{L^r}\leq &\Bigl\|\Bigl\{\sum_{j\in\Z}\bigl[\sum_{j'\geq j-5}2^{j-j'}|2^{j'}\D_{j'}X^k| {\rm M}g\bigr]^2\Bigr\}^{\f12}\Bigr\|_{L^r}\\
\leq &C\Bigl\|{\rm M}g\Bigl\{\sum_{j\in\Z}\bigl[ 2^{j'}\D_{j'}X^k\bigr]^2\Bigr\}^{\f12}\Bigr\|_{L^r}\leq C\|g\|_{L^q}\|\na X\|_{L^p}.
\end{split}
\eeno
The same estimate holds for $R_iR_j(R(X^k,\p_kg)).$

Finally let us turn to the first term on the right hand side of \eqref{Saeq2a}. We first get, by a similar
derivation of \eqref{S9eq3qe}, that
\beno
[T_{X^k}; R_iR_j]\p_kg=
\sum_{|j-\ell|\leq 4}2^{2\ell}\int_{\R^2}\check{\th}(2^{\ell}z)\int_0^1\D_j\na X(x-\tau z)\cdot
z\,d\tau S_j\p_kg(x-z)\,dz,
\eeno
where $\th(\xi)\eqdefa\xi_i\xi_j|\xi|^{-2}\phi(\xi).$  Then we deduce from \eqref{Saeq3a} that
\beno
|[T_{X^k}; R_iR_j]\p_kg|\leq
\sum_{|j-\ell|\leq 4}2^{\ell}\int_{\R^2}\Psi(2^{\ell}z)\int_0^1|\D_j\na X(x-\tau z)|\,d\tau 2^j{\rm M}g(x-z)\,dz,
\eeno
for $\Psi(z)\eqdefa |z||\check{\th}(z)|.$ Now since $r\in ]1,\infty[$ and $\f1p+\f1q=\f1r<1,$ we can choose $\al,\beta\in ]0.1[$ satisfying
\beq \label{Saeq4a}
\al+\beta=1,\quad \al>\f12,\quad p\al>1\andf q\beta>1.
\eeq
We get, by applying H\"older's inequality, that
\beno
\begin{split}
|[T_{X^k}; R_iR_j]\p_kg(x)|\leq &\sum_{|j-\ell|\leq 4} \Bigl(\int_0^12^{2\ell}\int_{\R^2}\Psi(2^\ell z)|\D_j\na X(x-\tau z)|^{\f1\al}\,dz\,d\tau\Bigr)^\al\\
&\qquad\qquad\qquad\qquad\times \Bigl(2^{2\ell}\int_{\R^2}\Psi(2^\ell z)\bigl[{\rm M}g(x-z)\bigr]^{\f1\beta}\,dz\Bigr)^\beta\\
\leq &\sum_{|j-\ell|\leq 4}\bigl[{\rm M}(|\D_j\na X|^{\f1\al})\bigr]^\al\bigl[{\rm M}(({\rm M}g)^{\f1\beta})\bigr]^\beta,
\end{split}
\eeno
from which, \eqref{Saeq4a} and Lemma \ref{Salem1}, we deduce that
\beno
\begin{split}
\|[T_{X^k}; R_iR_j]\p_kg\|_{L^r}\leq &\Bigl\|\Bigl\{\sum_{j\in\Z}\bigl[{\rm M}(|\D_j\na X|^{\f1\al})\bigr]^{2\al}\Bigr\}^{\f1{2\al}}\Bigr\|_{L^{p\al}}^\al
\bigl\|{\rm M}({\rm M}g)^{\f1\beta})\bigr\|_{L^{q\beta}}^\beta\\
\leq &C\Bigl\|\Bigl\{\sum_{j\in\Z}\bigl[ 2^{j'}\D_{j'}X^k\bigr]^2\Bigr\}^{\f12}\Bigr\|_{L^p}\bigl\|{\rm M}g^{\f1\beta}\bigr\|_{L^{q\beta}}^\beta\\
\leq &C\|\na X\|_{L^p}\|g\|_{L^q}.
\end{split}
\eeno
By summing up the above estimate, we conclude the  proof of \eqref{Saeq1a}. \end{proof}

\section{Lipschitz estimate of elliptic equation of divergence form}\label{append}

The goal of this appendix is to generalize Proposition \ref{S9prop1} to elliptic equation of divergence form
 with bounded coefficients which may have a small gap across a  surface. The main result reads

\begin{prop}\label{Sapp}
{\sl Let $p\in ]2,\infty[,$ $X=\bigl(X_\la\bigr)_{\la\in\Lambda}$ be a non-degenerate family of vector fields in the sense of  Definition \ref{S1def1} with
$X_\la\in C^1_b(\R^2)$ and $\na X_\la\in L^p(\R^2)$ for each $\la\in\La.$
Let $a_{ij}\in L^\infty(\R^2)$ with $\sup_{\la\in\La}\bigl\|\partial_{X_\la} a_{i,j}\bigr\|_{L^\infty}\leq C$ and $\|(a_{ij})_{ij}-Id\|_{L^\infty}\leq \e_0$
for some $\e_0$ sufficiently small.
We assume moreover that $f\in L^p\cap L^{\f{2p}{2+p}}(\R^2)$ and  $\partial_{X_\la}f\in L^{\f{2p}{2+p}}(\R^2)$ for $\la\in\La.$
Then the following equation
\beq \label{Saeq3} \sum_{i,j=1}^2\partial_i(a_{ij}\partial_j u)=f \quad\mbox{for}\quad x\in\R^2,\eeq
has a unique solution $u\in \dot{W}^{1,p}(\R^2)\cap \dot{W}^{1,\infty}(\R^2) $ so that for any $s\in ]2/p,1[,$
\beq \label{Saeq1}
\begin{split}
\|\nabla u\|_{L^\infty}\leq  C_s\Bigl(&\e_0\Bigl[1+C(s,p,X)\sup_{\la\in\La}\|X_\la\|_{L^\infty}\bigl(\|\na X_\la\|_{L^p}^{\f2{p-2}}+\|\na X_\la\|_{L^p}^{\f2{ps-2}}\bigr)
\Bigr]\|f\|_{L^\f{2p}{2+p}}\\
&+\|f\|_{ L^{\f{2p}{2+p}}}^{1-\f2p}\|f\|_{L^p}^{\f2p}\Bigr)+\f{C}{I(X)}\sup_{\la\in\La}\Bigl(\e_0\|X_\la\|_{L^\infty}\|f\|_{L^\f{2p}{2+p}}\Bigr)^{1-\f2p}\\
&\qquad\times\Bigl(\bigl[1+\e_0\|\na X_\la\|_{L^\infty}\bigr]\|f\|_{L^\f{2p}{2+p}}+\e_0\bigl\|\partial_{X_\la}f\bigr\|_{ L^{\f{2p}{2+p}}}\Bigr)^{\f2p},
\end{split}\eeq for $C(s,p,X)$ given by \eqref{S6eq9a}.}
\end{prop}

\begin{proof} The proof of this proposition consists in the estimate of the striated regularity of the solution \eqref{Saeq3} and then applying
 Proposition \ref{S9prop1}. Let us denote
\beno
\frak{G}\eqdefa \left((a_{ij})_{2\times2}-{\rm Id}\right)\na u,
\eeno For simplicity, we just present the {\it a priori} estimate for smooth enough solutions of \eqref{Saeq3}.
We first write
\beq \label{Saeq5} \nabla u=\sum_{i,j=1}^2\nabla(-\Delta)^{-1}\partial_i((a_{ij}-\delta_{ij})\p_j u)+\nabla\Delta^{-1}f,\eeq
from which, we infer
\beno
\begin{split}
\|\nabla u\|_{L^p}\leq & C\bigl(\bigl\|\frak{G}\|_{L^p}+\|f\|_{L^\f{2p}{2+p}}\bigr)
\leq C\bigl(\e_0\|\na u\|_{L^p}+\|f\|_{L^\f{2p}{2+p}}\bigr).
\end{split}
\eeno
So that by taking $\e_0$ sufficiently small, we obtain
\beq\label{Saeq2}
 \|\nabla u\|_{L^p}\leq C\|f\|_{L^{\frac{2p}{2+p}}}.
 \eeq
Whereas for any $C^1$ vector field $X,$ we  get, by applying  $\partial_X$ to \eqref{Saeq3}, that
$$\sum_{i,j=1}^2\partial_i(a_{ij}\partial_j \partial_X u)=\partial_X f-\sum_{i,j=1}^2\partial_i(\partial_Xa_{ij}\partial_j u)-\sum_{i,j=1}^2\partial_i(a_{ij}\partial_j X\partial_j u).$$
Then along the same line to proof of \eqref{Saeq2}, we deduce
\beq \label{Saeq4}
\begin{split}
\|\nabla\partial_X u\|_{L^p}\leq & C\Bigl(\|\partial_Xf\|_{L^{\frac{2p}{2+p}}}+\bigl(\|\p_Xa_{ij}\|_{L^\infty}+\|\na X\|_{L^\infty}\bigr)\|\na u\|_{L^p}\Bigr)\\
\leq &C\Bigl(\|\partial_Xf\|_{L^{\frac{2p}{2+p}}}+\bigl(\|\p_Xa_{ij}\|_{L^\infty}+\|\na X\|_{L^\infty}\bigr)\|f\|_{L^{\frac{2p}{2+p}}}\Bigr).
\end{split}
\eeq
On the other hand,
for any $s\in ]2/p,1[,$ we deduce from Proposition \ref{S9prop1} and \eqref{Saeq5}
that
\beno
\begin{split}
\|&\na u\|_{L^\infty}\leq C\bigl(\bigl\|\frak{G}\bigr\|_{L^p}+\|\frak{G}\|_{L^\infty}\bigr)+\|\na\D^{-1}f\|_{L^\infty}+\f{C_s}{I(X)}\sup_{\la\in\La}\Bigl\{\bigl(\|X_\la\|_{L^\infty}\|\frak{G}\|_{L^p}\bigr)^{1-\f2p}\\
&\times\bigl(\|\na X_\la\|_{L^p}\|\frak{G}\|_{L^\infty}+\|\p_{X_\la}\frak{G}\|_{L^p}\bigr)^{\f2p}+\bigl(\|X_\la\|_{L^\infty}\|\frak{G}\|_{L^p}\bigr)^{1-\f2{ps}}\bigl(\|\na X_\la\|_{L^p}\|\frak{G}\|_{L^\infty}\bigr)^{\f2{ps}}\Bigr\}.
\end{split}
\eeno
Due to $p\in ]2,\infty[,$ we have
\beno
\|\na\D^{-1}f\|_{L^\infty}\leq C\|f\|_{L^p}^{\f2p}\|f\|_{L^{\f{2p}{2+p}}}^{1-\f2p}.
\eeno
Applying Young's inequality yields
\beno
\begin{split}
\f{C}{I(X)}&\bigl(\e_0\|X_\la\|_{L^\infty}\|\na u\|_{L^p}\bigr)^{1-\f2p}\bigl(\e_0\|\na X_\la\|_{L^p}\|\na u\bigr\|_{L^\infty}\bigr)^{\f2p}\\
&\leq C\e_0\Bigl(\|\na u\|_{L^\infty}+I(X)^{-\f{p}{p-2}}\|X_\la\|_{L^\infty}\|\na X_\la\|_{L^p}^{\f2{p-2}}\|\na u\|_{L^p}\Bigr).
\end{split}
\eeno
As a result, it comes out
\beno
\begin{split}
\|\na u\|_{L^\infty}\leq & C\|f\|_{L^p}^{\f2p}\|f\|_{L^{\f{2p}{2+p}}}^{1-\f2p}+ C\e_0\Bigl(\|\na u\|_{L^p}+\|\na u\|_{L^\infty}\\
&+C(s,p,X)\|X_\la\|_{L^\infty}\bigl(\|\na X_\la\|_{L^p}^{\f2{p-2}}+
\|\na X_\la\|_{L^p}^{\f2{ps-2}}\bigr)\|\na u\|_{L^p}\Bigr)\\
& +\f{C}{I(X)}\sup_{\la\in\La}\bigl(\e_0\|X_\la\|_{L^\infty}\|\na u\|_{L^p}\bigr)^{1-\f2p}\bigl(\|\na u\|_{L^p}+\e_0\bigl\|\na\p_{X_\la} u\|_{L^p}\bigr)^{\f2p}.
\end{split}
\eeno
By taking $\e_0$ sufficiently small and inserting the Estimates \eqref{Saeq2} and \eqref{Saeq4} to the above inequality, we achieve \eqref{Saeq1}.
 \end{proof}

 \begin{rmk}
By repeating  the proof of Proposition \ref{Sapp}, we can prove the same Lipschitz estimate \eqref{Saeq1} for the solutions
of the following Stokes type system:
$$\sum_{i,j=1}^2\partial_i\left(a_{ij}\partial_j u\right)=f-\nabla p\andf  \dive u=0 \quad\mbox{for}\quad x\in\R^2.$$
We can even work for the above problems in the  multi-dimensional case.
\end{rmk}

\bigbreak \noindent {\bf Acknowledgments.}
P. Zhang is partially supported
by NSF of China under Grants  11731007 and 11688101, Morningside Center of Mathematics of The Chinese Academy of Sciences and innovation grant from National Center for
Mathematics and Interdisciplinary Sciences.

\medskip


{\small

\end{document}